%% file: main.tex
\documentclass{interact}
\title{Generalized Left-Localized Cayley Parametrization\\ for Optimization with Orthogonality Constraints}
\author{
  \name{Keita Kume\textsuperscript{a}\thanks{CONTACT Keita Kume. Email: kume@sp.ce.titech.ac.jp} and Isao Yamada\textsuperscript{a}\thanks{CONTACT Isao Yamada. Email: isao@sp.ce.titech.ac.jp}\thanks{This is an original manuscript of an article published by Taylor \& Francis in {\it Optimization} on 15 Nov. 2022, available online: https://doi.org/10.1080/02331934.2022.2142471.}
}
\affil{\textsuperscript{a}Dept. of Information and Communications Engineering, Tokyo Institute of Technology, Tokyo, Japan}
}

\usepackage[caption=false]{subfig}
\usepackage[numbers,sort&compress]{natbib}
\input{bibpunct.tex}
\makeatletter
\def\NAT@def@citea{\def\@citea{\NAT@separator}}
\makeatother

\usepackage{amsmath,amssymb,amsfonts}
\usepackage{graphicx}
\usepackage{textcomp}
\usepackage{xcolor}

\usepackage{here}
\usepackage{amssymb}
\usepackage{ascmac}
\usepackage{bm}
\usepackage{algorithm}
\usepackage{algorithmic}
\usepackage{comment}
\usepackage{cases}
\usepackage{mathtools}
\usepackage{enumitem}
\usepackage{mathabx}
\usepackage{hyperref}
\usepackage{stmaryrd}
\usepackage{tikz}
\usetikzlibrary{arrows.meta}
\usetikzlibrary{positioning}
\usepackage{csvsimple}

\mathtoolsset{showonlyrefs=true}

\newcommand{\argmin}{\mathop{\mathrm{argmin}}\limits}

\newcommand{\inprod}[2]{{\langle #1,#2 \rangle}}
\newcommand{\trace}{{\rm Tr}}

\newcommand{\St}{{\rm St}}
\newcommand{\Gr}{{\rm Gr}}

\newcommand{\T}{\mathsf{T}}
\newcommand{\diag}{\mathrm{diag}}
\newcommand{\Skew}{\mathop{\mathrm{S_{kew}}}}
\newcommand{\dbra}[1]{\llbracket #1 \rrbracket}

\makeatletter
\newcommand{\doublewidetilde}[1]{{%
  \mathpalette\double@widetilde{#1}%
}}
\newcommand{\double@widetilde}[2]{%
  \sbox\z@{$\m@th#1\widetilde{#2}$}%
  \ht\z@=.9\ht\z@
  \widetilde{\box\z@}%
}
\makeatother

\theoremstyle{plain}
\newtheorem{theorem}{Theorem}[section]
\newtheorem{lemma}[theorem]{Lemma}

\newtheorem{proposition}[theorem]{Proposition}

\theoremstyle{definition}
\newtheorem{definition}[theorem]{Definition}

\newtheorem{problem}[theorem]{Problem}

\theoremstyle{remark}
\newtheorem{remark}[theorem]{Remark}

\newtheorem{fact}[theorem]{Fact}

\begin{document}
\articletype{ARTICLE TEMPLATE}

\maketitle

\begin{abstract}
We present a reformulation of optimization problems over the Stiefel manifold by using a Cayley-type transform, named the generalized left-localized Cayley transform, for the Stiefel manifold.
The reformulated optimization problem is defined over a vector space, whereby we can apply directly powerful computational arts designed for optimization over a vector space.
The proposed Cayley-type transform enjoys several key properties which are useful to
(i) study relations between the original problem and the proposed problem;
(ii) check the conditions to guarantee the global convergence of optimization algorithms.
Numerical experiments demonstrate that the proposed algorithm outperforms the standard algorithms designed with a retraction on the Stiefel manifold.
\end{abstract}


\noeqref{eq:prop:gradient_Lipschitz, eq:Cayley, eq:Cayley_inv_origin, eq:grad_propo, eq:optimality_CP_lemma, eq:gradient_translation}
\section{Introduction}\label{sec:introduction}
The Stiefel manifold
$\St(p,N):= \{\bm{U} \in \mathbb{R}^{N\times p} \mid \bm{U}^{\T}\bm{U} = \bm{I}_{p}\}$
is defined for
$(p,N) \in \mathbb{N} \times \mathbb{N}$
with
$p \leq N$,
where
$\bm{I}_{p}$
is the
$p\times p$
identity matrix
(see Appendix~\ref{appendix:facts} for basic facts on
$\St(p,N)$).

We consider an orthogonal constraint optimization problem formulated as:
\begin{problem}\label{problem:origin}
  For a given continuous function
  $f:\mathbb{R}^{N\times p }\to \mathbb{R}$,
  \begin{equation}
    \textrm{find} \  \bm{U}^{\star} \in \argmin_{\bm{U}\in \St(p,N)}f(\bm{U}), \label{eq:cost}
  \end{equation}
  where the existence of a minimizer in~\eqref{eq:cost} is automatically guaranteed by the compactness of
  $\St(p,N)$
  and the continuity of
  $f$
  over the $Np$-dimensional Euclidean space
  $\mathbb{R}^{N\times p}$.
\end{problem}
This problem belongs to the so-called Riemannian optimization problems (see~\cite{manifold_book} and Appendix~\ref{appendix:retraction}), and has rich applications, in the case
$p \ll N$
in particular, in data sciences including signal processing and machine learning as remarked recently in~\cite{M20} and~\cite{Sato21}.
These applications include, e.g., nearest low-rank correlation matrix problem~\cite{P04,G07,Z15B}, nonlinear eigenvalue problem~\cite{L98, Y06,Z15}, sparse principal component analysis~\cite{Z06,J10,L12}, 1-bit compressed sensing~\cite{B08,L11}, joint diagonalization problem for independent component analysis~\cite{J02,T09,S17} and enhancement of the generalization performance in deep neural network~\cite{H18,N18B}.
However, Problem~\ref{problem:origin} has inherent difficulties regarding the severe nonlinearity of
$\St(p,N)$
as an instance of general nonlinear Riemannian manifolds.

Minimization of a continuous
$f:\mathbb{R}^{N\times N}\to \mathbb{R}$
over the orthogonal group
${\rm O}(N) := \St(N,N)$
is a special instance of Problem~\ref{problem:origin} with
$p = N$.
This problem can be separated into two optimization problems over the special orthogonal group
${\rm SO}(N):=\{\bm{U}\in {\rm O}(N):=\St(N,N) \mid \det(\bm{U})=1\}$
as
\begin{equation}
  \textrm{find}\    \bm{U}_{1}^{\star} \in \argmin_{\bm{U}\in {\rm SO}(N)}f(\bm{U}) \label{eq:cost_SO}
\end{equation}
and, with an arbitrarily chosen
$\bm{Q} \in {\rm O}(N) \setminus {\rm SO}(N)$,
\begin{equation}
  \textrm{find}\    \bm{U}_{2}^{\star} \in \argmin_{\bm{U}\in {\rm SO}(N)}f(\bm{Q}\bm{U})
\end{equation}
because
${\rm O}(N)$
is the disjoint union of
${\rm SO}(N)$
and
${\rm O}(N) \setminus {\rm SO}(N)=\{\bm{U} \in {\rm O}(N) \mid \det(\bm{U}) = -1\} = \{\bm{Q}\bm{U} \in {\rm O}(N)\mid \bm{U} \in {\rm SO}(N)\}$.
For the problem in~\eqref{eq:cost_SO}, the Cayley transform
\begin{equation}
  \varphi:{\rm SO}(N)\setminus E_{N,N} \to Q_{N,N}:\bm{U}\mapsto (\bm{I}-\bm{U})(\bm{I}+\bm{U})^{-1} \label{eq:origin_Cayley}
\end{equation}
and its inversion mapping\footnote{
  $\varphi^{-1}$ is well-defined over
  $Q_{N,N}$
  because all eigenvalues of
  $\bm{V}\in Q_{N,N}$
  are pure imaginary.
  For the second expression in~\eqref{eq:inv_origin_Cayley}, see the beginning of Appendix~\ref{appendix:inverse}.
}
\begin{equation}
  \varphi^{-1}: Q_{N,N}\to {\rm SO}(N)\setminus E_{N,N}:\bm{V}\mapsto (\bm{I}-\bm{V})(\bm{I}+\bm{V})^{-1}= 2(\bm{I}+\bm{V})^{-1} -\bm{I} \label{eq:inv_origin_Cayley}
\end{equation}
have been utilized in~\cite{Y03,H10,H18} because
$\varphi$
translates a subset
${\rm SO}(N) \setminus E_{N,N}(={\rm O}(N)\setminus E_{N,N} {\rm [see~\eqref{eq:orthogonal_normal}]})$
of
${\rm SO}(N)$
into the vector space
$Q_{N,N}:=\{\bm{V} \in \mathbb{R}^{N\times N} \mid \bm{V}^{\T} = -\bm{V} \}$
of all skew-symmetric matrices, where
$E_{N,N}:=\{\bm{U} \in {\rm O}(N) \mid \det(\bm{I}+\bm{U}) =0 \}$
is called, in this paper, the singular-point set of
$\varphi$.
More precisely, this is because
$\varphi$
is a diffeomorphism between the dense subset\footnote{
  The closure of
  ${\rm SO}(N) \setminus E_{N,N}$
  is equal to
  ${\rm SO}(N)$.
  For every
  $\bm{U} \in {\rm SO}(N)$,
  we can approximate it by some sequence
  $(\bm{U}_{n})_{n=1}^{\infty}$
  of
  ${\rm SO}(N)\setminus E_{N,N}$
  with any accuracy, i.e.,
  $\lim_{n\to \infty} \bm{U}_{n} = \bm{U}$.
}
${\rm SO}(N)\setminus E_{N,N}$
of
${\rm SO}(N)$
and
$Q_{N,N}$.

The Cayley transform pair
$\varphi$
and
$\varphi^{-1}$
can be modified with an arbitrarily chosen
$\bm{S} \in {\rm O}(N)$
as
\begin{equation} \label{eq:LCT_O}
  \varphi_{\bm{S}}:{\rm O}(N)\setminus E_{N,N}(\bm{S})\to Q_{N,N}(\bm{S}):=Q_{N,N}:\bm{U}\mapsto \varphi(\bm{S}^{\T}\bm{U}) = (\bm{I}-\bm{S}^{\T}\bm{U})(\bm{I}+\bm{S}^{\T}\bm{U})^{-1}
\end{equation}
and
\begin{equation}\label{eq:ILCT_O}
  \varphi_{\bm{S}}^{-1}:Q_{N,N}(\bm{S})\to{\rm O}(N)\setminus E_{N,N}(\bm{S}):\bm{V}\mapsto \bm{S}\varphi^{-1}(\bm{V})=\bm{S}(\bm{I}-\bm{V})(\bm{I}+\bm{V})^{-1},
\end{equation}
where
$E_{N,N}(\bm{S}):=\{\bm{U}\in{\rm O}(N) \mid \det(\bm{I}+\bm{S}^{T}\bm{U}) = 0\}$
is the singular-point set of
$\varphi_{\bm{S}}$.
These mappings are also diffeomorphisms between their domains and images.
With the aid of
$\varphi_{\bm{S}}$\footnote{
  The domain of
  $\varphi_{\bm{S}}$
  with
  $\bm{S} \in {\rm SO}(N)$
  is a subset
  ${\rm O}(N)\setminus E_{N,N}(\bm{S}) = {\rm SO}(N)\setminus E_{N,N}(\bm{S})$
  of
  ${\rm SO}(N)$.
}
with
$\bm{S} \in {\rm SO}(N)$,
the following Problem~\ref{problem:CP_O} was considered in~\cite{Y03} as a relaxation of the problem in~\eqref{eq:cost_SO}.

\begin{problem}\label{problem:CP_O}
  For a given continuous function
  $f:\mathbb{R}^{N\times N }\to \mathbb{R}$,
  choose
  $\bm{S} \in {\rm SO}(N)$,
  and
  $\epsilon > 0$
  arbitrarily.
  Then,
  \begin{equation}
    \textrm{find} \ \bm{V}^{\star} \in  Q_{N,N}(\bm{S}) \ \textrm{such that} \ 
    f\circ\varphi_{\bm{S}}^{-1}(\bm{V}^{\star}) < \min f({\rm SO}(N)) + \epsilon.\label{eq:CP_O}
  \end{equation}
\end{problem}
\begin{remark}\label{remark:problem_CP_O}
  \begin{enumerate}[label=(\alph*)]
    \item (The existence of $\bm{V}^{\star}$ in Problem~\ref{problem:CP_O}). \label{enum:existence_CP_O}
      The existence of
      $\bm{V}^{\star}$
      satisfying~\eqref{eq:CP_O} is guaranteed because
      $\varphi_{\bm{S}}^{-1}(Q_{N,N})={\rm SO}(N)\setminus E_{N,N}(\bm{S})$
      is a dense subset of
      ${\rm SO}(N)$
      for any
      $\bm{S} \in {\rm SO}(N)$~\cite{Y03} (see Fact~\ref{fact:dense}) and
      $f\circ \varphi_{\bm{S}}^{-1}$
      is continuous.
    \item (Left-localized Cayley transform). \label{enum:LCT}
      We call
      $\varphi_{\bm{S}}$
      in~\eqref{eq:LCT_O}
      \textit{the left-localized Cayley transform centered at $\bm{S}\in {\rm O}(N)$}
      because
      $\bm{S}$
      is multiplied from the left of
      $\varphi^{-1}(\bm{V})$
      in~\eqref{eq:ILCT_O}, and
      $\varphi_{\bm{S}}(\bm{S})= \bm{0}$.
      Although
      $Q_{N,N}(\bm{S})$
      in~\eqref{eq:LCT_O} is the common set
      $Q_{N,N}$
      for all
      $\bm{S}\in {\rm O}(N)$,
      we distinguish
      $Q_{N,N}(\bm{S})$
      for each
      $\bm{S} \in {\rm O}(N)$
      as the domain of parametrization
      $\varphi_{\bm{S}}^{-1}$
      for a particular subset
      ${\rm O}(N)\setminus E_{N,N}(\bm{S})\subset {\rm O}(N)$.
  \end{enumerate}
\end{remark}

We note that Problem~\ref{problem:CP_O} is a realistic relaxation of the problem in~\eqref{eq:cost_SO} as long as our target is approximation of a solution to~\eqref{eq:cost_SO} algorithmically because
${\rm SO}(N)\setminus E_{N,N}(\bm{S})=\varphi_{\bm{S}}^{-1}(Q_{N,N}(\bm{S}))$
is dense in
${\rm SO}(N)$.
In reality with a digital computer, we can handle just a small subset of the rational numbers
$\mathbb{Q}$,
which is dense in
$\mathbb{R}$,
due to the limitation of the numerical precision.
This situation implies that it is reasonable to consider an approximation of
${\rm SO}(N)$
within its dense subset
${\rm SO}(N) \setminus E_{N,N}(\bm{S})$.

For Problem~\ref{problem:CP_O}, we can enjoy various arts of optimization over a vector space, e.g., the gradient descent method and Newton's method, because
$Q_{N,N}(\bm{S})$
is a vector space.
Thanks to the homeomorphism of
$\varphi_{\bm{S}}$,
we can estimate a solution to the problem in~\eqref{eq:cost_SO} by applying
$\varphi_{\bm{S}}^{-1}$
to a solution of Problem~\ref{problem:CP_O} with a sufficiently small
$\epsilon > 0$.
We call this strategy via Problem~\ref{problem:CP_O} {\it a Cayley parametrization (CP) strategy} for the problem in~\eqref{eq:cost_SO}.
The CP strategy has a notable advantage over the standard optimization strategies~\cite{manifold_book}, called {\it the retraction-based strategies}, in view that many powerful computational arts designed for optimization over a single vector space can be directly plugged into the CP strategy.
We will discuss the details in Remark~\ref{remark:relation_to_others}.

In this paper, we address a natural question regarding a possible extension of the CP strategy to Problem~\ref{problem:origin} for general $p<N$: can we parameterize a dense subset of
$\St(p,N)$
even with
$p<N$
in terms of a single vector space?
To answer this question positively, we propose \textit{a Generalized Left-Localized Cayley Transform (G-L$^{2}$CT)}:
\begin{align}
  \Phi_{\bm{S}}:\St(p,N) \setminus E_{N,p}(\bm{S})\to Q_{N,p}(\bm{S}):\bm{U}\mapsto
  \begin{bmatrix}
    \bm{A}_{\bm{S}}(\bm{U}) & -\bm{B}^{\T}_{\bm{S}}(\bm{U}) \\
    \bm{B}_{\bm{S}}(\bm{U}) & \bm{0}
  \end{bmatrix},
\end{align}
with
$\bm{S} \in {\rm O}(N)$,
as an extension of the left-localized Cayley transform
$\varphi_{\bm{S}}$
in~\eqref{eq:LCT_O},
where
$\bm{A}_{\bm{S}}(\bm{U}) \in Q_{p,p}$
and
$\bm{B}_{\bm{S}}(\bm{U}) \in \mathbb{R}^{(N-p)\times p}$
are determined with a center point
$\bm{S}\in {\rm O}(N)$
(see~\eqref{eq:Cay_A} and~\eqref{eq:Cay_B} in Definition~\ref{definition:Cayley}).
The set
$E_{N,p}(\bm{S}):= \{\bm{U} \in \St(p,N) \mid \det(\bm{I}_{p} + \bm{S}_{\rm le}^{\T}\bm{U}) = 0\}$
is called the singular-point set of
$\Phi_{\bm{S}}$
(see the notation in the end of this section), and
$Q_{N,p}(\bm{S})$
is a linear subspace of
$Q_{N,N}(\bm{S})$
(see~\eqref{eq:skew}).
For any
$\bm{S} \in {\rm O}(N)$,
we will show several key properties, e.g.,
(i)
$\Phi_{\bm{S}}$
is diffeomorphism between
$\St(p,N)\setminus E_{N,p}(\bm{S})$
and the vector space
$Q_{N,p}(\bm{S})$
with  the inversion mapping
$\Phi_{\bm{S}}^{-1}:Q_{N,p}(\bm{S})\to \St(p,N)\setminus E_{N,p}(\bm{S})$
(see Proposition~\ref{proposition:inverse});
(ii)
$\St(p,N)\setminus E_{N,p}(\bm{S})$
is a dense subset of
$\St(p,N)$
for
$p < N$
(see Theorem~\ref{theorem:dense}~\ref{enum:dense}).
Therefore, the proposed
$\Phi_{\bm{S}}$
and
$\Phi_{\bm{S}}^{-1}$
have inherent properties desired for applications in the CP strategy to Problem~\ref{problem:origin}.

To extend the CP strategy to Problem~\ref{problem:origin} for
$p < N$,
we consider Problem~\ref{problem:CP_St} below, which can be seen as an extension of Problem~\ref{problem:CP_O}.
For the same reason as in Remark~\ref{remark:problem_CP_O}~\ref{enum:existence_CP_O}, the existence of
$\bm{V}^{\star}$
achieving~\eqref{eq:simple_alcp} is guaranteed by the denseness of
$\St(p,N) \setminus E_{N,p}(\bm{S})$
in
$\St(p,N)$
(see Lemma~\ref{lemma:gap}).
\begin{problem} \label{problem:CP_St}
  For a given continuous function
  $f:\mathbb{R}^{N\times p}\to \mathbb{R}$
  with
  $p < N$,
  choose
  $\bm{S} \in {\rm O}(N)$,
  and
  $\epsilon > 0$
  arbitrarily.
  Then,
  \begin{equation}
    \textrm{find} \ \bm{V}^{\star} \in  Q_{N,p}(\bm{S}) \ \textrm{such that} \  f\circ\Phi_{\bm{S}}^{-1}(\bm{V}^{\star}) < \min f(\St(p,N)) + \epsilon. \label{eq:simple_alcp}
  \end{equation}
\end{problem}

Under a smoothness assumption on general
$f$,
a realistic goal for Problem~\ref{problem:origin} is to find a stationary point
$\bm{U}^{\star} \in \St(p,N)$
of
$f$
because Problem~\ref{problem:origin} is a non-convex optimization problem (see, e.g.,~\cite{manifold_book,W13,BXX18}) and any local minimizer must be a stationary point~\cite{W13,BXX18}.
In Lemma~\ref{lemma:optimality}, we present a characterization of a stationary point
$\bm{U}^{\star} \in \St(p,N)$
of
$f$
over
$\St(p,N)$,
with
$\bm{S} \in {\rm O}(N)$
satisfying
$\bm{U}^{\star} \in \St(p,N)\setminus E_{N,p}(\bm{S})$,
in terms of a stationary point
$\bm{V}^{\star} \in Q_{N,p}(\bm{S})$
of
$f\circ \Phi_{\bm{S}}^{-1}$
over the vector space
$Q_{N,p}(\bm{S})$,
i.e.,
$\nabla (f\circ\Phi_{\bm{S}}^{-1})(\bm{V}^{\star})=\bm{0}$.
To approximate a stationary point of
$f$
over
$\St(p,N)$,
we also consider the following problem:
\begin{problem}\label{problem:CP_grad}
  For a continuously differentiable function
  $f:\mathbb{R}^{N\times p}\to \mathbb{R}$
  with
  $p<N$,
  choose
  $\bm{S} \in {\rm O}(N)$
  and
  $\epsilon > 0$
  arbitrarily.
  Then,
  \begin{equation} \label{eq:problem_alcp_grad}
    \textrm{find}\ \bm{V}^{\star}\in Q_{N,p}(\bm{S}) \  \textrm{such that}\ \|\nabla (f\circ \Phi_{\bm{S}}^{-1})(\bm{V}^{\star})\|_{F} < \epsilon.
  \end{equation}
\end{problem}
\noindent
For Problem~\ref{problem:CP_grad}, we can apply many powerful arts for searching a stationary point of a non-convex function over a vector space.

Numerical experiments in Section~\ref{sec:numerical} demonstrate that the proposed CP strategy outperforms the standard algorithms designed with a retraction on
$\St(p,N)$
(see Appendix~\ref{appendix:retraction})
in the scenario of a certain eigenbasis extraction problem.

\textbf{Notation}
$\mathbb{N}$
and
$\mathbb{R}$
denote the set of all positive integers and the set of all real numbers respectively.
For general
$n\in \mathbb{N}$,
we use
$\bm{I}_{n}$
for the identity matrix in
$\mathbb{R}^{n\times n}$,
but for simplicity, we use
$\bm{I}$
for the identity matrix in
$\mathbb{R}^{N\times N}$.
For
$p \leq N$,
$\bm{I}_{N\times p} \in \mathbb{R}^{N\times p}$
denotes the matrix of the first
$p$
columns of
$\bm{I}$.
For a matrix
$\bm{X} \in \mathbb{R}^{n\times m}$,
$[\bm{X}]_{ij}$
$(1\leq i \leq n, 1 \leq j\leq m)$
denotes the
$(i,j)$
entry of
$\bm{X}$,
and
$\bm{X}^{\T}$
denotes the transpose of
$\bm{X}$.
For a square matrix
$\bm{X}:= \begin{bmatrix} \bm{X}_{11} \in \mathbb{R}^{p\times p} & \bm{X}_{12} \in \mathbb{R}^{p\times (N-p)} \\ \bm{X}_{21} \in \mathbb{R}^{(N-p)\times p} & \bm{X}_{22} \in \mathbb{R}^{(N-p)\times (N-p)} \end{bmatrix}\in \mathbb{R}^{N\times N}$,
we use the notation
$\dbra{\bm{X}}_{ij}:= \bm{X}_{ij}$
for
$i,j \in \{1,2\}$.
For
$\bm{U} \in \mathbb{R}^{N\times p}$,
the matrices
$\bm{U}_{\rm up} \in \mathbb{R}^{p \times p}$
and
$\bm{U}_{\rm lo} \in \mathbb{R}^{(N-p) \times p}$
respectively denote the upper and the lower block matrices of
$\bm{U}=[\bm{U}_{\rm up}^{\T}\ \bm{U}_{\rm lo}^{\T}]^{\T}$.
For
$\bm{S} \in \mathbb{R}^{N\times N}$,
the matrices
$\bm{S}_{\rm le} \in \mathbb{R}^{N\times p}$
and
$\bm{S}_{\rm ri} \in \mathbb{R}^{N\times (N-p)}$
respectively denote the left and right block matrices of
$\bm{S}=[\bm{S}_{\rm le}\ \bm{S}_{\rm ri}]$.
For a matrix
$\bm{X}\in\mathbb{R}^{n\times n}$,
$\Skew(\bm{X}) = (\bm{X}-\bm{X}^{\T})/2$
denotes the skew-symmetric component of
$\bm{X}$.
For square matrices
$\bm{X}_{i}\in \mathbb{R}^{n_{i}\times n_{i}}\ (1\leq i \leq k)$,
$\diag(\bm{X}_{1},\bm{X}_{2},\ldots,\bm{X}_{k})\in \mathbb{R}^{(\sum_{i=1}^{k}n_{i})\times  (\sum_{i=1}^{k}n_{i})}$
denotes the block diagonal matrix with diagonal blocks
$\bm{X}_{1},\bm{X}_{2},\ldots,\bm{X}_{k}$.
For a given matrix,
$\|\cdot\|_{2}$
and
$\|\cdot\|_{F}$
denote the spectral norm and the Frobenius norm respectively.
The functions
$\sigma_{\max}(\cdot)$
and
$\sigma_{\min}(\cdot)$
denote respectively the largest and the nonnegative smallest singular values of a given matrix.
The function
$\lambda_{\max}(\cdot)$
denotes the largest eigenvalue of a given symmetric matrix.
For a vector space
$\mathcal{X}$
of matrices,
$B_{\mathcal{X}}(\bm{X}^{\star},\epsilon):= \{\bm{X} \in \mathcal{X} \mid \|\bm{X}-\bm{X}^{\star}\|_{F} < \epsilon \}$
denotes an open ball centered at
$\bm{X}^{\star} \in \mathcal{X}$
with radius
$\epsilon > 0$.
To distinguish from the symbol for the orthogonal group
${\rm O}(N)$,
the symbol
$\mathfrak{o}(\cdot)$
is used in place of the standard big O notation for computational complexity.

\section{Generalized left-localized Cayley Transform (G-L$^{2}$CT)}\label{sec:cayley_stiefel}
\subsection{Definition and Properties of G-L$^{2}$CT}\label{sec:localized}
In this subsection, we introduce the Generalized Left-Localized Cayley Transform (G-L${}^2$CT) for the parametrization of
$\St(p,N)$
as a natural extension of
$\varphi_{\bm{S}}$
in~\eqref{eq:LCT_O}.
Indeed, the G-L${}^{2}$CT inherits key properties satisfied by
$\varphi_{\bm{S}}$
(see Proposition~\ref{proposition:inverse} and Theorem~\ref{theorem:dense}).
\begin{definition}[Generalized left-localized Cayley transform]\label{definition:Cayley}
  For
  $p,N\in \mathbb{N}$
  satisfying
  $p\leq N$,
  let
  $\bm{S} \in {\rm O}(N)$,
  $E_{N,p}(\bm{S}):= \{\bm{U} \in \St(p,N) \mid \det(\bm{I}_{p} + \bm{S}_{\rm le}^{\T}\bm{U}) = 0\}$,
  and
  \begin{equation} \label{eq:skew}
    Q_{N,p}(\bm{S}) := Q_{N,p}:=\left\{
    \left.\begin{bmatrix} \bm{A} & -\bm{B}^{\T} \\ \bm{B} & \bm{0} \end{bmatrix}
      \;\right|\;\substack{ -\bm{A}^{\T} = \bm{A} \in \mathbb{R}^{p\times p},\\ \bm{B} \in \mathbb{R}^{(N-p)\times p}}\right\} \subset Q_{N,N}.
  \end{equation}
  The generalized left-localized Cayley transform  centered at
  $\bm{S}$
  is defined by
  \begin{align}
    \Phi_{\bm{S}}:\St(p,N) \setminus E_{N,p}(\bm{S})\to Q_{N,p}(\bm{S}):\bm{U}\mapsto
    \begin{bmatrix}
      \bm{A}_{\bm{S}}(\bm{U}) & -\bm{B}^{\T}_{\bm{S}}(\bm{U}) \\
      \bm{B}_{\bm{S}}(\bm{U}) & \bm{0}
    \end{bmatrix} 
    \label{eq:Cayley}
  \end{align}
  with
  \begin{align}
    \bm{A}_{\bm{S}}(\bm{U}) &:= 2(\bm{I}_{p}+\bm{S}_{\rm le}^{\T}\bm{U})^{-\mathrm{T}}\Skew(\bm{U}^{\T}\bm{S}_{\rm le})(\bm{I}_{p}+\bm{S}_{\rm le}^{\T}\bm{U})^{-1} \in Q_{p,p} \label{eq:Cay_A}\\
    \bm{B}_{\bm{S}}(\bm{U}) &:= - \bm{S}_{\rm ri}^{\T}\bm{U}(\bm{I}_{p}+\bm{S}_{\rm le}^{\T}\bm{U})^{-1} \in \mathbb{R}^{(N-p) \times p}, \label{eq:Cay_B}
  \end{align}
  where we call
  $\bm{S}$
  the center point of
  $\Phi_{\bm{S}}$,
  and
  $E_{N,p}(\bm{S})$
  the singular-point set of
  $\Phi_{\bm{S}}$.
\end{definition}
\begin{proposition}[Inversion of G-L${}^2$CT]\label{proposition:inverse}
  The mapping
  $\Phi_{\bm{S}}$
  with
  $\bm{S}\in {\rm O}(N)$
  is a diffeomorphism between a subset
  $\St(p,N)\setminus E_{N,p}(\bm{S})\subset \St(p,N)$
  and
  $Q_{N,p}(\bm{S})$\footnote{
    As in~\eqref{eq:skew},
    $Q_{N,p}(\bm{S})$
    is the common set
    $Q_{N,p}$
    for every
    $\bm{S}\in{\rm O}(N)$.
    However, we distinguish
    $Q_{N,p}(\bm{S})$
    for each
    $\bm{S} \in {\rm O}(N)$
    as a parametrization of the particular subset
    $\St(p,N)\setminus E_{N,p}(\bm{S})$
    of
    $\St(p,N)$
    (see also Remark~\ref{remark:problem_CP_O}~\ref{enum:LCT}).
  }.
  The inversion mapping is given, in terms of
  $\varphi_{\bm{S}}^{-1}$
  in~\eqref{eq:ILCT_O},
  by
  \begin{align}
    \Phi_{\bm{S}}^{-1}:Q_{N,p}(\bm{S})\to \St(p,N)\setminus E_{N,p}(\bm{S}):\bm{V} \mapsto
    & \Xi\circ\varphi_{\bm{S}}^{-1}(\bm{V})=\bm{S}(\bm{I}-\bm{V})(\bm{I}+\bm{V})^{-1}\bm{I}_{N\times p},\\ \label{eq:Cayley_inv_origin}
  \end{align}
  where
  $\Xi:{\rm O}(N)\to \St(p,N):\bm{U}\mapsto \bm{U}\bm{I}_{N\times p}$.
  Moreover, for
  $\bm{V} \in Q_{N,p}(\bm{S})$,
  we have the following expressions
  \begin{align}
     \quad \Phi_{\bm{S}}^{-1}(\bm{V})
    & = \Xi\circ\varphi_{\bm{S}}^{-1}(\bm{V})
    = 2\bm{S}(\bm{I}+\bm{V})^{-1}\bm{I}_{N\times p} - \bm{S}\bm{I}_{N\times p} \label{eq:Cayley_inv_alt}\\
    & = 2\bm{S}\begin{bmatrix}\bm{M}^{-1} \\ -\dbra{\bm{V}}_{21}\bm{M}^{-1} \end{bmatrix} - \bm{S}_{\rm le}
    = 2(\bm{S}_{\rm le} - \bm{S}_{\rm ri}\dbra{\bm{V}}_{21})\bm{M}^{-1}-\bm{S}_{\rm le}, \label{eq:Cayley_inv}
  \end{align}
  where
  $\bm{M}:= \bm{I}_{p}+\dbra{\bm{V}}_{11}+\dbra{\bm{V}}_{21}^{\T}\dbra{\bm{V}}_{21}\in \mathbb{R}^{p\times p}$
  is the Schur complement matrix of
  $\bm{I}+\bm{V}\in \mathbb{R}^{N\times N}$ (see Fact~\ref{fact:Schur}).
\end{proposition}
\begin{proof}
  See Appendix~\ref{appendix:inverse}.
\end{proof}

\begin{theorem}[Denseness of $\St(p,N)\setminus E_{N,p}(\bm{S})$]\label{theorem:dense}
  Let
  $\bm{S} \in {\rm O}(N)$
  and
  $p < N$.
  Then, the following hold:
  \begin{enumerate}[label=(\alph*)]
    \item
      \label{enum:St_Xi_O}
      $\St(p,N)\setminus E_{N,p}(\bm{S}) = \Xi({\rm O}(N)\setminus E_{N,N}(\bm{S}))$,
      i.e.,
      $\Phi_{\bm{S}}^{-1}(Q_{N,p})=\Xi\circ\varphi_{\bm{S}}^{-1}(Q_{N,p}) = \Xi\circ\varphi_{\bm{S}}^{-1}(Q_{N,N})$,
      where
      $\Xi$
      is defined as in Proposition~\ref{proposition:inverse}.
    \item \label{enum:dense}
      $\St(p,N)\setminus E_{N,p}(\bm{S})$
      is an open dense subset of
      $\St(p,N)$
      (see Fact~\ref{fact:stiefel}~\ref{enum:submanifold} for the topology of $\St(p,N)$).
    \item \label{enum:intersection_dense}
      For
      $\bm{S}_{1},\bm{S}_{2} \in {\rm O}(N)$,
      the subset
      $\Delta(\bm{S}_{1},\bm{S}_{2}):=  (\St(p,N)\setminus E_{N,p}(\bm{S}_{1})) \cap (\St(p,N)\setminus E_{N,p}(\bm{S}_{2}))$
      is a nonempty open dense subset of
      $\St(p,N)$.
    \item \label{enum:characterize_singular}
      Let
      $g:Q_{N,p}(\bm{S})\to\mathbb{R}:\bm{V}\mapsto \det(\bm{I}_{p} + \bm{S}_{\rm le}^{\T}\Phi_{\bm{S}}^{-1}(\bm{V}))$.
      Then,
      $g$
      is a positive-valued function and
      $\displaystyle \lim_{\substack{\bm{V}\in Q_{N,p}(\bm{S}) \\ \|\bm{V}\|_{2}\to\infty}} g(\bm{V}) = 0$.
      Conversely, if a sequence
      $(\bm{V}_{n})_{n=0}^{\infty} \subset Q_{N,p}(\bm{S})$
      satisfies
      $\lim_{n\to \infty}g(\bm{V}_{n}) = 0$,
      then
      $\lim_{n\to \infty}\|\bm{V}_{n}\|_{2} = \infty$.
  \end{enumerate}
\end{theorem}
\begin{proof}
  See Appendix~\ref{appendix:dense}.
\end{proof}

\begin{proposition}[Properties of G-L$^{2}$CT in view of the manifold theory]\label{proposition:properties}
  \mbox{}
  \begin{enumerate}[label= (\alph*)]
    \item (Chart).\label{enum:chart}
      For
      $\bm{S} \in {\rm O}(N)$,
      the ordered pair
      $(\St(p,N)\setminus E_{N,p}(\bm{S}),\Phi_{\bm{S}})$
      is a chart of
      $\St(p,N)$,
      i.e.,
      (i)
      $\St(p,N)\setminus E_{N,p}(\bm{S})$
      is an open subset of
      $\St(p,N)$;
      (ii)
      $\Phi_{\bm{S}}$
      is a homeomorphism between
      $\St(p,N)\setminus E_{N,p}(\bm{S})$
      and the
      $Np-\frac{1}{2}p(p+1)$
      dimensional Euclidean space
      $Q_{N,p}(\bm{S})$.
    \item (Smooth atlas).\label{enum:atlas}
      The set
      $(\St(p,N)\setminus E_{N,p}(\bm{S}), \Phi_{\bm{S}})_{\bm{S} \in {\rm O}(N)}$
      is a smooth atlas of
      $\St(p,N)$,
      i.e.,
      (i)
      $\bigcup_{\bm{S}\in {\rm O}(N)} (\St(p,N)\setminus E_{N,p}(\bm{S})) = \St(p,N)$;
      (ii)
      for every pair
      $\bm{S}_{1},\bm{S}_{2} \in {\rm O}(N)$,
      $\Phi_{\bm{S}_{2}} \circ \Phi_{\bm{S}_{1}}^{-1}$
      is smooth over
      $\Phi_{\bm{S}_{1}}(\Delta(\bm{S}_{1},\bm{S}_{2}))$,
      where
      $\Delta(\bm{S}_{1},\bm{S}_{2}) \neq \emptyset$
      has been defined in Theorem~\ref{theorem:dense}~\ref{enum:intersection_dense}.
  \end{enumerate}
\end{proposition}
\begin{proof}
  \ref{enum:chart}
  (i)
  See Theorem~\ref{theorem:dense}~\ref{enum:dense}.
  (ii)
  From Proposition~\ref{proposition:inverse},
  $\Phi_{\bm{S}}$
  is a homeomorphism between
  $\St(p,N)\setminus E_{N,p}(\bm{S})$
  and
  $Q_{N,p}(\bm{S})$.
  Clearly the dimension of the vector space
  $Q_{N,p}(\bm{S})$
  is
  $Np-p(p+1)/2$.

  \ref{enum:atlas}
  (i) Recall that
  $\St(p,N) = \bigcup_{\bm{S} \in{\rm O}(N)}\{\bm{S}\bm{I}_{N\times p}\} = \bigcup_{\bm{S}\in {\rm O}(N)}\{\bm{S}_{\rm le}\}=\bigcup_{\bm{S}\in {\rm O}(N)}\{\Phi_{\bm{S}}^{-1}(\bm{0})\}$.
  (ii) See Proposition~\ref{proposition:inverse}.
\end{proof}

\begin{remark}
  \begin{enumerate}[label=(\alph*)]
    \item (Relation between the Cayley transform-based retraction and $\Phi_{\bm{S}}^{-1}$).
      By using
      $\varphi^{-1}$
      in~\eqref{eq:inv_origin_Cayley}, the Cayley transform-based retraction has been utilized for Problem~\ref{problem:origin}, e.g.,~\cite{W13,Z17,Zhu-Sato20} (see Appendix~\ref{appendix:retraction} for the retraction-based strategy).
      The Cayley transform-based retraction can be expressed by using the proposed
      $\Phi_{\bm{S}}^{-1}$
      (see~\eqref{eq:Phi_retraction}).
      In Section~\ref{sec:Cayley_retraction}, we will clarify a diffeomorphic property of this retraction through
      $\Phi_{\bm{S}}^{-1}$.
    \item
      (Parametrization of $\St(p,N)$ with $\Phi_{\bm{S}}$).
      By
      $\St(p,N) \setminus E_{N,p}(\bm{S}) (=\Phi_{\bm{S}}^{-1}(Q_{N,p}(\bm{S}))) \subsetneq \St(p,N)$,
      for a given pair of
      $\bm{U} \in \St(p,N)$
      and
      $\bm{S} \in {\rm O}(N)$,
      the inclusion
      $\bm{U} \in \St(p,N)\setminus E_{N,p}(\bm{S})$
      is not guaranteed in general.
      However, Proposition~\ref{proposition:properties}~\ref{enum:atlas} ensures the existence of
      $\bm{S} \in {\rm O}(N)$
      satisfying
      $\bm{U} \in \St(p,N)\setminus E_{N,p}(\bm{S})$.
      Indeed, we can construct such
      $\bm{S}$
      by using a singular value decomposition of
      $\bm{U}_{\rm up}\in \mathbb{R}^{p\times p}$
      as shown later in Theorem~\ref{theorem:construct_center}.
      This fact tells us that the availability of general
      $\bm{S} \in {\rm O}(N)$
      can realize overall parametrization of
      $\St(p,N)$
      with
      $\Phi_{\bm{S}}^{-1}$.
      We note that a naive idea for using
      $\Phi_{\bm{I}}^{-1}$,
      i.e., a special case of
      $\Phi_{\bm{S}}^{-1}$
      with
      $\bm{S} = \bm{I}$,
      in optimization over
      $\St(p,N)$
      has been reported shortly in~\cite{C07}, which can be seen as an extension of the Cayley parametrization in~\cite{Y03} for optimization over
      ${\rm O}(N)$.
    \item
      (On the choice of $\Xi:{\rm O}(N) \to \St(p,N)$ for $\Phi_{\bm{S}}^{-1}=\Xi\circ\varphi_{\bm{S}}^{-1}$ in Proposition~\ref{proposition:inverse}).
      Since
      $\Xi$
      defined in Proposition~\ref{proposition:inverse} selects the first
      $p$
      column vectors from an orthogonal matrix,
      $\Phi_{\bm{S}}^{-1}(\bm{V}) = \Xi\circ\varphi_{\bm{S}}^{-1}(\bm{V})$
      for
      $\bm{V} \in Q_{N,p}(\bm{S})$
      can be regarded as the matrix of the first
      $p$
      column vectors selected from an orthogonal matrix
      $\varphi_{\bm{S}}^{-1}(\bm{V})$.
      Proposition~\ref{proposition:inverse} guarantees that the matrix
      $\Phi_{\bm{S}}^{-1}(\bm{V})$
      of the first
      $p$
      column vectors of
      $\varphi_{\bm{S}}^{-1}(\bm{V})$
      does not overlap in
      $\bm{V} \in Q_{N,p}(\bm{S})$.
      Although there are many other selection rules
      $\Xi_{\langle \bm{\mathfrak{U}} \rangle}:{\rm O}(N) \to \St(p,N):\bm{U}\mapsto \bm{U}\bm{\mathfrak{U}}$
      with
      $\bm{\mathfrak{U}} \in \{\bm{\mathfrak{U}}' \in \St(p,N) \mid [\bm{\mathfrak{U}}']_{ij} \in \{0,1\},\ 1\leq i \leq N, 1 \leq j \leq p\}$
      of
      $p$
      column vectors from
      $\varphi_{\bm{S}}^{-1}(\bm{V})$,
      $\Xi_{\langle \bm{\mathfrak{U}} \rangle} \circ \varphi_{\bm{S}}^{-1}$
      can not necessarily parameterize
      $\St(p,N)$
      without any overlap as shown below.
      For simplicity, assume
      $2p < N$.
      Consider
      $\bm{\mathfrak{U}}$
      satisfying
      $\bm{\mathfrak{U}}_{\rm up} = \bm{0}$
      ($\bm{\mathfrak{U}}:=[\bm{0} \ \bm{I}_{p}]^{\T}$ is such a typical instance).
      Then, we can verify that
      $\Xi_{\langle\bm{\mathfrak{U}}\rangle}\circ \varphi_{\bm{S}}^{-1}(\bm{V})$
      is not an injection on
      $Q_{N,p}$
      (see Appendix~\ref{appendix:how_to_choice_projection}).
      Note that an idea for using
      $\Xi_{\langle \bm{\mathfrak{U}} \rangle} \circ \varphi_{\bm{S}}^{-1}$
      only with
      $\bm{S} = \bm{I}$
      have been considered in~\cite{C07}.
      However, for parametrization of
      $\St(p,N)$,
      it seems to suggest the special selection
      $\bm{\mathfrak{U}} = \bm{I}_{N\times p}$,
      which corresponds to
      $\Phi_{\bm{I}}^{-1}$.
  \end{enumerate}
\end{remark}
By using Theorem~\ref{theorem:dense}, we deduce Lemma~\ref{lemma:gap}, which guarantees the existence of a solution to Problem~\ref{problem:CP_St} for any
$\epsilon > 0$.
Theorem~\ref{theorem:dense} will also be used in Lemma~\ref{lemma:gap_grad} to ensure the existence of a solution to Problem~\ref{problem:CP_grad}.
\begin{lemma} \label{lemma:gap}
  Let
  $f:\mathbb{R}^{N\times p} \to \mathbb{R}$
  be continuous with
  $p < N$
  and
  $\bm{S} \in {\rm O}(N)$.
  Then, it holds
  \begin{equation} \label{eq:min_inf}
    \min_{\bm{U}\in \St(p,N)} f(\bm{U}) = \inf_{\bm{U}\in \St(p,N)\setminus E_{N,p}(\bm{S})} f(\bm{U}) = \inf_{\bm{V} \in Q_{N,p}(\bm{S})} f\circ \Phi_{\bm{S}}^{-1}(\bm{V}).
  \end{equation}
\end{lemma}
\begin{proof}
  The second equality in~\eqref{eq:min_inf} is verified from the homeomorphism of
  $\Phi_{\bm{S}}^{-1}$.
  Let
  $\bm{U}^{\star} \in \St(p,N)$
  be a global minimizer of
  $f$
  over
  $\St(p,N)$,
  i.e.,
  $f(\bm{U}^{\star}) = \min f(\St(p,N))$.
  From the denseness of
  $\St(p,N) \setminus E_{N,p}(\bm{S})$
  in
  $\St(p,N)$
  (see Theorem~\ref{theorem:dense}~\ref{enum:dense}), there exists a sequence
  $(\bm{U}_{n})_{n=0}^{\infty} \subset \St(p,N) \setminus E_{N,p}(\bm{S})$
  satisfying
  $\lim_{n\to \infty} \bm{U}_{n} = \bm{U}^{\star}$.
  The continuity of
  $f$
  yields
  $\lim_{n\to \infty} f(\bm{U}_{n}) = f(\bm{U}^{\star})$,
  i.e.,
  $\inf_{\bm{U}\in \St(p,N)\setminus E_{N,p}(\bm{S})} f(\bm{U}) = \min f(\St(p,N))$.
\end{proof}

\subsection[Computational complexities for G-L2CT and its inversion]{Computational complexities for $\Phi_{\bm{S}}$ and $\Phi_{\bm{S}}^{-1}$ with $\bm{S} \in {\rm O}_{p}(N)$}\label{sec:complexity}
From the expressions
in~\eqref{eq:Cayley}-\eqref{eq:Cay_B} and~\eqref{eq:Cayley_inv}, both
$\Phi_{\bm{S}}$
and
$\Phi_{\bm{S}}^{-1}$
with general
$\bm{S} \in{\rm O}(N)$
require
$\mathfrak{o}(N^2p+p^3)$
flops (FLoating-point OPerationS [not "FLoating point Operations Per Second"]), which are dominated by the matrix multiplications
$\bm{S}_{\rm ri}^{\T}\bm{U}$
in~\eqref{eq:Cay_B} and
$\bm{S}_{\rm ri}\dbra{\bm{V}}_{21}$
in~\eqref{eq:Cayley_inv} respectively.
However, if we employ a special center point
\begin{equation}
  \bm{S} \in \textrm{O}_{p}(N) := \left\{\diag(\bm{T},\bm{I}_{N-p})\mid  \bm{T}\in \textrm{O}(p)\right\}   \subset {\rm O}(N), \label{eq:structure}
\end{equation}
then the complexities for
$\Phi_{\bm{S}}$
and
$\Phi_{\bm{S}}^{-1}$
can be reduced to
$\mathfrak{o}(Np^{2} +p^3)$ flops.
Indeed, for
$\bm{T}\in\textrm{O}(p)$
and
$\bm{U} \in \St(p,N)$,
we have
$[\diag(\bm{T},\bm{I}_{N-p})]_{\rm le}^{\T}\bm{U} = \bm{T}^{\T}\bm{U}_{\rm up}$
and
$[\diag(\bm{T},\bm{I}_{N-p})]_{\rm ri}^{\T}\bm{U} = \bm{U}_{\rm lo}$.
Hence
$\Phi_{\diag(\bm{T},\bm{I}_{N-p})}(\bm{U})$
requires
$Np^2+\mathfrak{o}(p^3)$ flops due to
\begin{align}
  \bm{A}_{\diag(\bm{T},\bm{I}_{N-p})}(\bm{U}) &= 2(\bm{I}_{p}+\bm{T}^{\T}\bm{U}_{\rm up})^{-\mathrm{T}}\Skew(\bm{U}_{\rm up}^{\T}\bm{T})(\bm{I}_{p}+\bm{T}^{\T}\bm{U}_{\rm up})^{-1} \in \mathbb{R}^{p\times p}, \\
  \bm{B}_{\diag(\bm{T},\bm{I}_{N-p})}(\bm{U}) &= - \bm{U}_{\rm lo}(\bm{I}_{p}+\bm{T}^{\T}\bm{U}_{\rm up})^{-1} \in \mathbb{R}^{(N-p)\times p}.
\end{align}
Moreover, for
$\bm{V}\in Q_{N,p}(\diag(\bm{T},\bm{I}_{N-p}))$
and
$\bm{M} := \bm{I}_{p}+ \dbra{\bm{V}}_{11} + \dbra{\bm{V}}_{21}^{\T}\dbra{\bm{V}}_{21}\in \mathbb{R}^{p\times p}$,
it follows from
$[\diag(\bm{T},\bm{I}_{N-p})]_{\rm ri}\dbra{\bm{V}}_{21} = \begin{bmatrix} \bm{0}_{p} & \dbra{\bm{V}}_{21}^{\T} \end{bmatrix}^{\T}$ and~\eqref{eq:Cayley_inv} that
\begin{align}
  \Phi_{\diag(\bm{T},\bm{I}_{N-p})}^{-1}(\bm{V})
  = \begin{bmatrix} 2\bm{T}\bm{M}^{-1} - \bm{T} \\ -2 \dbra{\bm{V}}_{21}\bm{M}^{-1} \end{bmatrix}\label{eq:special_expression}
\end{align}
requires
$2Np^2 + \mathfrak{o}(p^3)$ flops.

For a given
$\bm{U}\in \St(p,N)$,
Theorem~\ref{theorem:construct_center} below presents a way to select
$\bm{T}\in {\rm O}(p)$
satisfying
$\bm{U} \in \St(p,N)\setminus E_{N,p}(\diag(\bm{T},\bm{I}_{N-p}))$,
where
$\bm{T}$
is designed with a singular value decomposition of
$\bm{U}_{\rm up}\in\mathbb{R}^{p\times p}$,
requiring thus at most
$\mathfrak{o}(p^3)$ flops.
\begin{theorem}[Parametrization of $\St(p,N)$ by $\Phi_{\bm{S}}$ with $\bm{S} \in {\rm O}_{p}(N)$]\label{theorem:construct_center}
  Let
  $\bm{U}=\begin{bmatrix} \bm{U}_{\rm up}^{\T} & \bm{U}_{\rm lo}^{\T} \end{bmatrix}^{\T}\in \St(p,N)$,
  and
  $\bm{U}_{\rm up} = \bm{Q}_{1}\bm{\Sigma}\bm{Q}_{2}^{\T}$
  be a singular value decomposition of
  $\bm{U}_{\rm up}\in\mathbb{R}^{p\times p}$,
  where
  $\bm{Q}_{1},\bm{Q}_{2} \in \textrm{O}(p)$
  and
  $\bm{\Sigma}\in\mathbb{R}^{p\times p}$
  is a diagonal matrix with non-negative entries.
  Define
  $\bm{S}:=\diag(\bm{T},\bm{I}_{N-p}) \in {\rm O}_{p}(N)$
  with
  $\bm{T}:=\bm{Q}_{1}\bm{Q}_{2}^{\T}\in {\rm O}(p)$.
  Then, the following hold:
  \begin{enumerate}[label=(\alph*)]
    \item \label{enum:dense_nonsingular}
      $\det(\bm{I}_{p}+\bm{S}_{\rm le}^{\T}\bm{U}) \geq 1$
      and
      $\bm{U}\in \St(p,N)\setminus E_{N,p}(\bm{S})$.
    \item \label{enum:dense_norm}
      \begin{equation}
        \Phi_{\bm{S}}(\bm{U}) =
        \begin{bmatrix}
          \bm{0} & \bm{Q}_{2}(\bm{I}_{p}+\bm{\Sigma})^{-1}\bm{Q}_{2}^{\T}\bm{U}_{\rm lo}^{\T}\\
          -\bm{U}_{\rm lo}\bm{Q}_{2}(\bm{I}_{p}+\bm{\Sigma})^{-1}\bm{Q}_{2}^{\T} & \bm{0}
        \end{bmatrix}, \label{eq:const_state}
      \end{equation}
      where
      $\|\bm{B}_{\bm{S}}(\bm{U})\|_{2} \overset{\eqref{eq:Cay_B}}{=} \|\bm{U}_{\rm lo}\bm{Q}_{2}(\bm{I}_{p}+\bm{\Sigma})^{-1}\bm{Q}_{2}^{\T}\|_{2} \leq 1$.
  \end{enumerate}
\end{theorem}
\begin{proof}
  \ref{enum:dense_nonsingular}
  By
  $\bm{S}_{\rm le}^{\T}\bm{U} =\bm{T}^{\T}\bm{U}_{\rm up} = \bm{Q}_{2}\bm{\Sigma}\bm{Q}_{2}^{\T}$,
  it holds
  $\det(\bm{I}_{p}+\bm{S}_{\rm le}^{\T}\bm{U}) = \det(\bm{Q}_{2}(\bm{I}_{p}+\bm{\Sigma})\bm{Q}_{2}^{\T}) = \det(\bm{I}_{p}+\bm{\Sigma}) \geq 1$,
  which implies
  $\bm{U} \in \St(p,N)\setminus E_{N,p}(\bm{S})$
  by Definition~\ref{definition:Cayley}.

  \ref{enum:dense_norm}
  Substituting
  $\bm{S}_{\rm le}^{\T}\bm{U} =\bm{Q}_{2}\bm{\Sigma}\bm{Q}_{2}^{\T}$
  and
  $\bm{S}_{\rm ri}^{\T}\bm{U} = \bm{U}_{\rm lo}$
  into~\eqref{eq:Cay_A} and~\eqref{eq:Cay_B}, we obtain~\eqref{eq:const_state}.
  From~\eqref{eq:Cay_B},
  $\|\bm{B}_{\bm{S}}(\bm{U})\|_{2}$
  is bounded above as
  \begin{align}
    & \|\bm{B}_{\bm{S}}(\bm{U})\|_{2}
    = \|\bm{S}_{\rm ri}^{\T}\bm{U}(\bm{I}_{p}+\bm{S}_{\rm le}^{\T}\bm{U})^{-1}\|_{2}
    \leq \|\bm{S}_{\rm ri}\|_{2}\|\bm{U}\|_{2}\|(\bm{I}_{p}+\bm{Q}_{2}\bm{\Sigma}\bm{Q}_{2}^{\T})^{-1}\|_{2}\\
    & = \|\bm{Q}_{2}(\bm{I}_{p}+\bm{\Sigma})^{-1}\bm{Q}_{2}^{\T}\|_{2}
    = \|(\bm{I}_{p}+\bm{\Sigma})^{-1}\|_{2} \leq 1. \label{eq:theorem_construct_BS}
  \end{align}
\end{proof}
\begin{remark}[Comparisons to commonly used retractions of $\St(p,N)$]\label{remark:comparison_complexity}
  The computational complexity
  $2Np^2 + \mathfrak{o}(p^3)$
  flops for
  $\Phi_{\bm{S}}^{-1}$
  with
  $\bm{S} \in {\rm O}_{p}(N)$
  is competitive to that for commonly used retractions, which map a tangent vector to a point in
  $\St(p,N)$
  (for the retraction-based strategy, see Appendix~\ref{appendix:retraction}).
  Indeed, retractions based on QR decomposition, the polar decomposition~\cite{manifold_book} and the Cayley transform~\cite{W13} require respectively
  $2Np^2 + \mathfrak{o}(p^3)$
  flops,
  $3Np^2 + \mathfrak{o}(p^3)$
  flops and
  $6Np^2 + \mathfrak{o}(p^3)$
  flops~\cite[Table 1]{Z17}.
\end{remark}

\subsection{Gradient of function after the Cayley parametrization}\label{sec:gradient}
For the applications of
$\Phi_{\bm{S}}$
(G-L$^{2}$CT)
with
$\bm{S}\in {\rm O}(N)$
to Problems~\ref{problem:CP_St} and~\ref{problem:CP_grad}, we present an expression of the gradient of
$f\circ\Phi_{\bm{S}}^{-1}$
denoted by
$\nabla (f\circ\Phi_{\bm{S}}^{-1})$
(Proposition~\ref{proposition:gradient})
and its useful properties (Proposition~\ref{proposition:change_center}, Remark~\ref{remark:gradient} and also Proposition~\ref{proposition:gradient_property}).
\begin{proposition}[Gradient of function after the Cayley parametrization] \label{proposition:gradient}
  For a differentiable function
  $f:\mathbb{R}^{N\times p}\to \mathbb{R}$
  and
  $\bm{S} \in {\rm O}(N)$,
  the function
  $f_{\bm{S}}:=f\circ\Phi_{\bm{S}}^{-1}:Q_{N,p}(\bm{S}) \to \mathbb{R}$
  is differentiable with
  \begin{align}
    (\bm{V}\in Q_{N,p}(\bm{S})) \quad \nabla f_{\bm{S}}(\bm{V}) = 2\Skew(\bm{W}^{f}_{\bm{S}}(\bm{V})) = \bm{W}^{f}_{\bm{S}}(\bm{V}) - \bm{W}^{f}_{\bm{S}}(\bm{V})^{\T}\in Q_{N,p}(\bm{S}), \label{eq:grad_propo}
  \end{align}
  where
  \begin{equation}
    \bm{W}^{f}_{\bm{S}}(\bm{V})
    := \begin{bmatrix}
      \dbra{\overline{\bm{W}}_{\bm{S}}^{f}(\bm{V})}_{11} & \dbra{\overline{\bm{W}}_{\bm{S}}^{f}(\bm{V})}_{12} \\ 
      \dbra{\overline{\bm{W}}_{\bm{S}}^{f}(\bm{V})}_{21} & \bm{0}
    \end{bmatrix} \in \mathbb{R}^{N\times N} \label{eq:matrix_W}
  \end{equation}
  and
  {
  \thickmuskip=0.0\thickmuskip
  \medmuskip=0.0\medmuskip
  \thinmuskip=0.0\thinmuskip
  \begin{align}
    & \overline{\bm{W}}^{f}_{\bm{S}}(\bm{V})
    := (\bm{I}+\bm{V})^{-1}\bm{I}_{N\times p}\nabla f(\Phi_{\bm{S}}^{-1}(\bm{V}))^{\T}\bm{S}(\bm{I}+\bm{V})^{-1} \label{eq:matrix_W_bar} \\
    = & \begin{bmatrix}
      \bm{M}^{-1}\nabla f(\bm{U})^{\T}(\bm{S}_{\rm le} -\bm{S}_{\rm ri}\dbra{\bm{V}}_{21})\bm{M}^{-1}
      &
      \bm{M}^{-1}
      \nabla f(\bm{U})^{\T}
      ((\bm{S}_{\rm le} -\bm{S}_{\rm ri}\dbra{\bm{V}}_{21})\bm{M}^{-1}\dbra{\bm{V}}_{21}^{\T} +\bm{S}_{\rm ri})
      \\
      -\dbra{\bm{V}}_{21}\bm{M}^{-1}\nabla f(\bm{U})^{\T}(\bm{S}_{\rm le} -\bm{S}_{\rm ri}\dbra{\bm{V}}_{21})\bm{M}^{-1}
      &
      -\dbra{\bm{V}}_{21}\bm{M}^{-1}
      \nabla f(\bm{U})^{\T}
      ((\bm{S}_{\rm le} -\bm{S}_{\rm ri}\dbra{\bm{V}}_{21})\bm{M}^{-1}\dbra{\bm{V}}_{21}^{\T} +\bm{S}_{\rm ri})
    \end{bmatrix} \\
      \in&  \mathbb{R}^{N\times N}\label{eq:gradient_block_original}
  \end{align}
}%
  in terms of
  $\bm{U}:=\Phi_{\bm{S}}^{-1}(\bm{V}) \in \St(p,N)\setminus E_{N,p}(\bm{S})$
  and
  $\bm{M}:=\bm{I}_{p}+\dbra{\bm{V}}_{11}+\dbra{\bm{V}}_{21}^{\T}\dbra{\bm{V}}_{21}\in \mathbb{R}^{p\times p}$.
  In particular, by
  $\bm{S}_{\rm le} = \Phi_{\bm{S}}^{-1}(\bm{0})$
  in~\eqref{eq:Cayley_inv_origin},
  \begin{align}
    \nabla f_{\bm{S}}(\bm{0}) & = \begin{bmatrix} \nabla f(\bm{S}_{\rm le})^{\T}\bm{S}_{\rm le}-\bm{S}_{\rm le}^{\T}\nabla f(\bm{S}_{\rm le}) & \nabla f(\bm{S}_{\rm le})^{\T}\bm{S}_{\rm ri} \\ -\bm{S}_{\rm ri}^{\T}\nabla f(\bm{S}_{\rm le}) & \bm{0} \end{bmatrix}. \label{eq:gradient_0}
  \end{align}
\end{proposition}
\begin{proof}
  See Appendix~\ref{appendix:gradient}.
\end{proof}

\begin{proposition}[Transformation formula for gradients of function]\label{proposition:change_center}
  For
  $\bm{S}_{1},\bm{S}_{2}\in {\rm O}(N)$,
  suppose that
  $\bm{V}_{1}\in Q_{N,p}(\bm{S}_{1})$
  and
  $\bm{V}_{2}\in Q_{N,p}(\bm{S}_{2})$
  satisfy
  $\Phi_{\bm{S}_{1}}^{-1}(\bm{V}_{1}) = \Phi_{\bm{S}_{2}}^{-1}(\bm{V}_{2})$.
  Then, for a differentiable function
  $f:\mathbb{R}^{N\times p}\to \mathbb{R}$,
  the following hold:
  \begin{enumerate}[label=(\alph*)]
    \item \label{enum:change_formula}
      $\bm{\mathfrak{X}}:= [\varphi_{\bm{S}_{1}}^{-1}(\bm{V}_{1})]_{\rm ri}^{\T}[\varphi_{\bm{S}_{2}}^{-1}(\bm{V}_{2})]_{\rm ri} \in {\rm O}(N-p)$
      is guaranteed.
      Moreover, by using
      \begin{align}
        &\mathcal{G}_{\bm{S}_{1},\bm{S}_{2}}(\bm{V}_{1},\bm{V}_{2})
        :=
        (\bm{I}+\bm{V}_{1})^{-1}\begin{bmatrix}\bm{I}_{p}&\bm{0}\\\bm{0}&\bm{\mathfrak{X}}\end{bmatrix}(\bm{I}+\bm{V}_{2}) \\
        & \left( 
        \nabla f_{\bm{S}_{2}}(\bm{V}_{2})-\begin{bmatrix}\bm{0}&\bm{0}\\ \dbra{\bm{V}_{2}}_{21}&\bm{I}_{N-p}\end{bmatrix} \nabla f_{\bm{S}_{2}}(\bm{V}_{2})\begin{bmatrix}\bm{0}&\dbra{\bm{V}_{2}}_{21}^{\T}\\ \bm{0}&\bm{I}_{N-p}\end{bmatrix}
        \right)
        (\bm{I}+\bm{V}_{2})^{\T}\begin{bmatrix}\bm{I}_{p}&\bm{0}\\\bm{0}&\bm{\mathfrak{X}}^{\T}\end{bmatrix}(\bm{I}+\bm{V}_{1})^{-\T} \in Q_{N,N}, \label{eq:cal_F}
      \end{align}
      we have
      \begin{align}\label{eq:center_change_grad}
        \nabla f_{\bm{S}_{1}}(\bm{V}_{1})
        & = \begin{bmatrix}
          \dbra{\mathcal{G}_{\bm{S}_{1},\bm{S}_{2}}(\bm{V}_{1},\bm{V}_{2})}_{11} & \dbra{\mathcal{G}_{\bm{S}_{1},\bm{S}_{2}}(\bm{V}_{1},\bm{V}_{2})}_{12} \\
          \dbra{\mathcal{G}_{\bm{S}_{1},\bm{S}_{2}}(\bm{V}_{1},\bm{V}_{2})}_{21} & \bm{0}
        \end{bmatrix} \in Q_{N,p}(\bm{S}_{1})\\
        & = \mathcal{G}_{\bm{S}_{1},\bm{S}_{2}}(\bm{V}_{1},\bm{V}_{2})
        - \begin{bmatrix} \bm{0}&\bm{0}\\\bm{0}&\bm{I}_{N-p}\end{bmatrix}
        \mathcal{G}_{\bm{S}_{1},\bm{S}_{2}}(\bm{V}_{1},\bm{V}_{2})
        \begin{bmatrix} \bm{0}&\bm{0}\\\bm{0}&\bm{I}_{N-p}\end{bmatrix}. \label{eq:gradient_translation}
      \end{align}
    \item \label{enum:change_center}
      $\|\nabla f_{\bm{S}_{1}}(\bm{V}_{1})\|_{F} \leq 2(1+\|\bm{V}_{2}\|_{2}^{2})\|\nabla f_{\bm{S}_{2}}(\bm{V}_{2})\|_{F}$.
    \item \label{enum:zero}
      $\nabla f_{\bm{S}_{1}}(\bm{V}_{1}) = \bm{0}$
      if and only if
      $\nabla f_{\bm{S}_{2}}(\bm{V}_{2}) = \bm{0}$.
  \end{enumerate}
\end{proposition}
\begin{proof}
  See Appendix~\ref{appendix:change_center}.
\end{proof}

\begin{remark}\label{remark:gradient}
  \begin{enumerate}[label=(\alph*)]
    \item (Computational complexity for $\nabla (f\circ\Phi_{\bm{S}}^{-1})$ with $\bm{S}\in {\rm O}_{p}(N)$ in~\eqref{eq:structure}).\label{enum:economical_gradient}
      Let
      $\bm{S}:=\diag(\bm{T},\bm{I}_{N-p}) \in {\rm O}_{p}(N)$
      with
      $\bm{T} \in \textrm{O}(p)$
      and
      $\bm{V} \in Q_{N,p}(\bm{S})$.
      From~\eqref{eq:matrix_W} and~\eqref{eq:gradient_block_original}, computation of
      $\nabla (f\circ \Phi_{\bm{S}}^{-1})(\bm{V})(= 2\Skew(\bm{W}^{f}_{\bm{S}}(\bm{V})))$
      requires at most
      $5Np^{2} + \mathfrak{o}(p^3)$
      flops due to
      \begin{equation}
        \hspace{-2em}
        \thickmuskip=0.1\thickmuskip
        \medmuskip=0.1\medmuskip
        \thinmuskip=0.1\thinmuskip
        \arraycolsep=0.0\arraycolsep
        \bm{W}^{f}_{\bm{S}}(\bm{V}) =
        \begin{bmatrix}
          \bm{M}^{-1}\nabla f(\bm{U})^{\T}
          \begin{bmatrix}
            \bm{T}\\ -\dbra{\bm{V}}_{21}
          \end{bmatrix}
          \bm{M}^{-1}
          &
          \bm{M}^{-1}\nabla f(\bm{U})^{\T}
          \left(
          \begin{bmatrix}
            \bm{T}\\ -\dbra{\bm{V}}_{21}
          \end{bmatrix}
          \bm{M}^{-1}\dbra{\bm{V}}_{21}^{\T}
          +
          \begin{bmatrix}
            \bm{0} \\ \bm{I}_{N-p}
          \end{bmatrix}
          \right)
          \\
          -\dbra{\bm{V}}_{21}\bm{M}^{-1}\nabla f(\bm{U})^{\T}
          \begin{bmatrix}
            \bm{T}\\ -\dbra{\bm{V}}_{21}
          \end{bmatrix}
          \bm{M}^{-1}
          &
          \bm{0}
        \end{bmatrix},
      \end{equation}
      where
      $\bm{U}:=\Phi_{\bm{S}}^{-1}(\bm{V}) \in \St(p,N)\setminus E_{N,p}(\bm{S})$
      and
      $\bm{M}:=\bm{I}_{p}+\dbra{\bm{V}}_{11}+\dbra{\bm{V}}_{21}^{\T}\dbra{\bm{V}}_{21} \in \mathbb{R}^{p\times p}$.
    \item (Relation of gradients after Cayley parametrization).
      Proposition~\ref{proposition:change_center} illustrates the relations of two gradients after Cayley parameterization with different center points.
      These relations will be used in Lemmas~\ref{lemma:optimality} and~\ref{lemma:gap_grad} to characterize the first-order optimality condition with the proposed Cayley parametrization.
    \item (Useful properties of the gradient after Cayley parametrization). \label{enum:gradient_property}
      In Appendix~\ref{appendix:gradient_property}, we present useful properties
      (i) Lipschitz continuity; (ii) the boundedness; (iii) the variance bounded;
      of
      $\nabla (f\circ\Phi_{\bm{S}}^{-1})$
      for the minimization of
      $f\circ\Phi_{\bm{S}}^{-1}$
      over
      $Q_{N,p}(\bm{S})$.
      These properties have been exploited in distributed optimization and stochastic optimization, e.g.,~\cite{Reddi-Hefny-Sra16,G16,Zeyuan18,Ward-Wu-Bottou,Chen-Liu-Sun-Hong19,Tatarenko-Touri17}.
  \end{enumerate}
\end{remark}

\section{Optimization over the Stiefel manifold with the Cayley parametrization}\label{sec:technique}
\subsection{Optimality condition via the Cayley parametrization}\label{sec:optimality}
We present simple characterizations of (i) local minimizer, and (ii) stationary point, of a real valued function over
$\St(p,N)$
in terms of
$\Phi_{\bm{S}}$.

Let
$\mathcal{X}$
be a vector space of matrices.
A point
$\bm{X}^{\star} \in \mathcal{Y} \subset \mathcal{X}$
is said to be {\it a local minimizer of
$J:\mathcal{X} \to \mathbb{R}$
over
$\mathcal{Y} \subset \mathcal{X}$}
if there exists
$\epsilon > 0$
satisfying
$J(\bm{X}^{\star}) \leq J(\bm{X})$
for all
$\bm{X} \in B_{\mathcal{X}}(\bm{X}^{\star},\epsilon) \cap \mathcal{Y}$.
Under the smoothness assumption on
$J$,
$\bm{X} \in \mathcal{X}$
is said to be {\it a stationary point of
$J$
over the vector space
$\mathcal{X}$}
if
$\nabla J(\bm{X}) = \bm{0}$.
For a smooth function
$f:\mathbb{R}^{N\times p}\to \mathbb{R}$,
$\bm{U} \in \St(p,N)$
is said to be {\it a stationary point of
$f$
over
$\St(p,N)$}~\cite{W13,BXX18}
if
$\bm{U}$
satisfies the following conditions:
\begin{equation}
  \begin{cases}
    (\bm{I}-\bm{U}\bm{U}^{\T})\nabla f(\bm{U}) & = \bm{0} \\
    \bm{U}^{\T}\nabla f(\bm{U}) -\nabla f(\bm{U})^{\T}\bm{U} & = \bm{0}.
  \end{cases}\label{eq:optimality}
\end{equation}
The above conditions
$\nabla J(\bm{X}) = \bm{0}$
and~\eqref{eq:optimality} are called {\it the first-order optimality conditions} because they are respectively necessary conditions for
$\bm{X}$
to be a local minimizer of
$J$
over
$\mathcal{X}$
(see, e.g.,~\cite[Theorem 2.2]{numerical_optimization}),
and for
$\bm{U}$
to be a local minimizer of
$f$
over
$\St(p,N)$
(see \cite[Definition~2.1, Remark~2.3]{BXX18} and~\cite[Lemma~1]{W13}).

In Lemma~\ref{lemma:local} below, we characterize a local minimizer of
$f$
over
$\St(p,N)$
as a local minimizer of
$f\circ\Phi_{\bm{S}}^{-1}$
with a certain
$\bm{S} \in {\rm O}(N)$
over the vector space
$Q_{N,p}(\bm{S})$.
\begin{lemma}[Equivalence of local minimizers in the two senses] \label{lemma:local}
  Let
  $f:\mathbb{R}^{N\times p} \to \mathbb{R}$
  be continuous.
  Let
  $\bm{U}^{\star} \in \St(p,N)$
  and
  $\bm{S} \in {\rm O}(N)$
  satisfy
  $\bm{U}^{\star} \in \St(p,N) \setminus E_{N,p}(\bm{S})$.
  Then,
  $\bm{U}^{\star}$
  is a local minimizer of
  $f$
  over
  $\St(p,N)$
  if and only if
  $\bm{V}^{\star}:=\Phi_{\bm{S}}(\bm{U}^{\star}) \in Q_{N,p}(\bm{S})$
  is a local minimizer of
  $f\circ\Phi_{\bm{S}}^{-1}$
  over the vector space
  $Q_{N,p}(\bm{S})$.
\end{lemma}
\begin{proof}
  Let
  $\bm{U}^{\star}$
  be a local minimizer of
  $f$
  over
  $\St(p,N)$
  and
  $\epsilon > 0$
  satisfy
  $f(\bm{U}^{\star}) \leq f(\bm{U})$
  for all
  $\bm{U} \in B_{\mathbb{R}^{N\times p}}(\bm{U}^{\star}, \epsilon)\cap \St(p,N)=:\mathcal{N}_{1}(\bm{U}^{\star}) \subset \St(p,N)\setminus E_{N,p}(\bm{S})$.
  From the homeomorphism of
  $\Phi_{\bm{S}}$
  in Proposition~\ref{proposition:inverse},
  $\mathcal{N}_{2}(\bm{V}^{\star}):=\Phi_{\bm{S}}(\mathcal{N}_{1}(\bm{U}^{\star})) \subset Q_{N,p}(\bm{S})$
  is a nonempty open subset of
  $Q_{N,p}(\bm{S})$
  containing
  $\bm{V}^{\star}$.
  Then, there exists
  $\widehat{\epsilon} > 0$
  satisfying
  $B_{Q_{N,p}(\bm{S})}(\bm{V}^{\star}, \widehat{\epsilon}) \subset \mathcal{N}_{2}(\bm{V}^{\star})$.
  Since
  $f\circ \Phi_{\bm{S}}^{-1}(B_{Q_{N,p}(\bm{S})}(\bm{V}^{\star}, \widehat{\epsilon})) \subset f\circ \Phi_{\bm{S}}^{-1}(\mathcal{N}_{2}(\bm{V}^{\star})) = f(\mathcal{N}_{1}(\bm{U}^{\star}))$,
  we obtain
  $f(\bm{U}^{\star}) = f\circ \Phi_{\bm{S}}^{-1}(\bm{V}^{\star}) \leq f\circ \Phi_{\bm{S}}^{-1}(\bm{V})$
  for all
  $\bm{V}\in B_{Q_{N,p}(\bm{S})}(\bm{V}^{\star}, \widehat{\epsilon})$,
  implying thus
  $\bm{V}^{\star}$
  is a local minimizer of
  $f\circ \Phi_{\bm{S}}^{-1}$
  over
  $Q_{N,p}(\bm{S})$.
  In a similar way, we can prove its converse.
\end{proof}

Under a special assumption on
$f$
in Theorem~\ref{theorem:range_minimizer} below, yet found especially in many data science scenarios (see Remark~\ref{remark:not_restricted}),
we can characterize a global minimizer of Problem~\ref{problem:origin} via
$f\circ \Phi_{\bm{S}}^{-1}$
with any
$\bm{S} \in {\rm O}(N)$.
In this case, a global minimizer
$\bm{V}^{\star} \in Q_{N,p}(\bm{S})$
of
$f\circ \Phi_{\bm{S}}^{-1}$
is guaranteed to exist in the unit ball
$\{\bm{V} \in Q_{N,p}(\bm{S}) \mid \|\bm{V}\|_{2} \leq 1\}$.
\begin{theorem} \label{theorem:range_minimizer}
  Let
  $\bm{S} \in {\rm O}(N)$.
  Assume that
  $f:\mathbb{R}^{N\times p}\to \mathbb{R}$
  is continuous and {\it right orthogonal invariant}, i.e.,
  $f(\bm{U}) = f(\bm{U}\bm{Q})$
  for
  $\bm{U} \in \St(p,N)$
  and
  $\bm{Q} \in {\rm O}(p)$.
  Then, there exists a global minimizer
  $\bm{V}^{\star} \in Q_{N,p}(\bm{S})$
  of
  $f\circ\Phi_{\bm{S}}^{-1}$
  achieving
  $f\circ\Phi_{\bm{S}}^{-1}(\bm{V}^{\star}) = \min_{\bm{U}\in \St(p,N)} f(\bm{U})$,
  $\|\dbra{\bm{V}^{\star}}_{21}\|_{2} \leq 1$
  and
  $\|\bm{V}^{\star}\|_{2} \leq 1$.
\end{theorem}
\begin{proof}
  Let
  $\bm{U}^{\diamond} \in \St(p,N)$
  be a global minimizer of
  $f$
  over
  $\St(p,N)$,
  and
  $\bm{S}_{\rm le}^{\T}\bm{U}^{\diamond} = \bm{Q}_{1}\bm{\Sigma}\bm{Q}_{2}^{\T}$
  be a singular value decomposition with
  $\bm{Q}_{1},\bm{Q}_{2} \in {\rm O}(p)$
  and nonnegative-valued diagonal matrix
  $\bm{\Sigma} \in \mathbb{R}^{p\times p}$.
  Then, we obtain
  $\bm{U}^{\star}:=\bm{U}^{\diamond}\bm{Q} \in \St(p,N) \setminus E_{N,p}(\bm{S})$
  with
  $\bm{Q} := \bm{Q}_{2}\bm{Q}_{1}^{\T} \in {\rm O}(p)$
  by
  $|\det(\bm{I}_{p}+\bm{S}_{\rm le}^{\T}\bm{U}^{\star})| = |\det(\bm{I}_{p}+\bm{Q}_{1}\bm{\Sigma}\bm{Q}_{2}^{\T}\bm{Q}_{2}\bm{Q}_{1}^{\T})| = |\det(\bm{I}_{p}+\bm{\Sigma})| \geq 1$.
  The right orthogonal invariance of
  $f$
  ensures
  $f(\bm{U}^{\diamond}) = f(\bm{U}^{\star}) = f\circ\Phi_{\bm{S}}^{-1}(\bm{V}^{\star})$
  with
  $\bm{V}^{\star} := \Phi_{\bm{S}}(\bm{U}^{\star})$.

  Substituting
  $\bm{S}_{\rm le}^{\T}\bm{U}^{\star} = \bm{Q}_{1}\bm{\Sigma}\bm{Q}_{1}^{\T}$
  into~\eqref{eq:Cay_A} and~\eqref{eq:Cay_B}, we obtain
  $\dbra{\bm{V}^{\star}}_{11} = \bm{A}_{\bm{S}}(\bm{U}^{\star}) = \bm{0}$
  and
  $\dbra{\bm{V}^{\star}}_{21} = \bm{B}_{\bm{S}}(\bm{U}^{\star}) = -\bm{S}_{\rm ri}^{\T}\bm{U}^{\diamond}\bm{Q}(\bm{I}_{p}+\bm{Q}_{1}\bm{\Sigma}\bm{Q}_{1}^{\T})^{-1} = -\bm{S}_{\rm ri}^{\T}\bm{U}^{\diamond}\bm{Q}_{2}(\bm{I}_{p}+\bm{\Sigma})^{-1}\bm{Q}_{1}^{\T}$.
  In a similar manner to~\eqref{eq:theorem_construct_BS}, the last equality implies
  $\|\dbra{\bm{V}^{\star}}_{21}\|_{2} \leq 1$.
  The last statement is verified by
  \begin{align}
    & \|\bm{V}^{\star}\|_{2}^{2} = \lambda_{\max}\left(\begin{bmatrix} \bm{0} & \dbra{\bm{V}^{\star}}_{21}^{\T} \\ -\dbra{\bm{V}^{\star}}_{21} & \bm{0} \end{bmatrix}\begin{bmatrix} \bm{0} & -\dbra{\bm{V}^{\star}}_{21}^{\T} \\ \dbra{\bm{V}^{\star}}_{21} & \bm{0} \end{bmatrix}\right) \\
    & = \lambda_{\max}\left(\begin{bmatrix} \dbra{\bm{V}^{\star}}_{21}^{\T}\dbra{\bm{V}^{\star}}_{21} & \bm{0} \\ \bm{0} & \dbra{\bm{V}^{\star}}_{21}\dbra{\bm{V}^{\star}}_{21}^{\T} \end{bmatrix}\right)
    = \lambda_{\max}(\dbra{\bm{V}^{\star}}_{21}^{\T}\dbra{\bm{V}^{\star}}_{21}) = \|\dbra{\bm{V}^{\star}}_{21}\|_{2}^{2} \leq 1.
  \end{align}
\end{proof}

\begin{remark}[Right orthogonal invariance] \label{remark:not_restricted}
  Under the right orthogonal invariance of
  $f$,
  Problem~\ref{problem:origin} arises in, e.g., low-rank matrix completion~\cite{Boumal-Absil2015,Pitaval-Dai-Tirkkonen2015}, eigenvalue problems~\cite{manifold_book,W13,Sato-Iwai2014,Z17}, and optimal
  $\mathcal{H}_{2}$
  model reduction~\cite{Xu-Zeng2013,Sato21}.
  These applications can be formulated as optimization problems over {\it the Grassmann manifold}
  $\Gr(p,N)$,
  which is the set of all $p$-dimensional subspace of
  $\mathbb{R}^{N}$.
  Practically,
  $\Gr(p,N)$
  is represented numerically by
  $\{[\bm{U}] \mid \bm{U} \in \St(p,N)\}$,
  where
  $[\bm{U}]:=\{\bm{U}\bm{Q} \in \St(p,N) \mid \bm{Q}\in {\rm O}(p)\}$
  is an equivalence class, because the column space of
  $\bm{U} \in \St(p,N)$
  equals that of
  $\bm{U}\bm{Q} \in \St(p,N)$
  for all
  $\bm{Q}\in {\rm O}(p)$.
  Since the value of the right orthogonal invariant
  $f$
  depends only on the equivalence class
  $[\bm{U}]$,
  Problem~\ref{problem:origin} of such
  $f$
  can be regarded as an optimization problem over
  $\Gr(p,N)$.
\end{remark}

In Lemma~\ref{lemma:optimality} below,
we characterize a stationary point of
$f$
over
$\St(p,N)$
by a stationary point of
$f\circ\Phi_{\bm{S}}^{-1}$
with a certain
$\bm{S} \in {\rm O}(N)$
over the vector space
$Q_{N,p}(\bm{S})$.
Moreover, Lemma~\ref{lemma:gap_grad} ensures the existence of solutions to Problem~\ref{problem:CP_grad} with any
$\epsilon > 0$.
Therefore, we can approximate a stationary point of
$f$
over
$\St(p,N)$
by solving Problem~\ref{problem:CP_grad} with a sufficiently small
$\epsilon > 0$.
\begin{lemma}[First-order optimality condition]\label{lemma:optimality}
  Let
  $f:\mathbb{R}^{N\times p} \to \mathbb{R}$
  be differentiable.
  Let
  $\bm{U} \in \St(p,N)$
  and
  $\bm{S} \in {\rm O}(N)$
  satisfy
  $\bm{U} \in \St(p,N) \setminus E_{N,p}(\bm{S})$.
  Then, the first-order optimality condition in~\eqref{eq:optimality} can be stated equivalently as
  \begin{equation}\label{eq:optimality_CP_lemma}
    \nabla f_{\bm{S}}(\Phi_{\bm{S}}(\bm{U})) = \bm{0},
  \end{equation}
  where
  $f_{\bm{S}} := f\circ\Phi_{\bm{S}}^{-1}$.
\end{lemma}
\begin{proof}
  Let
  $\bm{U}_{\perp}\in\St(N-p,N)$
  satisfy
  $\bm{U}^{\T}\bm{U}_{\perp}= \bm{0}$.
  Then, we have
  $\bm{U} = \Phi_{[\bm{U}\ \bm{U}_{\perp}]}^{-1}(\bm{0})$.
  For
  $\bm{S}\in{\rm O}(N)$
  satisfying
  $\bm{U}\in \St(p,N)\setminus E_{N,p}(\bm{S})$
  and
  $\bm{V}:=\Phi_{\bm{S}}(\bm{U}) \in Q_{N,p}(\bm{S})$,
  i.e.,
  $\bm{U}=\Phi_{\bm{S}}^{-1}(\bm{V})$,
  Proposition~\ref{proposition:change_center}~\ref{enum:zero} asserts that
  $\nabla f_{\bm{S}}(\bm{V}) = \bm{0}$
  if and only if
  $\nabla f_{[\bm{U}\ \bm{U}_{\perp}]}(\bm{0}) = \bm{0}$.
  To prove the equivalence between~\eqref{eq:optimality} and~\eqref{eq:optimality_CP_lemma}, it is sufficient to show the equivalence between the condition in~\eqref{eq:optimality} and
  $\nabla f_{[\bm{U}\ \bm{U}_{\perp}]}(\bm{0}) = \bm{0}$.
  By~\eqref{eq:gradient_0}, we have
  \begin{equation}
    \nabla f_{[\bm{U}\ \bm{U}_{\perp}]}(\bm{0}) = \begin{bmatrix} \nabla f(\bm{U})^{\T}\bm{U}-\bm{U}^{\T}\nabla f(\bm{U}) & \nabla f(\bm{U})^{\T}\bm{U}_{\perp} \\ -\bm{U}_{\perp}^{\T}\nabla f(\bm{U}) & \bm{0} \end{bmatrix},
  \end{equation}
  which yields
  $\dbra{\nabla f_{[\bm{U}\ \bm{U}_{\perp}]}(\bm{0})}_{11} = \bm{0}$
  if and only if the second condition in~\eqref{eq:optimality} holds true.

  In the following, we show the equivalence of
  $\bm{U}_{\perp}^{\T}\nabla f(\bm{U}) = \bm{0}$
  and
  $(\bm{I}-\bm{U}\bm{U}^{\T})\nabla f(\bm{U}) = \bm{0}$.
  By noting
  $\begin{bmatrix} \bm{U} & \bm{U}_{\perp} \end{bmatrix}\begin{bmatrix} \bm{U} & \bm{U}_{\perp} \end{bmatrix}^{\T} = \bm{U}\bm{U}^{\T}+\bm{U}_{\perp}\bm{U}_{\perp}^{\T} = \bm{I}$,
  the equality
  $\bm{U}_{\perp}^{\T}\nabla f(\bm{U})  = \bm{0}$
  implies
  $\bm{0}=\bm{U}_{\perp}\bm{U}_{\perp}^{\T}\nabla f(\bm{U}) = (\bm{I}-\bm{U}\bm{U}^{\T})\nabla f(\bm{U})$.
  Conversely,
  $(\bm{I}-\bm{U}\bm{U}^{\T})\nabla f(\bm{U}) = \bm{0}$
  implies
  $\bm{0}= \bm{U}_{\perp}^{\T}(\bm{I}-\bm{U}\bm{U}^{\T})\nabla f(\bm{U}) =  \bm{U}_{\perp}^{\T}\nabla f(\bm{U})$.
\end{proof}

\begin{lemma} \label{lemma:gap_grad}
  Let
  $f:\mathbb{R}^{N\times p} \to \mathbb{R}$
  be continuously differentiable with
  $p<N$
  and
  $\bm{S} \in {\rm O}(N)$.
  Then,
  $\inf_{\bm{V}\in Q_{N,p}(\bm{S})} \|\nabla (f\circ\Phi_{\bm{S}}^{-1})(\bm{V})\|_{F} = 0$.
\end{lemma}
\begin{proof}
  Let
  $\bm{U}^{\star} \in \St(p,N)$
  be a global minimizer of
  $f$
  over
  $\St(p,N)$,
  and
  $\bm{S}^{\star} \in {\rm O}(N)$
  satisfy
  $\bm{U}^{\star} \in \St(p,N) \setminus E_{N,p}(\bm{S}^{\star})$.
  Then,
  $\bm{U}^{\star}$
  is a stationary point of
  $f$
  over
  $\St(p,N)$,
  and we have
  $\|\nabla (f\circ\Phi_{\bm{S}^{\star}}^{-1})(\bm{V}^{\star})\|_{F} = 0$
  with
  $\bm{V}^{\star} := \Phi_{\bm{S}^{\star}}(\bm{U}^{\star}) \in Q_{N,p}(\bm{S}^{\star})$
  from Lemma~\ref{lemma:optimality}.

  Theorem~\ref{theorem:dense}~\ref{enum:intersection_dense} ensures the denseness of
  $\Delta(\bm{S},\bm{S}^{\star}):=(\St(p,N) \setminus E_{N,p}(\bm{S})) \cap (\St(p,N)\setminus E_{N,p}(\bm{S}^{\star}))$
  in
  $\St(p,N)$.
  Then, we obtain a sequence
  $(\bm{U}_{n})_{n=0}^{\infty}$
  of
  $\Delta(\bm{S},\bm{S}^{\star})$
  converging to
  $\bm{U}^{\star}$.
  Let
  $(\bm{V}_{n}^{\star})_{n=0}^{\infty}$
  and
  $(\bm{V}_{n})_{n=0}^{\infty}$
  be sequences of
  $\bm{V}_{n}^{\star}:= \Phi_{\bm{S}^{\star}}(\bm{U}_{n}) \in Q_{N,p}(\bm{S}^{\star})$
  and
  $\bm{V}_{n}:= \Phi_{\bm{S}}(\bm{U}_{n}) \in Q_{N,p}(\bm{S})$.
  The continuity of
  $\Phi_{\bm{S}^{\star}}$
  yields
  $\lim_{n\to\infty}\bm{V}_{n}^{\star} = \bm{V}^{\star}$,
  implying the boundedness of
  $(\bm{V}_{n}^{\star})_{n=0}^{\infty}$.
  From
  $\Phi_{\bm{S}}^{-1}(\bm{V}_{n}) = \bm{U}_{n} = \Phi_{\bm{S}^{\star}}^{-1}(\bm{V}_{n}^{\star})$
  and Proposition~\ref{proposition:change_center}~\ref{enum:change_center}, we have
  $0 \leq \|\nabla (f\circ\Phi_{\bm{S}}^{-1})(\bm{V}_{n})\|_{F} \leq 2(1+\|\bm{V}_{n}^{\star}\|_{2}^{2})\|\nabla (f\circ\Phi_{\bm{S}^{\star}}^{-1})(\bm{V}_{n}^{\star})\|_{F}$.
  The right-hand side of the above inequality converges to zero from the boundedness of
  $(\bm{V}_{n}^{\star})_{n=0}^{\infty}$
  and
  $\|\nabla (f\circ\Phi_{\bm{S}^{\star}}^{-1})(\bm{V}^{\star})\|_{F} = 0$.
  Therefore, we have
  $\lim_{n\to \infty}\|\nabla (f\circ\Phi_{\bm{S}}^{-1})(\bm{V}_{n})\|_{F} = 0$,
  from which we completed the proof.
\end{proof}
\subsection{Basic framework to incorporate optimization techniques designed over a vector space with the Cayley parametrization}\label{sec:CP}

\begin{algorithm}[t] 
  \caption{Cayley parametrization strategy (Algorithm~\ref{alg:proposed}+$\mathcal{A}$)}
\begin{algorithmic}
  \label{alg:proposed}
\REQUIRE
$\bm{U}_0 \in \St(p,N)$,
$\bm{S} \in {\rm O}(N)$,
$\mathcal{A}:Q_{N,p}(\bm{S})\to Q_{N,p}(\bm{S})$: update rule
\STATE
  $\bm{V}_0 =\Phi_{\bm{S}}(\bm{U}_0)$
\FOR{$n=0,1,2,\ldots,m-1$}
  \STATE
  $\bm{V}_{n+1} = \mathcal{A}(\bm{V}_{n})$
  \STATE
  $\bm{U}_{n+1} = \Phi_{\bm{S}}^{-1}(\bm{V}_{n+1})$
\ENDFOR
\ENSURE
$\bm{U}_{m}$
\end{algorithmic}
\end{algorithm}
We illustrate a general scheme of the Cayley parametrization strategy in Algorithm~\ref{alg:proposed}\footnote{
  Algorithm~\ref{alg:proposed} can serve as a central building block in our further advanced Cayley parametrization strategies, reported partially in~\cite{Kume-Yamada19,Kume-Yamada20,Kume-Yamada21}.
},
where
$\bm{U}_{0} \in \St(p,N)$
is an initial estimate for a solution to Problem~\ref{problem:origin} with
$p < N$,
$\bm{S} \in {\rm O}(N)$
is a center point for parametrization of a dense subset
$\St(p,N) \setminus E_{N,p}(\bm{S}) \subset \St(p,N)$
in terms of the vector space
$Q_{N,p}(\bm{S})$,
and a mapping
$\mathcal{A}:Q_{N,p}(\bm{S}) \to Q_{N,p}(\bm{S})$
is a certain update rule for decreasing the value of
$f\circ\Phi_{\bm{S}}^{-1}$.
In principle, we can employ any optimization update scheme over a vector space as
$\mathcal{A}$,
which is a notable advantage of the proposed strategy over the standard strategy (see Remark~\ref{remark:relation_to_others}).
As a simplest example, we will employ, in Section~\ref{sec:numerical}, a gradient descent-type update scheme
$\mathcal{A}^{\rm GDM}:Q_{N,p}(\bm{S})\to Q_{N,p}(\bm{S}):\bm{V} \mapsto \bm{V} - \gamma \nabla (f\circ \Phi_{\bm{S}}^{-1})(\bm{V})$
with a stepsize
$\gamma > 0$
determined by a certain line-search algorithm (see, e.g.,~\cite{numerical_optimization}).

To parameterize
$\bm{U}_{0} \in \St(p,N)$
by
$\Phi_{\bm{S}}^{-1}$,
$\bm{S} \in {\rm O}(N)$
must be chosen to satisfy
$\bm{U}_{0} \in \St(p,N) \setminus E_{N,p}(\bm{S})$.
An example of selection of such
$\bm{S}$
for a given
$\bm{U}_{0}$
is
$\bm{S}:=\diag(\bm{Q}_{1}\bm{Q}_{2}^{\T}, \bm{I}_{N-p}) \in {\rm O}_{p}(N)$
by using a singular value decomposition
$[\bm{U}_{0}]_{\rm up} = \bm{Q}_{1}\bm{\Sigma}\bm{Q}_{2}^{\T} \in \mathbb{R}^{p\times p}$
with
$\bm{Q}_{1}, \bm{Q}_{2} \in {\rm O}(p)$
and a diagonal matrix
$\bm{\Sigma} \in \mathbb{R}^{p\times p}$
with non-negative entries (see Theorem~\ref{theorem:construct_center}).

\begin{remark}[Comparison to the retraction-based strategy]\label{remark:relation_to_others}
  As reported in~\cite{manifold_book,E98,N02,N05,A07,T08,A12,R12,W13,H15,J15,M15,S15,Z17,K18,Zhu-Sato20,Sato21}, Problem~\ref{problem:origin} has been tackled with a retraction
  $R:T\St(p,N):=\{\{\bm{U}\}\times T_{\bm{U}}\St(p,N) \mid \bm{U} \in \St(p,N)\} \to \St(p,N):(\bm{U}, \bm{\mathcal{V}})\mapsto R_{\bm{U}}(\bm{\mathcal{V}})$ (see, e.g.,~\cite{manifold_book})
  by exploiting only a local diffeomorphism\footnote{
    The local diffeomorphism of
    $R_{\bm{U}}$
    around
    $\bm{0} \in T_{\bm{U}}\St(p,N)$
    can be verified with the inverse function theorem and the condition (ii) in Definition~\ref{definition:retraction}.
  }
  of each
  $R_{\bm{U}}$
  between a sufficiently small neighborhood of
  $\bm{0} \in T_{\bm{U}}\St(p,N)$
  in the tangent space
  $T_{\bm{U}}\St(p,N)$,
  at
  $\bm{U} \in \St(p,N)$
  to
  $\St(p,N)$,
  and its image in
  $\St(p,N)$
  (see Appendix~\ref{appendix:retraction} for its basic idea).
  At the $n$th iteration, these retraction-based strategies decrease the time-varying function
  $f\circ R_{\bm{U}_{n}}$
  at
  $\bm{0} \in T_{\bm{U}_{n}}\St(p,N)$
  over the time-varying vector space
  $T_{\bm{U}_{n}}\St(p,N)$,
  where
  $\bm{U}_{n} \in \St(p,N)$
  is the $n$th estimate for a solution.
  Many computational mechanisms for finding a descent direction
  $\bm{\mathcal{D}}_{n} \in T_{\bm{U}_{n}}\St(p,N)$
  in the tangent space
  $T_{\bm{U}_{n}}\St(p,N)$
  have been motivated by standard ideas for optimization over a fixed vector space.
  To achieve fast convergence in optimization over a vector space, many researchers have been trying to utilize the past updating directions
  for estimating a current descent direction, e.g., the conjugate gradient method, quasi-Newton's method and Nesterov accelerated gradient method~\cite{numerical_optimization,N83,G16,Reddi-Hefny-Sra16}.
  However, in the retraction-based strategy, since the past updating directions
  $(\bm{\mathcal{D}}_{k})_{k=0}^{n-1}$
  no longer live in the current tangent space
  $T_{\bm{U}_{n}}\St(p,N)$,
  we can not utilize directly
  $(\bm{\mathcal{D}}_{k})_{k=0}^{n-1}$
  for estimating a new descent direction
  $\bm{\mathcal{D}}_{n} \in T_{\bm{U}_{n}}\St(p,N)$.
  To be exploited the past updating directions with a retraction, those directions must be translated into the current tangent space with certain mappings, e.g., {\it a vector transport}~\cite{manifold_book} and the inversion mapping of retractions~\cite{Zhu-Sato20}.

  On the other hand, Algorithm~\ref{alg:proposed} decreases the fixed cost function
  $f\circ \Phi_{\bm{S}}^{-1}$
  with a fixed
  $\bm{S} \in {\rm O}(N)$
  over the fixed vector space
  $Q_{N,p}(\bm{S})$
  during the process of Algorithm~\ref{alg:proposed} by exploiting the diffeomorphism of
  $\Phi_{\bm{S}}^{-1}$
  between
  $Q_{N,p}(\bm{S})$
  and an open dense subset
  $\St(p,N)\setminus E_{N,p}(\bm{S})$
  of
  $\St(p,N)$
  (see Proposition~\ref{proposition:inverse} and Theorem~\ref{theorem:dense}~\ref{enum:dense}).
  Since every past updating direction lives in the same vector space
  $Q_{N,p}(\bm{S})$,
  we can utilize the past updating directions without requiring any additional computation such as a vector transport and the inversion mapping of retractions.
  Therefore, we can transplant powerful computational arts, e.g.,~\cite{numerical_optimization,Reddi-Hefny-Sra16,G16,Zeyuan18,Ward-Wu-Bottou,Chen-Liu-Sun-Hong19,Tatarenko-Touri17}, designed for optimization over a vector space, into the proposed strategy.
  For many such algorithms, Proposition~\ref{proposition:gradient_property} must be useful for checking whether conditions, regarding the cost function, for a global convergence of optimization techniques hold true or not.
\end{remark}

\subsection{Singular-point issue in the Cayley parametrization strategy}\label{sec:mobility}
Numerical performance of Algorithm~\ref{alg:proposed} heavily depends on tuning
$\bm{S} \in {\rm O}(N)$
in general.
If we choose
$\bm{S}$
such that a minimizer
$\bm{U}^{\star} \in \St(p,N)$
of Problem~\ref{problem:origin} is close to the singular-point set
$E_{N,p}(\bm{S})$,
then a risk of a slow convergence of Algorithm~\ref{alg:proposed} arises due to an insensitivity of
$\Phi_{\bm{S}}^{-1}$
to the change
around
$\Phi_{\bm{S}}(\bm{U}^{\star})$
in the vector space
$Q_{N,p}(\bm{S})$.
In a case where
$p = N$,
this risk has been reported by~\cite{Y03,H10}.
We can see this insensitivity of
$\Phi_{\bm{S}}^{-1}$
via Proposition~\ref{proposition:mobility} below.

\begin{proposition}[The mobility of $\Phi_{\bm{S}}^{-1}$] \label{proposition:mobility}
  Let
  $p,N\in \mathbb{N}$
  satisfy
  $p < N$,
  $\bm{S} \in {\rm O}(N)$,
  $\bm{V} \in Q_{N,p}(\bm{S})$,
  and
  $\bm{\mathcal{E}}\in Q_{N,p}(\bm{S})$
  satisfy
  $\|\bm{\mathcal{E}}\|_{F} = 1$.
  Then, we have
  \begin{equation}\label{eq:mobility}
    \|\Phi_{\bm{S}}^{-1}(\bm{V}+\tau \bm{\mathcal{E}}) - \Phi_{\bm{S}}^{-1}(\bm{V})\|_{F} \leq \tau r(\bm{V}),
  \end{equation}
  where
  \begin{align}
    r(\bm{V}):=
    \frac{2\sqrt{1+\|\dbra{\bm{V}}_{21}\|_{2}^{2}}}{1+\sigma_{\min}^{2}(\dbra{\bm{V}}_{21})}.\label{eq:mobility_rate}
  \end{align}
  We call
  $r:Q_{N,p}(\bm{S})\to \mathbb{R}$
  the mobility of
  $\Phi_{\bm{S}}^{-1}$,
  which is bounded as
  \begin{equation}
      r(\bm{V}) \geq 2(1+\|\dbra{\bm{V}}_{21}\|_{2}^{2})^{-1/2}, \label{eq:bound_ratio}
  \end{equation}
  where the equality holds when
  $\sigma_{\min}(\dbra{\bm{V}}_{21}) = \sigma_{\max}(\dbra{\bm{V}}_{21})(=\|\dbra{\bm{V}}_{21}\|_{2})$.
\end{proposition}
\begin{proof}
  See Appendix~\ref{appendix:mobility}.
\end{proof}

To interpret the result in Proposition~\ref{proposition:mobility}, we consider two simple examples.
Under the condition
$\sigma_{\min}(\dbra{\bm{V}}_{21}) = \sigma_{\max}(\dbra{\bm{V}}_{21})(=\|\dbra{\bm{V}}_{21}\|_{2})$,
we observe from~\eqref{eq:bound_ratio} that the mobility
$r(\bm{V})$
becomes small when
$\|\dbra{\bm{V}}_{21}\|_{2}$
increases.
On the other hand, because
$r(\bm{V}) = 2$
is achieved by
$\|\dbra{\bm{V}}_{21}\|_{2} = 0$
from~\eqref{eq:mobility_rate},
$\dbra{\bm{V}}_{21}$
around zero does not lead small
$r(\bm{V})$.

These tendencies can be observed numerically in Figure~\ref{fig:mobility}, where the plot shows the norm
$\|\dbra{\bm{V}}_{21}\|_{2}$
on the horizontal axis versus the values
$\|\Phi_{\bm{S}}^{-1}(\bm{V}+\bm{\mathcal{E}}) - \Phi_{\bm{S}}^{-1}(\bm{V})\|_{F}$
and
$r(\bm{V})$,
with randomly chosen
$\bm{V},\bm{\mathcal{E}} \in Q_{N,p}(\bm{S})$
satisfying
$\|\bm{\mathcal{E}}\|_{F} = 1$,
on the vertical axis
for each $N\in \{500,1000,2000\}$
and
$p=10$.
From this figure, we observe that the mobility
$r(\bm{V})$
decreases and
$\Phi_{\bm{S}}^{-1}$
becomes insensitive as
$\|\dbra{\bm{V}}_{21}\|_{2}$
increases.

This insensitivity of
$\Phi_{\bm{S}}^{-1}$,
at distant points from zero, causes a risk of slow convergence of Algorithm~\ref{alg:proposed} even if the current estimate
$\bm{V}_{n} \in Q_{N,p}(\bm{S})$
is not sufficiently close to a solution
$\bm{V}^{\star}\in Q_{N,p}(\bm{S})$
of Problem~\ref{problem:CP_St} or Problem~\ref{problem:CP_grad}.
Since Theorem~\ref{theorem:dense}~\ref{enum:characterize_singular} implies that
$\|\Phi_{\bm{S}}(\bm{U})\|_{2}$
increases as
$\bm{U}\in \St(p,N) \setminus E_{N,p}(\bm{S})$
approaches
$E_{N,p}(\bm{S})$,
the risk of the slow convergence, say {\it a singular-point issue}, can arise in a case where a global minimizer
$\bm{U}^{\star} \in \St(p,N)$
stays around
$E_{N,p}(\bm{S})$.
In Section~\ref{sec:numerical_singular}, we will see that the numerical performance of Algorithm~\ref{alg:proposed} employing the gradient descent-type  method tends to deteriorate as
$\bm{U}^{\star}$
approaches
$E_{N,p}(\bm{S})$.

To remedy the singular-point issue in Algorithm~\ref{alg:proposed}, it is recommendable to use
$\bm{S}$
such that
$\Phi_{\bm{S}}(\bm{U}^{\star})$
is close to zero in
$Q_{N,p}(\bm{S})$.
Although we can not determine for a given
$\bm{S}$
whether
$\Phi_{\bm{S}}(\bm{U}^{\star})$
is close to zero or not
in advance of minimization for general
$f$,
Theorem~\ref{theorem:range_minimizer} guarantees,
under the right orthogonal invariance of
$f$, 
the existence of a global minimizer
$\bm{U}^{\star}$
satisfying
$\|\dbra{\Phi_{\bm{S}}(\bm{U}^{\star})}_{21}\|_{2} \leq 1$
for every
$\bm{S} \in {\rm O}(N)$.
In this case, by
$r(\Phi_{\bm{S}}(\bm{U}^{\star})) \geq \sqrt{2}$
in~\eqref{eq:bound_ratio}
and the continuity of
$r$,
the mobility
$r$
of
$\Phi_{\bm{S}}^{-1}$
can be maintained in a neighborhood of
$\Phi_{\bm{S}}(\bm{U}^{\star})$
to which a point sequence
$(\bm{V}_{n})_{n=0}^{\infty}$
generated by Algorithm~\ref{alg:proposed} is desired to approach.
Therefore, we do not need to be nervous about the influence by the singular-point set around
$\Phi_{\bm{S}}(\bm{U}^{\star})$.

For general
$f$,
to remedy the singular-point issue, we reported shortly in~\cite{Kume-Yamada19,Kume-Yamada20} that this issue can be avoided by a Cayley parametrization-type strategy, for Problem~\ref{problem:ALCP} below, by updating not only
$\bm{V}_{n} \in Q_{N,p}$
but also a preferable center point
$\bm{S}_{n} \in {\rm O}(N)$
strategically.
Due to the space consuming discussion, we will present its fully detailed discussion in another occasion.
\begin{problem} \label{problem:ALCP}
  For a given continuous function
  $f:\mathbb{R}^{N\times p}\to \mathbb{R}$,
  choose
  $\epsilon > 0$
  arbitrarily.
  Then,
  \begin{align}\label{eq:problem_alcp}
    \textrm{find} \ (\bm{V}^{\star},\bm{S}^{\star}) \in  Q_{N,p} \times {\rm O}(N) \ \textrm{such that} \  f\circ\Phi_{\bm{S}^{\star}}^{-1}(\bm{V}^{\star}) < \min f(\St(p,N)) + \epsilon.
  \end{align}
\end{problem}

\begin{figure}[t]
  \begin{center}
    \includegraphics[width=0.8\linewidth]{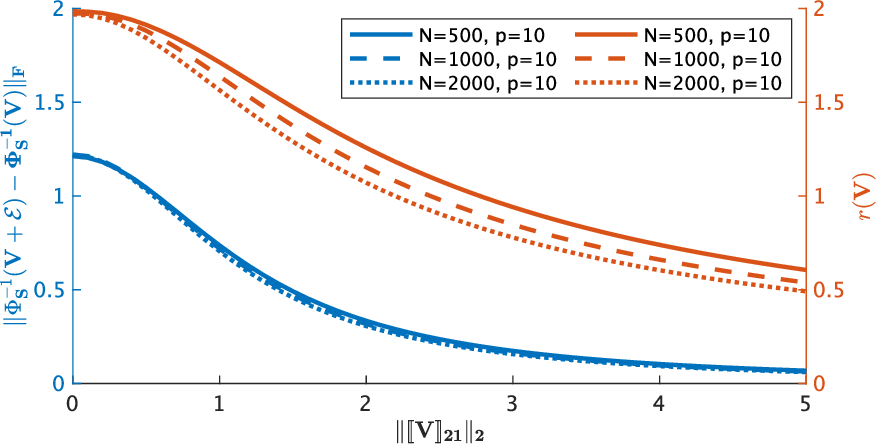}
    \caption{The average values of the change $\|\Phi_{\bm{S}}^{-1}(\bm{V}+\bm{\mathcal{E}})-\Phi_{\bm{S}}^{-1}(\bm{V})\|_{F}$ and the mobility $r(\bm{V})$ for each $\|\dbra{\bm{V}}_{21}\|_{2}$ over $10$ trials in the case
    $N = \{500, 1000, 2000\}$
    and
    $p = 10$.
    In each trial, we generate
    $\bm{\widetilde{V}},\bm{\widetilde{\mathcal{E}}}\in \mathbb{R}^{N\times N}$
    of which each entry is uniformly chosen from
    $[-0.5,0.5]$
    except for the $(N-p)$-by-$(N-p)$ right lower block matrix.
    Then, with
    $\bm{\mathcal{E}}:=\Skew(\bm{\widetilde{\mathcal{E}}})/\|\Skew(\bm{\widetilde{\mathcal{E}}})\|_{F} \in Q_{N,p}$
    satisfying
    $\|\bm{\mathcal{E}}\|_{F} = 1$,
    we evaluate
    $\|\Phi_{\bm{S}}^{-1}(\bm{V}+\bm{\mathcal{E}})-\Phi_{\bm{S}}^{-1}(\bm{V})\|_{F}$
    and
    $r(\bm{V})$
    at
    $\bm{V} \in Q_{N,p}$
    with
    $\dbra{\bm{V}}_{11} = \dbra{\Skew(\bm{\widetilde{V}})}_{11}$
    and
    $\dbra{\bm{V}}_{21} = c\dbra{\Skew(\bm{\widetilde{V}})}_{21}$
    by changing
    $c \in [0, 5/\|\dbra{\Skew(\bm{\widetilde{V}})}_{21}\|_{2}]$.
  }
    \label{fig:mobility}
  \end{center}
\end{figure}

\subsection[Relation between the Cayley transform-based retraction and the inversion of G-L$^2$CT]{Relation between the Cayley transform-based retraction and $\Phi_{\bm{S}}^{-1}$} \label{sec:Cayley_retraction}
The proposed
$\Phi_{\bm{S}}^{-1}$
can be regarded as another form of {\it the Cayley transform-based retraction} for
$\St(p,N)$.
By using the inversion
$\varphi^{-1}$
in~\eqref{eq:inv_origin_Cayley}, the Cayley transform-based retraction
$R^{\rm Cay}:T\St(p,N) \to \St(p,N):(\bm{U}, \bm{\mathcal{V}})\mapsto R_{\bm{U}}^{\rm Cay}(\bm{\mathcal{V}})$
was introduced explicitly in~\cite{W13,Z17}, where the tangent bundle
$T\St(p,N)=\{\{\bm{U}\}\times T_{\bm{U}}\St(p,N)\mid \bm{U} \in \St(p,N)\}$
is defined with the tangent space
$T_{\bm{U}}\St(p,N)$
to
$\St(p,N)$
at
$\bm{U} \in \St(p,N)$
(see Fact~\ref{fact:stiefel}~(d)).
For
$\bm{U} \in \St(p,N)$,
$R_{\bm{U}}^{\rm Cay}$
can be expressed with
$\bm{P}_{\bm{U}}:=\bm{I}-\bm{U}\bm{U}^{\T}/2 \in \mathbb{R}^{N\times N}$
as
\begin{equation} \label{eq:Cayley_retraction}
  R_{\bm{U}}^{\rm Cay}:T_{\bm{U}}\St(p,N)\to \St(p,N):\bm{\mathcal{V}}\mapsto \varphi^{-1}(\Skew(\bm{U}\bm{\mathcal{V}}^{\T}\bm{P}_{\bm{U}}))\bm{U}.
\end{equation}
By passing through the linear mapping
\begin{equation}
  \Psi_{[\bm{U}\ \bm{U}_{\perp}]}:T_{\bm{U}}\St(p,N) \to Q_{N,p}([\bm{U}\ \bm{U}_{\perp}]):\bm{\mathcal{V}}\mapsto
        -\frac{1}{2}\begin{bmatrix} \bm{U}^{\T}\bm{\mathcal{V}} & - (\bm{U}_{\perp}^{\T}\bm{\mathcal{V}})^{\T} \\ \bm{U}_{\perp}^{\T}\bm{\mathcal{V}} & \bm{0}\end{bmatrix} \label{eq:tangent_to_Q},
\end{equation}
with
$\bm{U}_{\perp} \in \St(N-p,N)$
satisfying
$\bm{U}^{\T}\bm{U}_{\perp} = \bm{0}$,
we have the following relation
\begin{equation}
  (\bm{\mathcal{V}} \in T_{\bm{U}}\St(p,N)) \quad \Phi_{[\bm{U}\ \bm{U}_{\perp}]}^{-1} \circ \Psi_{[\bm{U}\ \bm{U}_{\perp}]}(\bm{\mathcal{V}}) = R_{\bm{U}}^{\rm Cay}(\bm{\mathcal{V}}). \label{eq:Phi_retraction}
\end{equation}
This relation can be verified specially with
$\bm{S}:= [\bm{U}\ \bm{U}_{\perp}] \in {\rm O}(N)$
by
\begin{align}
  & (\bm{V}\in Q_{N,p}(\bm{S})) \quad
  \Phi_{\bm{S}}^{-1}(\bm{V})
  = \bm{S}(\bm{I}-\bm{V})(\bm{I}+\bm{V})^{-1}\bm{I}_{N\times p} \\
  & = (\bm{I}-\bm{S}\bm{V}\bm{S}^{-1})(\bm{I}+\bm{S}\bm{V}\bm{S}^{-1})^{-1}\bm{S}\bm{I}_{N\times p}
  = (\bm{I}-\bm{S}\bm{V}\bm{S}^{\T})(\bm{I}+\bm{S}\bm{V}\bm{S}^{\T})^{-1}\bm{U}
  = \varphi^{-1}(\bm{S}\bm{V}\bm{S}^{\T})\bm{U}
\end{align}
and
\begin{align}
  & (\bm{\mathcal{V}} \in T_{\bm{U}}\St(p,N)) \quad
  \bm{S}\Psi_{\bm{S}}(\bm{\mathcal{V}})\bm{S}^{\T}
  = -\frac{1}{2}
  \begin{bmatrix} \bm{U} & \bm{U}_{\perp} \end{bmatrix}
  \begin{bmatrix} \bm{U}^{\T}\bm{\mathcal{V}} & -\bm{\mathcal{V}}^{\T}\bm{U}_{\perp} \\ \bm{U}_{\perp}^{\T}\bm{\mathcal{V}} & \bm{0} \end{bmatrix}
  \begin{bmatrix} \bm{U}^{\T} \\ \bm{U}_{\perp}^{\T} \end{bmatrix} \\
  & = -\frac{1}{2}(\bm{U}\bm{U}^{\T}\bm{\mathcal{V}}\bm{U}^{\T} + \bm{U}_{\perp}\bm{U}_{\perp}^{\T}\bm{\mathcal{V}}\bm{U}^{\T} -\bm{U}\bm{\mathcal{V}}^{\T}\bm{U}_{\perp}\bm{U}_{\perp}^{\T})  \\
  & = -\frac{1}{2}\left(\bm{U}\bm{U}^{\T}\bm{\mathcal{V}}\bm{U}^{\T} + (\bm{I} - \bm{U}\bm{U}^{\T})\bm{\mathcal{V}}\bm{U}^{\T} -\bm{U}\bm{\mathcal{V}}^{\T}(\bm{I} - \bm{U}\bm{U}^{\T})\right) \quad (\because \bm{U}\bm{U}^{\T} + \bm{U}_{\perp}\bm{U}_{\perp}^{\T} = \bm{I}) \\
  & = \frac{1}{2}(\bm{U}\bm{\mathcal{V}}^{\T}-\bm{\mathcal{V}}\bm{U}^{\T} -  \bm{U}\bm{\mathcal{V}}^{\T}\bm{U}\bm{U}^{\T}) \\
  & = \frac{1}{2}\left(\bm{U}\bm{\mathcal{V}}^{\T}-\bm{\mathcal{V}}\bm{U}^{\T} -  \frac{1}{2}\bm{U}(\bm{\mathcal{V}}^{\T}\bm{U}-\bm{U}^{\T}\bm{\mathcal{V}})\bm{U}^{\T}\right) \quad (\because \bm{\mathcal{V}}\in T_{\bm{U}}\St(p,N) \Leftrightarrow \bm{U}^{\T}\bm{\mathcal{V}} + \bm{\mathcal{V}}^{\T}\bm{U} = \bm{0}) \\
  & = \Skew\left(\bm{U}\bm{\mathcal{V}}^{\T} - \frac{1}{2}\bm{U}\bm{\mathcal{V}}^{\T}\bm{U}\bm{U}^{\T}\right)
  = \Skew\left(\bm{U}\bm{\mathcal{V}}^{\T}\left(\bm{I}-\frac{1}{2}\bm{U}\bm{U}^{\T}\right)\right) 
  = \Skew(\bm{U}\bm{\mathcal{V}}^{\T}\bm{P}_{\bm{U}}).
\end{align}

Through the relation in~\eqref{eq:Phi_retraction}, we obtain a diffeomorphic property of
$R_{\bm{U}}^{\rm Cay}$
in the following.
The linear mapping
$\Psi_{\bm{S}}$
is a bijection between
$T_{\bm{U}}\St(p,N)$
and
$Q_{N,p}(\bm{S})$
with its inversion mapping
$\Psi_{\bm{S}}^{-1}:Q_{N,p}(\bm{S})\to T_{\bm{U}}\St(p,N):\bm{V}\mapsto -2\bm{S}\bm{V}\bm{I}_{N\times p}$.
From
$\Psi_{\bm{S}}(T_{\bm{U}}\St(p,N)) = Q_{N,p}(\bm{S})$,~\eqref{eq:Phi_retraction} and Proposition~\ref{proposition:inverse},
$R_{\bm{U}}^{\rm Cay}$
is a diffeomorphic between
$T_{\bm{U}}\St(p,N)$
and a subset
$\St(p,N)\setminus E_{N,p}(\bm{S})$
of
$\St(p,N)$.
Clearly, the inversion mapping of
$R_{\bm{U}}^{\rm Cay}$
is given by
$R_{\bm{U}}^{{\rm Cay}^{-1}}:\St(p,N)\setminus E_{N,p}(\bm{S}) \to T_{\bm{U}}\St(p,N):\bm{\mathfrak{U}}\mapsto \Psi_{\bm{S}}^{-1} \circ \Phi_{\bm{S}}(\bm{\mathfrak{U}})$.

We present an explicit formula for
$R_{\bm{U}}^{{\rm Cay}^{-1}}$.
From Definition~\ref{definition:Cayley}, we have
\begin{align}
  & (\bm{\mathfrak{U}}\in \St(p,N)\setminus E_{N,p}(\bm{S})) \quad
  R_{\bm{U}}^{{\rm Cay}^{-1}}(\bm{\mathfrak{U}})
  = -2\bm{S}\begin{bmatrix} \bm{A}_{\bm{S}}(\bm{\mathfrak{U}}) & -\bm{B}_{\bm{S}}(\bm{\mathfrak{U}}) \\
  \bm{B}_{\bm{S}}(\bm{\mathfrak{U}}) & \bm{0} \end{bmatrix} \bm{I}_{N\times p} \\
  & = -2\begin{bmatrix}\bm{U} & \bm{U}_{\perp} \end{bmatrix}
  \begin{bmatrix} \bm{A}_{\bm{S}}(\bm{\mathfrak{U}}) \\ \bm{B}_{\bm{S}}(\bm{\mathfrak{U}}) \end{bmatrix}
  = -2\bm{U}\bm{A}_{\bm{S}}(\bm{\mathfrak{U}}) -2\bm{U}_{\perp}\bm{B}_{\bm{S}}(\bm{\mathfrak{U}}).\label{eq:Cayley_retraction_inv_calc}
\end{align}
From~\eqref{eq:Cay_A} and~\eqref{eq:Cay_B}, each term in~\eqref{eq:Cayley_retraction_inv_calc} is evaluated as
\begin{align}
  -2\bm{U}\bm{A}_{\bm{S}}(\bm{\mathfrak{U}})
  & = -4\bm{U}(\bm{I}_{p}+\bm{\mathfrak{U}}^{\T}\bm{U})^{-1}\Skew(\bm{\mathfrak{U}}^{\T}\bm{U})(\bm{I}_{p}+\bm{U}^{\T}\bm{\mathfrak{U}})^{-1} \\
  & = 2\bm{U}(\bm{I}_{p}+\bm{\mathfrak{U}}^{\T}\bm{U})^{-1}\left((\bm{I}_{p}+\bm{U}^{\T}\bm{\mathfrak{U}}) - (\bm{I}_{p}+\bm{\mathfrak{U}}^{\T}\bm{U})\right)(\bm{I}_{p}+\bm{U}^{\T}\bm{\mathfrak{U}})^{-1} \\
  & = 2\bm{U}(\bm{I}_{p}+\bm{\mathfrak{U}}^{\T}\bm{U})^{-1} - 2\bm{U}(\bm{I}_{p} + \bm{U}^{\T}\bm{\mathfrak{U}})^{-1}, \\
  -2\bm{U}_{\perp}\bm{B}_{\bm{S}}(\bm{\mathfrak{U}})
  & = 2\bm{U}_{\perp}\bm{U}_{\perp}^{\T}\bm{\mathfrak{U}}(\bm{I}_{p} + \bm{U}^{\T}\bm{\mathfrak{U}})^{-1}
  = 2(\bm{I}-\bm{U}\bm{U}^{\T})\bm{\mathfrak{U}}(\bm{I}_{p} + \bm{U}^{\T}\bm{\mathfrak{U}})^{-1}.
\end{align}
By substituting these equalities into~\eqref{eq:Cayley_retraction_inv_calc}, we have
\begin{align}
  & R_{\bm{U}}^{{\rm Cay}^{-1}}(\bm{\mathfrak{U}})
  = 2\bm{U}(\bm{I}_{p}+\bm{\mathfrak{U}}^{\T}\bm{U})^{-1} - 2\bm{U}(\bm{I}_{p} + \bm{U}^{\T}\bm{\mathfrak{U}})^{-1}
  +2(\bm{I}-\bm{U}\bm{U}^{\T})\bm{\mathfrak{U}}(\bm{I}_{p} + \bm{U}^{\T}\bm{\mathfrak{U}})^{-1} \\
  & = 2\bm{U}(\bm{I}_{p}+\bm{\mathfrak{U}}^{\T}\bm{U})^{-1} + 2\bm{\mathfrak{U}}(\bm{I}_{p} + \bm{U}^{\T}\bm{\mathfrak{U}})^{-1} - 2\bm{U}(\bm{I}_{p} + \bm{U}^{\T}\bm{\mathfrak{U}})(\bm{I}_{p}+\bm{U}^{\T}\bm{\mathfrak{U}})^{-1} \\
  & = 2\bm{U}(\bm{I}_{p}+\bm{\mathfrak{U}}^{\T}\bm{U})^{-1} + 2\bm{\mathfrak{U}}(\bm{I}_{p} + \bm{U}^{\T}\bm{\mathfrak{U}})^{-1} - 2\bm{U}.\label{eq:Cayley_retraction_inv}
\end{align}

Although the expression~\eqref{eq:Cayley_retraction_inv} of
$R_{\bm{U}}^{{\rm Cay}^{-1}}$
has been given by~\cite{S19,Zhu-Sato20}, our discussion via~\eqref{eq:Phi_retraction} presents much more comprehensive information about
$R_{\bm{U}}^{\rm Cay}$.
In~\cite{S19,Zhu-Sato20}, it has been reported that a certain restriction of
$R_{\bm{U}}^{\rm Cay}$
to a sufficiently small open neighborhood of
$\bm{0} \in T_{\bm{U}}\St(p,N)$
is invertible with
$R_{\bm{U}}^{{\rm Cay}^{-1}}$.
Meanwhile, we clarify that
$R_{\bm{U}}^{\rm Cay}$
is invertible on
$T_{\bm{U}}\St(p,N)$
entirely by passing through
$\Phi_{\bm{S}}^{-1}$.
The following proposition summarizes the above discussion.
\begin{proposition}\label{proposition:Cayley_retraction}
  For
  $\bm{U} \in \St(p,N)$,
  let
  $\bm{U}_{\perp} \in \St(N-p,N)$
  satisfy
  $\bm{U}^{\T}\bm{U}_{\perp} = \bm{0}$.
  Then, the Cayley transform-based retraction
  $R_{\bm{U}}^{\rm Cay}$
  in~\eqref{eq:Cayley_retraction}~\cite{W13,Z17} is diffeomorphic between
  $T_{\bm{U}}\St(p,N)$
  and
  $\St(p,N) \setminus E_{N,p}(\bm{S})$,
  and its inversion mapping
  $R_{\bm{U}}^{{\rm Cay}^{-1}}$
  is given by~\eqref{eq:Cayley_retraction_inv},
  where
  $\bm{S}:=\begin{bmatrix} \bm{U}&\bm{U}_{\perp} \end{bmatrix}$.
  In addition, for
  $p < N$,
  the image
  $\St(p,N)\setminus E_{N,p}(\bm{S})$
  of
  $R_{\bm{U}}^{\rm Cay}$
  is an open dense subset of
  $\St(p,N)$
  (see Theorem~\ref{theorem:dense}~\ref{enum:dense}).
\end{proposition}

\begin{remark}[Minimization of $f\circ R_{\bm{U}}^{\rm Cay}$ with a fixed $\bm{U}$] \label{remark:parametrization_comparison}
  By using the Cayley transform-based retraction
  $R_{\bm{U}}^{\rm Cay}$,
  the Cayley parametrization strategy in Algorithm~\ref{alg:proposed} can be modified to the minimization of
  $f\circ R_{\bm{U}}^{\rm Cay}$
  with a fixed
  $\bm{U} \in \St(p,N)$
  over
  $T_{\bm{U}}\St(p,N)$.
  The explicit formula for the gradient of
  $f\circ R_{\bm{U}}^{\rm Cay}$
  is given in Appendix~\ref{appendix:gradient_retraction}.
  Compared to the minimization of
  $f\circ R_{\bm{U}}^{\rm Cay}$
  over
  $T_{\bm{U}}\St(p,N)$,
  advantages of the minimization of
  $f\circ\Phi_{\bm{S}}^{-1}$
  with
  $\bm{S} \in {\rm O}_{p}(N)$
  over
  $Q_{N,p}(\bm{S})$
  are as follows.
  \begin{enumerate}[label=(\alph*)]
    \item
      The complexity
      $2Np^{2} + \mathfrak{o}(p^{3})$
      flops of
      $\Phi_{\bm{S}}^{-1}$
      with
      $\bm{S} \in {\rm O}_{p}(N)$
      is more efficient than
      $6Np^{2} + \mathfrak{o}(p^{3})$
      flops of
      $R_{\bm{U}}^{\rm Cay}$
      (see Remark~\ref{remark:comparison_complexity}).
      In a case where we employ the gradient descent-type method for the minimization of
      $f\circ\Phi_{\bm{S}}^{-1}$
      and
      $f\circ R_{\bm{U}}^{\rm Cay}$,
      the difference of constant factor affects run time of algorithm in practice because
      $\Phi_{\bm{S}}^{-1}$
      and
      $R_{\bm{U}}^{\rm Cay}$
      are used to estimate a stepsize many times within a line-search algorithm, e.g., the backtracking algorithm (Algorithm~\ref{alg:backtracking}), in each iteration (see, e.g.,~\cite{numerical_optimization}).
    \item
      $R_{\bm{U}}^{\rm Cay}$
      has been exploited with the aid of the Sherman-Morrison-Woodbury formula (see Fact~\ref{fact:SMW}) to reduce the complexity for matrix inversion, which can induce the deterioration of the orthogonal feasibility due to the numerical instability of its formula~\cite{W13}.
      On the other hand,
      $\Phi_{\bm{S}}^{-1}$
      does not use the formula, and thus is numerically stabler than
      $R_{\bm{U}}^{\rm Cay}$.
      This will be demonstrated numerically in Section~\ref{sec:numerical}.
      Indeed, for
      $\bm{V} \in Q_{N,p}(\bm{S})$,
      the condition number
      $\kappa(\bm{M}) := \|\bm{M}\|_{2}\|\bm{M}^{-1}\|_{2}$
      of
      $\bm{M}:=\bm{I}_{p}+\dbra{\bm{V}}_{11} + \dbra{\bm{V}}_{21}^{\T}\dbra{\bm{V}}_{21}$
      in~\eqref{eq:Cayley_inv} is upper bounded by\footnote{
        Let
        $\bm{I}_{p}+\dbra{\bm{V}}_{21}^{\T}\dbra{\bm{V}}_{21} = \bm{Q}(\bm{I}_{p} + \bm{\Sigma})\bm{Q}^{\T}$
        be the eigenvalue decomposition with
        $\bm{Q} \in {\rm O}(N)$
        and a nonnegative-valued diagonal matrix
        $\bm{\Sigma} \in \mathbb{R}^{p\times p}$.
        From~\eqref{eq:mobility_proof_A} in Appendix~\ref{appendix:mobility}, we have
        $\|\bm{M}^{-1}\|_{2} \leq \|(\bm{I}_{p} + \bm{\Sigma})^{-1}\|_{F} = (1+\sigma_{\min}^{2}(\dbra{\bm{V}}_{21}))^{-1} \leq 1$.
        Thus, we have
        $\kappa(\bm{M}) \leq \|\bm{M}\|_{2} \leq 1+\|\dbra{\bm{V}}_{11}\|_{2} + \|\dbra{\bm{V}}_{21}\|_{2}^{2}$.
      }
      $1+\|\dbra{\bm{V}}_{11}\|_{2} + \|\dbra{\bm{V}}_{21}\|_{2}^{2}$,
      implying thus
      $\bm{M}$
      is hardly become ill-conditioned whenever
      $\|\bm{V}\|_{2}$
      is not very large
      (this is usual case, e.g., in application of G-L$^{2}$CT for optimization of right orthogonal invariant functions [see Theorem~\ref{theorem:range_minimizer}]).
  \end{enumerate}
\end{remark}

\section{Numerical experiments}\label{sec:numerical}
We illustrate the performance of the proposed CP strategy in Algorithm~\ref{alg:proposed} by numerical experiments.
To demonstrate the effectiveness of the proposed formulation in Problem~\ref{problem:CP_St} in a simple situation, we implemented Algorithm~\ref{alg:proposed} with a gradient descent-type update scheme
$\mathcal{A}^{\rm GDM}:Q_{N,p}(\bm{S})\to Q_{N,p}(\bm{S}):\bm{V} \mapsto \bm{V} - \gamma \nabla f_{\bm{{S}}}(\bm{V})$
in MATLAB, where
$f_{\bm{S}} := f\circ\Phi_{\bm{S}}^{-1}$.
In
$\mathcal{A}^{\rm GDM}$
for a given
$\bm{V} \in Q_{N,p}(\bm{S})$,
we use a stepsize
$\gamma > 0$,
satisfying the so-called Armijo rule, generated by the backtracking algorithm (see, e.g.,~\cite{numerical_optimization}) with predetermined
$\gamma_{\rm initial} > 0$
and
$\rho, c \in (0,1)$
(see Algorithm~\ref{alg:backtracking}).
Armijo rule has been utilized to design a stepsize for decreasing the function value sufficiently in numerical optimization.
All the experiments were performed on MacBook Pro (13-inch, 2017) with Intel Core i5-7360U and 16GB of RAM.
\begin{algorithm}[t]
\caption{Backtracking algorithm}
\label{alg4}
  \begin{algorithmic}\label{alg:backtracking}
\REQUIRE 
$c\in (0,1),\ \rho \in (0, 1),\ \gamma_{\rm initial} > 0, \ \bm{S} \in {\rm O}(N),\ \bm{V}\in Q_{N,p}(\bm{S})$
\STATE
  $\gamma \leftarrow \gamma_{\rm initial}$
\WHILE{$f_{\bm{S}}(\bm{V} - \gamma \nabla f_{\bm{S}}(\bm{V})) > f_{\bm{S}}(\bm{V}) - c\gamma \|\nabla f_{\bm{S}}(\bm{V})\|_{F}^{2} $}
  \STATE
  $\gamma \leftarrow \rho \gamma$
\ENDWHILE
\ENSURE
$\gamma$
\end{algorithmic}
\end{algorithm}

\subsection{Comparison to the retraction-based strategy} \label{sec:numerical_comparison}
We compared Algorithm~\ref{alg:proposed}+$\mathcal{A}^{\rm GDM}$ (abbreviated as GDM+CP)
and three {\it retraction-based strategies}~\cite{manifold_book} with the steepest descent solver implemented in Manopt~\cite{manopt} in the scenario of eigenbasis extraction problem below.
Since the Cayley transform-based retraction
$R^{\rm Cay}$
in~\eqref{eq:Cayley_retraction} can be utilized for a parametrization of a subset of
$\St(p,N)$ (see Section~\ref{sec:Cayley_retraction} and Proposition~\ref{proposition:Cayley_retraction}), to see differences in performance between
$\Phi_{\bm{S}}^{-1}$
and
$R_{\bm{U}}^{\rm Cay}$,
we also compared the proposed GDM+CP and its modified version with replacement of
$\Phi_{\bm{S}}^{-1}$
by
$R_{\bm{U}}^{\rm Cay}$
(abbreviated by GDM+CP-retraction) illustrated in
Algorithm~\ref{alg:proposed_retraction}+$\widehat{\mathcal{A}}^{\rm GDM}:\bm{\mathcal{V}}\mapsto \bm{\mathcal{V}} - \gamma \nabla (f\circ R_{\bm{U}}^{\rm Cay})(\bm{\mathcal{V}})$ for the minimization of
$f\circ R_{\bm{U}}^{\rm Cay}$
with a fixed
$\bm{U} \in \St(p,N)$
over
$T_{\bm{U}}\St(p,N)$.

\begin{problem}[Eigenbasis extraction problem (e.g.,~\cite{manifold_book,W13,Z17})]\label{problem:eigen}
  For a given symmetric matrix
  $\bm{A} \in \mathbb{R}^{N\times N}$,
  \begin{equation}
    {\rm find}\ \bm{U}^{\star} \in \argmin_{\bm{U}\in\St(p,N)} f(\bm{U})\left(:=-\trace(\bm{U}^{\T}\bm{A}\bm{U})\right). \label{eq:eigen}
  \end{equation}
  Any solution
  $\bm{U}^{\star}$
  of Problem~\ref{problem:eigen} is an orthonormal eigenbasis associated with the
  $p$
  largest eigenvalues of
  $\bm{A}$.
  In our experiment, we used
  $\bm{A}:=\widetilde{\bm{A}}^{\T} \widetilde{\bm{A}} \in \mathbb{R}^{N\times N}$
  with randomly chosen
  $\widetilde{\bm{A}} \in \mathbb{R}^{N\times N}$
  of which each entry is sampled by the standard normal distribution
  $\mathcal{N}(0,1)$.
  Note that
  $f$
  is right orthogonal invariant, and thus we can exploit Theorem~\ref{theorem:range_minimizer} for GDM+CP.
\end{problem}

\begin{algorithm}[t] 
  \caption{Cayley parametrization strategy with the Cayley transform-based retraction (Algorithm~\ref{alg:proposed_retraction}+$\mathcal{A}$)}
\begin{algorithmic}
  \label{alg:proposed_retraction}
\REQUIRE
$\bm{U},\bm{U}_0 \in \St(p,N)$,
$\mathcal{A}:T_{\bm{U}}\St(p,N)\to T_{\bm{U}}\St(p,N)$: update rule
\STATE
$\bm{\mathcal{V}}_0 =R_{\bm{U}}^{{\rm Cay}^{-1}}(\bm{U}_{0})$
\FOR{$n=0,1,2,\ldots,m-1$}
  \STATE
  $\bm{\mathcal{V}}_{n+1} = \mathcal{A}(\bm{\mathcal{V}}_{n})$
  \STATE
  $\bm{U}_{n+1} = R_{\bm{U}}^{\rm Cay}(\bm{\mathcal{V}}_{n+1})$
\ENDFOR
\ENSURE
$\bm{U}_{m}$
\end{algorithmic}
\end{algorithm}

For the retraction-based strategies, we employed three retractions: (i) Cayley transform-based (abbreviated by GDM+Cayley)~\cite{W13}; (ii) QR decomposition-based (abbreviated by GDM+QR)~\cite{manifold_book}; (iii) polar decomposition-based (abbreviated by GDM+polar)~\cite{manifold_book}.
In the steepest descent solver in Manopt, we calculated a stepsize for the current estimate
$\bm{U} \in \St(p,N)$
with Algorithm~\ref{alg:backtracking} after replacement of the criterion
$f_{\bm{S}}(\bm{V}-\gamma \nabla f_{\bm{S}}(\bm{V})) > f_{\bm{S}}(\bm{V}) - c\gamma \|\nabla f_{\bm{S}}(\bm{V})\|_{F}^{2}$
by
$f\circ R_{\bm{U}}(-\gamma \mathrm{grad}\ f(\bm{U})) > f(\bm{U}) - c\gamma\|\mathrm{grad}\ f(\bm{U})\|_{F}^{2}$
(see, e.g.,~\cite[Algorithm 3.1]{Sato21}),
where
$\mathrm{grad}\ f(\bm{U}) = \mathcal{P}_{T_{\bm{U}}\St(p,N)} (\nabla f(\bm{U})) \in T_{\bm{U}}\St(p,N)$
is the Riemannian gradient of
$f$
at
$\bm{U} \in \St(p,N)$
(for the projection mapping
$\mathcal{P}_{T_{\bm{U}}\St(p,N)}$,
see Fact~\ref{fact:stiefel}~\ref{enum:tangent}).

For an initial point
$\bm{U}_{0} \in \St(p,N)$,
we used a center point for GDM+CP as
$\bm{S} := \diag(\bm{Q}_{1}\bm{Q}_{2}^{\T},\bm{I}_{N-p}) \in {\rm O}_{p}(N)$
by using a singular value decomposition of
$[\bm{U}_{0}]_{\rm up} = \bm{Q}_{1}\bm{\Sigma}\bm{Q}_{2}^{\T}$
with
$\bm{Q}_{1},\bm{Q}_{2} \in {\rm O}(p)$
and a nonnegative-valued diagonal matrix
$\bm{\Sigma} \in \mathbb{R}^{p\times p}$ (see Theorem~\ref{theorem:construct_center}).
For GDM+CP-retraction, we used a fixed
$\bm{U}:= \bm{U}_{0}$
for the minimization of
$f\circ R_{\bm{U}}^{\rm Cay}$.
We note that the choice of
$\bm{U}:= \bm{U}_{0}$
is reasonable because the procedure of Algorithm~\ref{alg:proposed_retraction}($\bm{U}_{0},\bm{U}_{0},\widehat{A}^{\rm GDM}$), which tries to decrease
$f\circ R_{\bm{U}_{0}}^{\rm Cay}$
from the initial point
$R_{\bm{U}_{0}}^{{\rm Cay}^{-1}}(\bm{U}_{0}) = \bm{0} \in T_{\bm{U}_{0}}\St(p,N)$,
is the same as the procedure of GDM+Cayley in the first iteration.
The explicit formula for the gradient of
$f\circ R_{\bm{U}}^{\rm Cay}$
is given in Appendix~\ref{appendix:gradient_retraction}.

For five algorithms, we used the default parameters
$\rho = 0.5$
and
$c = 2^{-13}$
in Manopt.
We employed several initial stepsizes
$\gamma_{\rm initial} \in \{0.1,0.01,0.001\}$.
We generated an initial point
$\bm{U}_{0} \in \St(p,N)$
by using
"orth(rand(N,p))"
in MATLAB.

For each algorithm, we stopped the update at $n$th iteration when it achieved the following conditions (used in~\cite{Zhu-Sato20}) with
$\bm{D}_{n} := \nabla f_{\bm{S}}(\bm{V}_{n})$,
$\nabla (f\circ R_{\bm{U}_{0}}^{\rm Cay})(\bm{\mathcal{V}}_{n})$,
$\mathrm{grad}\ f(\bm{U}_{n})$:
\begin{equation}
  n \geq 5000\ {\rm or} \ 
  \frac{\|\bm{D}_{n}\|_{F}}{\|\bm{D}_{0}\|_{F}} \leq 10^{-10}\ 
  {\rm or}\ 
  \frac{|f(\bm{U}_{n})-f(\bm{U}_{n-1})|}{|f(\bm{U}_{n})|} \leq 10^{-20}. \label{eq:stopping}
\end{equation}

Table~\ref{table:results} illustrates average results for
$10$
trials of each algorithm employing the initial stepsize
$\gamma_{\rm initial} \in \{0.1,0.01,0.001\}$
with the shortest CPU time to reach the stopping criteria in the scenario of Problem~\ref{problem:eigen} with
$(N,p) \in \{1000,2000\} \times \{10,50\}$.
In the table, "fval" means the value
$f(\bm{U})$
at the output
$\bm{U} \in \St(p,N)$,
"fval-optimal" means
$f(\bm{U}) - f(\bm{U}^{\star})$
with the global minimizer
$\bm{U}^{\star} \in \St(p,N)$
obtained by the eigenvalue decomposition of
$\bm{A}$,
"feasi" means the feasibility
$\|\bm{I}_{p}-\bm{U}^{\T}\bm{U}\|_{F}$,
"nrmg" means the norm
$\|\nabla f_{\bm{S}}(\Phi_{\bm{S}}(\bm{U}))\|_{F}$,
$\|\nabla (f\circ R_{\bm{U}_{0}}^{{\rm Cay}})(R_{\bm{U}_{0}}^{{\rm Cay}^{-1}}(\bm{U}))\|_{F}$
or
$\|\mathrm{grad}\ f(\bm{U})\|_{F}$,
"itr" means the number of iterations, and "time" means the CPU time (s).
Figure~\ref{fig:eigen} shows the convergence history of algorithms for each problem size respectively.
The plots show CPU time on the horizontal axis versus the value
$f(\bm{U}) - f(\bm{U}^{\star})$
on the vertical axis.

We observe that the proposed GDM+CP reaches the stopping criteria with the shortest CPU time among all five algorithms for every problem size.
Possible reasons for the superiority of the proposed Cayley parametrization strategy to the retraction-based strategy are as follows.
\begin{enumerate}[label=(\roman*)]
  \item
    The Cayley parametrization strategy exploits the diffeomorphic property of
    $\Phi_{\bm{S}}^{-1}$
    between a vector space and an open dense subset of
    $\St(p,N)$
    while the retraction-based strategy exploits only a local diffeomorphic property around
    $\bm{0}$
    of retractions (see Remark~\ref{remark:relation_to_others}).
  \item
    For Problem~\ref{problem:eigen}, a global minimizer
    $\bm{V}^{\star} \in Q_{N,p}(\bm{S})$
    of
    $f\circ \Phi_{\bm{S}}^{-1}$
    exists\footnote{
      From the relation
      $\min_{\bm{U}\in\St(p,N)} f(\bm{U})  = \inf_{\bm{V}\in Q_{N,p}(\bm{S})} f\circ\Phi_{\bm{S}}^{-1}(\bm{V})$
      in Lemma~\ref{lemma:gap},
      $\Phi_{\bm{S}}^{-1}(\bm{V}^{\star}) \in\St(p,N)$
      is also a global minimizer of
      $f$
      over
      $\St(p,N)$,
    }
    within the unit ball
    $\{\bm{V} \in Q_{N,p}(\bm{S})\ \mid \|\bm{V}\|_{2} \leq 1\}$
    due to the right orthogonal invariance of
    $f$
    in~\eqref{eq:eigen} and Theorem~\ref{theorem:range_minimizer}.
    In comparison, the existence of a global minimizer, say
    $\bm{\mathcal{V}}^{\star} \in T_{\bm{U}_{n}}\St(p,N)$,
    of
    $f\circ R_{\bm{U}_{n}}$
    over
    $T_{\bm{U}_{n}}\St(p,N)$
    is not guaranteed for a general retraction
    $R$.
    Even if such a
    $\bm{\mathcal{V}}^{\star}$
    exists, it is not guaranteed that
    $R_{\bm{U}_{n}}(\bm{\mathcal{V}}^{\star}) \in \St(p,N)$
    is a global minimizer of
    $f$
    over
    $\St(p,N)$
    because
    $R_{\bm{U}_{n}}(T_{\bm{U}_{n}}\St(p,N))$
    is not necessarily dense in
    $\St(p,N)$.
\end{enumerate}
Additionally, GDM+CP can keep the feasibility at the same level as GDM+QR and GDM+polar, and better than GDM+Cayley.
These observations imply that the proposed CP strategy outperforms the retraction-based strategy.
Moreover, it is expected that the proposed CP strategy achieves fast convergence to a solution for Problem~\ref{problem:origin} when we plug more powerful computational arts designed for optimization over a vector space into the CP strategy (see also Remark~\ref{remark:relation_to_others}).

As shown in Propositions~\ref{proposition:inverse} and~\ref{proposition:Cayley_retraction},
both
$\Phi_{\bm{S}}^{-1}$
and
$R_{\bm{U}_{0}}^{\rm Cay}$
can parameterize respectively open dense subsets of
$\St(p,N)$.
However, we observe that (i) the proposed GDM+CP for the minimization of
$f\circ \Phi_{\bm{S}}^{-1}$
has faster convergence speed than GDM+CP-retraction for the minimization of
$f\circ R_{\bm{U}_{0}}^{\rm Cay}$;
(ii) the orthogonal feasibility in GDM+CP-retraction deteriorates compared than GDM+CP.
We believe that these performance differences are made respectively by (i) the computational complexity for
$\Phi_{\bm{S}}^{-1}$
is more efficient than that of
$R_{\bm{U}_{0}}^{\rm Cay}$,
and by
(ii)
calculations of
$R_{\bm{U}_{0}}^{\rm Cay}$
and
$\nabla (f\circ R_{\bm{U}_{0}}^{\rm Cay})$
require the Sherman-Morrison-Woodbury formula for matrix inversions in order to achieve comparable computational complexities, and its formula is known to have a numerical instability~\cite{W13} (see Remark~\ref{remark:parametrization_comparison}).

Moreover, although GDM+CP-retraction reaches the stopping criteria without achieving the same level of the final cost value as the others\footnote{We note that this early stopping of GDM+CP-retraction can be caused by the instability~\cite{W13} of the Sherman-Morrison-Woodbury formula used in
$R_{\bm{U}_{0}}^{\rm Cay}$
and
$\nabla (f\circ R_{\bm{U}_{0}}^{\rm Cay})$.
}, GDM+CP-retraction has the same or better performance than GDM+Cayley in view of convergence history in Figure~\ref{fig:eigen} at every time.
This indicates an efficacy of the parametrization strategy of
$\St(p,N)$
in the vector space reformulation for Problem~\ref{problem:origin} because GDM+CP-retraction and GDM+Cayley used the same Cayley transform-based retraction.

Finally, we remark that if
$\gamma_{\rm initial}$
is set as too large, numerical performance of the proposed GDM+CP can deteriorate because a generated sequence
$(\bm{V}_{n})_{n=0}^{\infty} \subset Q_{N,p}(\bm{S})$
can go away from
$\bm{0} \in Q_{N,p}(\bm{S})$
quickly, which induces the insensitivity of
$\Phi_{\bm{S}}^{-1}$ (see Section~\ref{sec:mobility}).
This tendency can be observed from Figure~\ref{fig:eigen_step}, which illustrates average convergence histories for
$10$
trials of GDM+CP with each stepsize
$\gamma_{\rm initial} \in \{0.1,0.01,0.001,0.0001\}$
in the scenario of Problem~\ref{problem:eigen}.
Figure~\ref{fig:eigen_step} shows that GDM+CP with
$\gamma_{\rm initial} = 0.001$
has the best performance among four algorithms.
This observation indicates that we need not set $\gamma_{\rm initial}$ as large for GDM+CP.
Not surprisingly, we also see that too small
$\gamma_{\rm initial}$
causes slow convergence speed of GDM+CP with move only a little along
$-\nabla f_{\bm{S}}(\bm{V}_{n})$
at each iteration.

\begin{table}[t]
  \centering
  \footnotesize
  \begin{filecontents*}{csv/table.csv}
Algorithm,fval,fval-optimal,feasi,nrmg,itr,nfe,time
$N=1000$ and $p=10$,,,,,,,
$\quad$GDM+CP ($\gamma_{\rm initial} = 0.001$),-3.81012e+04,3.70565e-08,1.054e-15,2.126e-03,1150.2,2296.2,5.35
$\quad$GDM+CP-retraction ($\gamma_{\rm initial} = 0.001$),-3.81012e+04,1.56552e-06,3.132e-15,5.437e-03,1520.8,3041.6,10.01
$\quad$GDM+Cayley ($\gamma_{\rm initial} = 0.01$),-3.81012e+04,7.00529e-08,3.311e-15,4.412e-03,707,4265.2,11.61
$\quad$GDM+QR ($\gamma_{\rm initial} = 0.01$),-3.81012e+04,2.26501e-08,1.116e-15,3.028e-03,732.5,4412.8,10.82
$\quad$GDM+polar ($\gamma_{\rm initial} = 0.01$),-3.81012e+04,1.95723e-08,2.163e-15,3.140e-03,733.2,4416.9,10.96

$N=1000$ and $p=50$,,,,,,,
$\quad$GDM+CP ($\gamma_{\rm initial} = 0.001$),-1.72485e+05,2.69647e-07,3.735e-15,3.951e-03,2029.7,4530.1,36.64
$\quad$GDM+CP-retraction ($\gamma_{\rm initial} = 0.01$),-1.72485e+05,9.80363e-07,1.104e-14,4.382e-03,1054.9,4947.5,57.99
$\quad$GDM+Cayley ($\gamma_{\rm initial} = 0.01$),-1.72485e+05,2.03174e-07,1.198e-14,7.957e-03,1012,6088.9,57.40
$\quad$GDM+QR ($\gamma_{\rm initial} = 0.01$),-1.72485e+05,2.58849e-07,2.622e-15,1.091e-02,993.4,5970.9,36.89
$\quad$GDM+polar ($\gamma_{\rm initial} = 0.01$),-1.72485e+05,1.63767e-07,6.798e-15,8.516e-03,1040.4,6253.7,42.72

$N=2000$ and $p=10$,,,,,,,
$\quad$GDM+CP ($\gamma_{\rm initial} = 0.001$),-7.76107e+04,2.28058e-07,9.884e-16,7.848e-03,1644.6,4567,37.74
$\quad$GDM+CP-retraction ($\gamma_{\rm initial} = 0.001$),-7.76107e+04,5.84120e-06,3.218e-15,1.035e-02,2006.4,4426.4,44.63
$\quad$GDM+Cayley ($\gamma_{\rm initial} = 0.01$),-7.76107e+04,2.41329e-07,3.419e-15,1.062e-02,1504.7,10549.4,64.72
$\quad$GDM+QR ($\gamma_{\rm initial} = 0.01$),-7.76107e+04,2.06463e-07,9.781e-16,7.950e-03,1526.9,10699.9,62.62
$\quad$GDM+polar ($\gamma_{\rm initial} = 0.01$),-7.76107e+04,1.88753e-07,2.159e-15,9.277e-03,1532.1,10736.5,61.37

$N=2000$ and $p=50$,,,,,,,
$\quad$GDM+CP ($\gamma_{\rm initial} = 0.001$),-3.64801e+05,1.28115e-06,3.438e-15,1.099e-02,2224.3,6622.1,96.03
$\quad$GDM+CP-retraction ($\gamma_{\rm initial} = 0.001$),-3.64801e+05,9.82929e-06,1.126e-14,1.499e-02,2536,6461.8,136.46
$\quad$GDM+Cayley ($\gamma_{\rm initial} = 0.01$),-3.64801e+05,1.44588e-06,1.169e-14,2.509e-02,1523.8,10701.4,157.08
$\quad$GDM+QR ($\gamma_{\rm initial} = 0.01$),-3.64801e+05,7.50762e-07,2.253e-15,2.177e-02,1542.7,10827.6,125.48
$\quad$GDM+polar ($\gamma_{\rm initial} = 0.01$),-3.64801e+05,1.35263e-06,6.902e-15,2.654e-02,1506.3,10571.8,130.28
\end{filecontents*}
\csvautobooktabular{csv/table.csv}
  \caption{Performance of each algorithm applied to Problem~\ref{problem:eigen}.}
  \label{table:results}
\end{table}
\begin{figure}[t]
  \subfloat[][$N=1000,p=10$]{
    \includegraphics[clip, width=0.5\columnwidth]{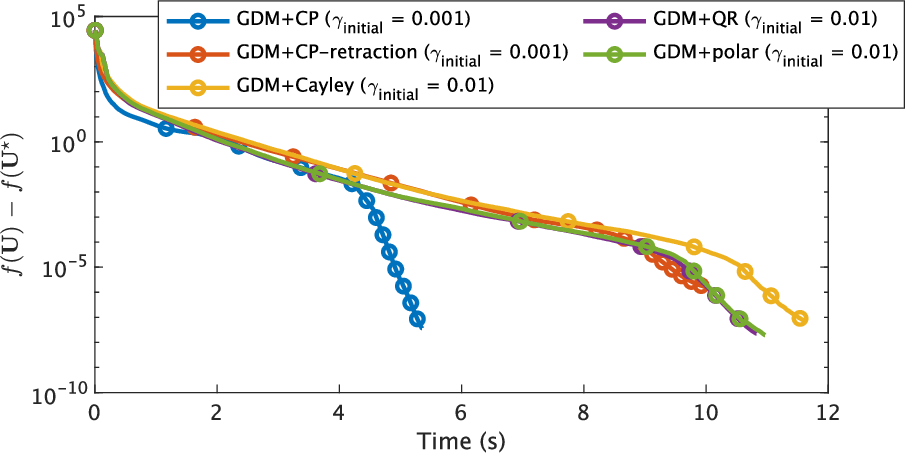}
  }
  \subfloat[][$N=1000,p=50$]{
    \includegraphics[clip, width=0.5\columnwidth]{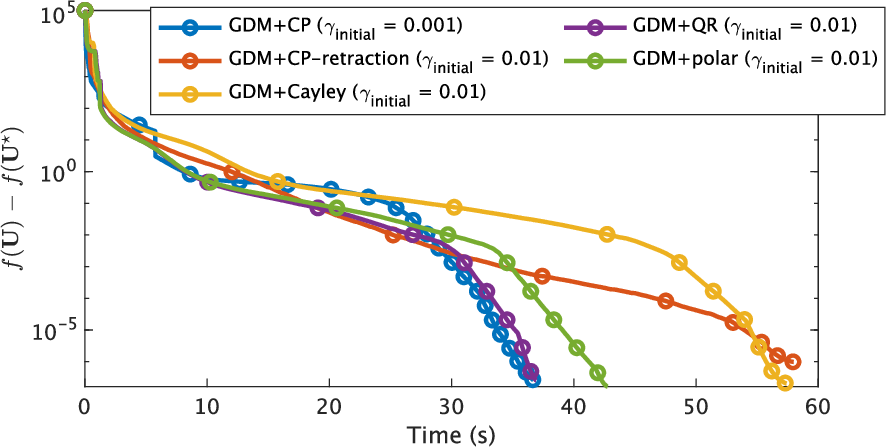}
  } \\
  \subfloat[][$N=2000,p=10$]{
    \includegraphics[clip, width=0.5\columnwidth]{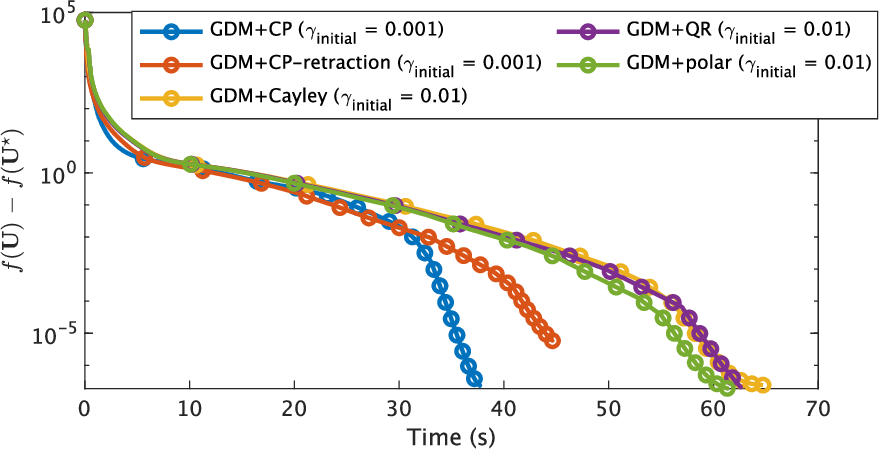}
  }
  \subfloat[][$N=2000,p=50$]{
    \includegraphics[clip, width=0.5\columnwidth]{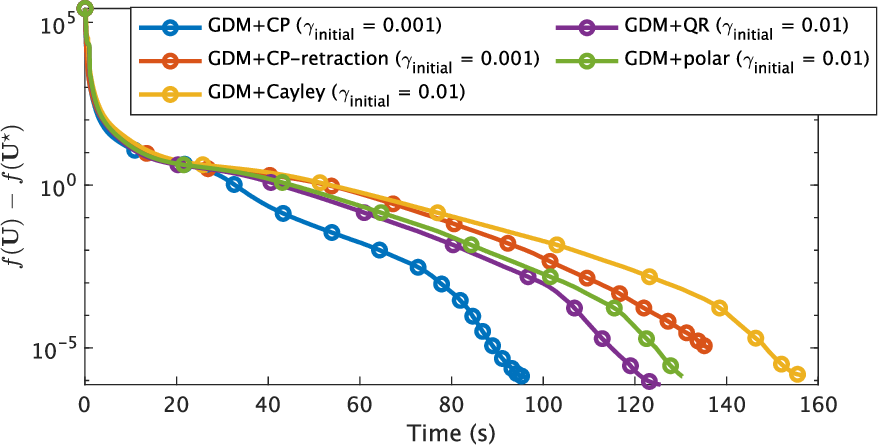}
  }
  \caption{Convergence histories of each algorithm applied to Problem~\ref{problem:eigen} regarding the value $f(\bm{U})-f(\bm{U}^{\star})$ at CPU time for each problem size.
    Markers are put at every 250 iterations.}
  \label{fig:eigen}
\end{figure}

\begin{figure}[t]
  \begin{center}
    \includegraphics[width=0.8\linewidth]{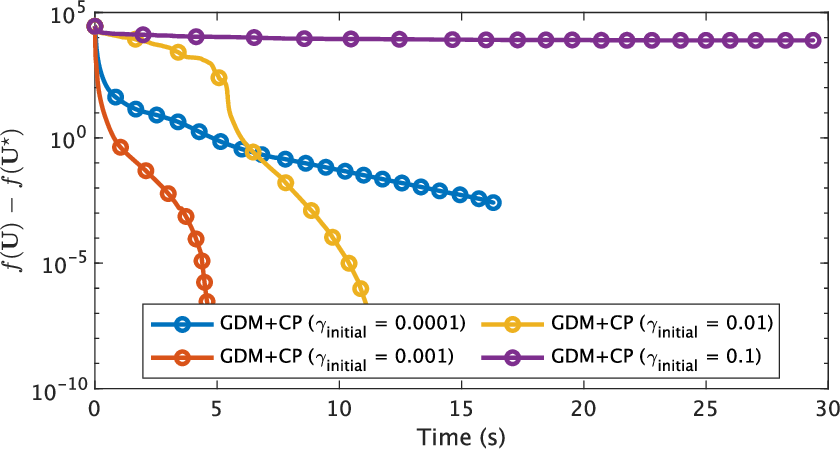}
    \caption{Convergence histories of GDM+CP with each $\gamma_{\rm initial}$ applied to Problem~\ref{problem:eigen} ($N=1000,p=10$) regarding the value $f(\bm{U})-f(\bm{U}^{\star})$ at CPU time for each problem size.
    Markers are put at every 250 iterations.}
  \label{fig:eigen_step}
  \end{center}
\end{figure}
\subsection{Singular-point issue} \label{sec:numerical_singular}
In this subsection, we tested how much the singular-points influence the performance of the proposed CP strategy.
As we mentioned in Section~\ref{sec:mobility}, a risk of the slow convergence of Algorithm~\ref{alg:proposed} can arise in a case where a global minimizer
$\bm{U}^{\star} \in \St(p,N)$
of Problem~\ref{problem:origin} is close to the singular-point set
$E_{N,p}(\bm{S})$.
To see such an influence, we compared CP strategies with several center point
$\bm{S}$
by a toy Problem~\ref{problem:origin} for the minimization of
$f(\bm{U}):=\frac{1}{2}\|\bm{U}-\bm{U}^{\star}\|_{F}^{2}$
with a given
$\bm{U}^{\star} \in \St(p,N)$.
Clearly, its solution is
$\bm{U}^{\star}$.

In this experiment, we used center points
$\bm{S}(\theta):=\diag(\bm{R}(\theta),\bm{I}_{N-p}) \in {\rm O}(N)\ (\theta=\pi/1000, \pi/4, \pi/2, \pi)$,
the global minimizer
$\bm{U}^{\star} := [\bm{S}(\pi)]_{\rm le}$
and an initial point
$\bm{U}_{0} := [\bm{S}(\pi/4)]_{\rm le}$,
where
$\bm{R}(\theta):=\begin{bmatrix} \cos(\theta) & -\sin(\theta) \\ \sin(\theta) & \cos(\theta) \end{bmatrix} \in {\rm SO}(2)$
is a rotation matrix.
From
$[\bm{S}(\theta)]_{\rm le}^{\T}\bm{U}^{\star} = \diag(-\bm{R}(\theta),\bm{I}_{p-2})$,
we have
$\det(\bm{I}_{p} + [\bm{S}(\theta)]_{\rm le}^{\T}\bm{U}^{\star}) = 2^{p-1}(1-\cos(\theta))$.
Therefore,
$E_{N,p}(\bm{S}(\theta)) = \{\bm{U} \in \St(p,N) \mid \det(\bm{I}+[\bm{S}(\theta)]_{\rm le}^{\T}\bm{U}) = 0\}$
approaches
$\bm{U}^{\star}$
as
$\theta \to 0$,
and
$E_{N,p}(\bm{S}(\pi))$
is farthest from
$\bm{U}^{\star}$.

We used the stopping criteria~\eqref{eq:stopping}, and parameters
$\rho = 0.5$,
$c = 2^{-13}$,
and
$\gamma_{\rm initial} = 0.1$
for Algorithm~\ref{alg:backtracking} to determine a stepsize
$\gamma > 0$.

Table~\ref{table:results_singular_S} illustrates average results for
$10$
trials of each algorithm with
$N=1000$
and
$p=10$
in this scenario.
Figure~\ref{fig:singular_S_1000_10} shows the convergence history of algorithms.
The plot shows CPU time on the horizontal axis versus the value
$f(\bm{U}) - f(\bm{U}^{\star})$
on the vertical axis.

From Figure~\ref{fig:singular_S_1000_10}, we observe that GDM+CP with
$\bm{S}(\pi)$
is the fastest among all algorithms.
On the other hand,
$\bm{U}_{n}$
generated by GDM+CP with
$\bm{S}(\pi/1000)$
does not approach a global minimizer
$\bm{U}^{\star}$.
This implies that the convergence speed of GDM+CP tends to become slower as
$\theta \to 0$,
or equivalently as
$\bm{U}^{\star}$
approaches the singular-point set.

From these observations, the performance of the proposed Algorithm~\ref{alg:proposed} depends heavily on tuning
$\bm{S}$
as mentioned in~\ref{sec:mobility}.
Since we can not see whether a solution
$\bm{U}^{\star}$
is distant from
$E_{N,p}(\bm{S})$
or not in advance before running algorithms, it is desired to circumvent the influence of this singular-point issue.
In~\cite{Kume-Yamada19,Kume-Yamada20}, we presented preliminary reports for a CP strategy with an adaptive changing center point scheme to avoid the singular-point issue by considering Problem~\ref{problem:ALCP} instead of Problem~\ref{problem:CP_St}.

\begin{table}[t]
  \centering
  \footnotesize
  \begin{filecontents*}{csv/singular1000_10-20221503160214595.csv}
Algorithm,fval,feasi,nrmg,itr,nfe,time
GDM+CP with $S(\pi/1000)$,4.41769e-02,4.227e-17,6.726e-03,5001,5001,1.82
GDM+CP with $S(\pi/4)$,3.33129e-28,1.570e-16,7.560e-15,3451,3486,1.04
GDM+CP with $S(\pi/2)$,1.46362e-30,0.000e+00,1.711e-15,321,360,0.10
GDM+CP with $S(\pi)$,2.43087e-63,0.000e+00,1.395e-31,184,224,0.06
\end{filecontents*}
\csvautobooktabular{csv/singular1000_10-20221503160214595.csv}
  \caption{Performance of each algorithm applied Problem~\ref{problem:origin} with $f(\bm{U}):=\frac{1}{2}\|\bm{U}-\bm{U}^{\star}\|_{F}^{2}$.
}
  \label{table:results_singular_S}
\end{table}

\begin{figure}[t]
  \begin{center}
    \includegraphics[width=0.8\linewidth]{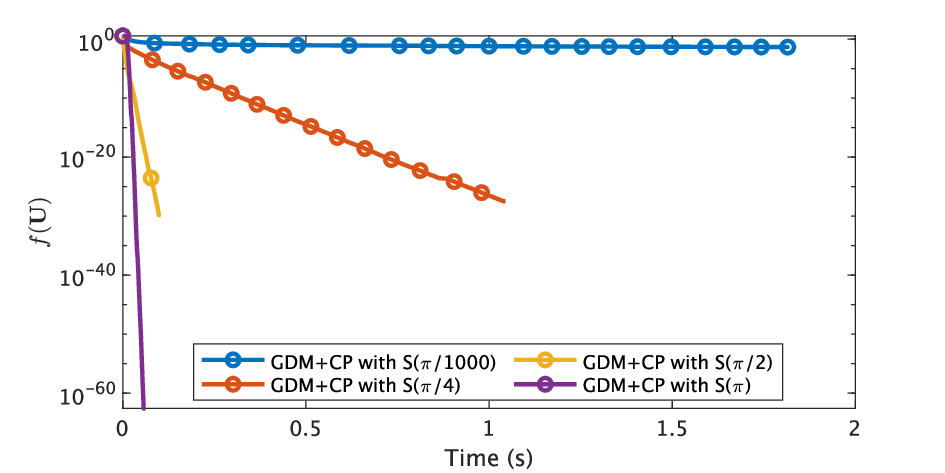}
    \caption{Convergence histories of each algorithm applied to Problem~\ref{problem:origin} with
      $f(\bm{U}):=\frac{1}{2}\|\bm{U}-\bm{U}^{\star}\|_{F}^{2}$,
      and
      $N=1000,p=10$ regarding the value $f(\bm{U})$ at CPU time.
    Markers are put at every 250 iterations.
  }
  \label{fig:singular_S_1000_10}
  \end{center}
\end{figure}
\section{Conclusion}
We presented a generalization of the Cayley transform for the Stiefel manifold to parameterize a dense subset of the Stiefel manifold in terms of a single vector space.
The proposed Cayley transform is diffeomorphic between a dense subset of the Stiefel manifold and a vector space.
Thanks to the diffeomorphic property, we proposed a new reformulation of optimization problem over the Stiefel manifold to transplant optimization techniques designed over a vector space.
Numerical experiments have shown that the proposed algorithm outperformed the standard algorithms designed with a retraction on the Stiefel manifold under a simple situation.

\section*{Funding}
This work was supported by JSPS Grants-in-Aid (19H04134) partially, by JSPS Grants-in-Aid (21J21353) and by JST SICORP (JPMJSC20C6).

\bibliographystyle{tfnlm}
\bibliography{main}

\appendix

\section{Basic facts on the Stiefel manifold, the Cayley transform and tools for matrix analysis}\label{appendix:facts}
In this section, we summarize basic properties on
$\St(p,N)$
and the Cayley transform together with elementary tools for matrix analysis.

\begin{fact}[Stiefel manifold~\cite{manifold_book,E98}]\label{fact:stiefel}
  \mbox{}
  \begin{enumerate}[label=(\alph*)]
      \item\label{enum:submanifold}
        The Stiefel manifold
        $\St(p,N)$
        is an embedded submanifold of
        $\mathbb{R}^{N\times p}$.
        The topology
        $\mathcal{O}(\St(p,N))$,
        the family of all open subsets, of
        $\St(p,N)$
        is defined as any union of sets in
        $\{\St(p,N) \cap B_{\mathbb{R}^{N\times p}}(\bm{X},r) \mid \bm{X} \in \mathbb{R}^{N\times p}, r>0 \}$.
      \item\label{enum:dimension}
          The dimension of
          $\St(p,N)$
          is
          $Np-\frac{1}{2}p(p+1)$,
          i.e., every point
          $\bm{U} \in \St(p,N)$
          has an open neighborhood
          $\mathcal{N}(\bm{U})\subset \St(p,N)$
          such that there exists a homeomorphism
          $\phi:\mathcal{N}(\bm{U}) \to \mathbb{R}^{Np-p(p+1)/2}$
          between
          $\mathcal{N}(\bm{U})$
          and some open subset of
          $\mathbb{R}^{Np-p(p+1)/2}$.
      \item\label{enum:connected}
        The Stiefel manifold
        $\St(p,N)$
        is compact.
        Moreover,
        $\St(p,N)$
        with
        $p<N$
        is connected while
        ${\rm O}(N) := \St(N,N)$
        is a disconnected union of connected subsets
        ${\rm SO}(N):=\{\bm{U}\in{\rm O}(N) \mid \det(\bm{U})=1\}$
        and
        ${\rm O}(N)\setminus {\rm SO}(N)$.
      \item\label{enum:tangent}
        The tangent space to
        $\St(p,N)$
        at
        $\bm{U} \in\St(p,N)$
        is expressed as
        \begin{align}
          T_{\bm{U}}\St(p,N)
          & = \{\bm{\mathcal{V}} \in \mathbb{R}^{N\times p} \mid \bm{U}^{\T}\bm{\mathcal{V}} + \bm{\mathcal{V}}^{\T}\bm{U} = \bm{0}\} \\
          & = \{\bm{U}\bm{\Omega} + \bm{U}_{\perp}\bm{K} \in \mathbb{R}^{N\times p} \mid \bm{\Omega}^{\T} = -\bm{\Omega} \in \mathbb{R}^{p\times p}, \bm{K} \in \mathbb{R}^{(N-p) \times p}\}
        \end{align}
        in terms of an arbitrarily chosen
        $\bm{U}_{\perp} \in \St(N-p,p)$
        satisfying
        $\bm{U}^{\T}\bm{U}_{\perp} = \bm{0} \in \mathbb{R}^{p\times (N-p)}$ (see, e.g.,~\cite[Example 3.5.2]{manifold_book}).
        The projection mapping
        $\mathcal{P}_{T_{\bm{U}}\St(p,N)}:\mathbb{R}^{N\times p} \to T_{\bm{U}}\St(p,N)$
        onto
        $T_{\bm{U}}\St(p,N)$
        is given by\footnote{
          The subspace
          $W_{1}:=\{\bm{U}\bm{\Omega} \in \mathbb{R}^{N\times p}\mid \bm{\Omega}^{\T}=-\bm{\Omega} \in \mathbb{R}^{p\times p}\} \subset \mathbb{R}^{N\times p}$
          is an orthogonal complement to the subspace
          $W_{2}:=\{\bm{U}_{\perp}\bm{K} \in \mathbb{R}^{N\times p} \mid \bm{K} \in \mathbb{R}^{(N-p)\times p}\}\subset \mathbb{R}^{N\times p}$
          with the inner product
          $\inprod{\bm{X}}{\bm{Y}} = \trace(\bm{X}^{\T}\bm{Y})\ (\bm{X},\bm{Y}\in \mathbb{R}^{N\times p})$.
          The tangent space
          $T_{\bm{U}}\St(p,N)$
          can be decomposed as
          $W_{1} \oplus W_{2}$
          with the direct sum
          $\oplus$.
          In view of the orthogonal decomposition, the first term and the second term in the right-hand side of~\eqref{eq:tangent_projection} can be regarded respectively as the orthogonal projection of
          $\bm{X}$
          onto
          $W_{1}$
          and
          $W_{2}$.
        } (see, e.g.,~\cite[Example 3.6.2]{manifold_book})
        \begin{align}
          (\bm{X}\in\mathbb{R}^{N\times p})\ 
          \mathcal{P}_{T_{\bm{U}}\St(p,N)}(\bm{X})
          &:= \argmin_{\bm{Z} \in T_{\bm{U}}\St(p,N)} \|\bm{X} - \bm{Z}\|_{F} \\
          & = \frac{1}{2}\bm{U}(\bm{U}^{\T}\bm{X}-\bm{X}^{\T}\bm{U}) + (\bm{I}-\bm{U}\bm{U}^{\T})\bm{X}.\label{eq:tangent_projection}
        \end{align}
      \end{enumerate}
\end{fact}

\begin{fact}[Commutativity of the Cayley transform pair, e.g.,~\cite{S_linear}]
  The Cayley transform
  $\varphi$
  in~\eqref{eq:origin_Cayley}
  and
  its inversion
  $\varphi^{-1}$
  in~\eqref{eq:inv_origin_Cayley}
  can be expressed as
  \begin{align}
    (\bm{U}\in {\rm O}(N)\setminus E_{N,N}) \quad& \varphi(\bm{U}) = (\bm{I}-\bm{U})(\bm{I}+\bm{U})^{-1} = (\bm{I}+\bm{U})^{-1}(\bm{I}-\bm{U}) \\
    (\bm{V} \in Q_{N,N})\quad & \varphi^{-1}(\bm{V}) = (\bm{I}-\bm{V})(\bm{I}+\bm{V})^{-1} = (\bm{I}+\bm{V})^{-1}(\bm{I}-\bm{V}).\label{eq:commute}
  \end{align}
\end{fact}

\begin{fact}[Denseness of ${\rm O}(N)\setminus E_{N,N}(\bm{S})$; see~\cite{Y03} for $\bm{S}=\bm{I}$] \label{fact:dense}
  For
  $\bm{S}\in {\rm O}(N)$,
  define
  ${\rm O}(N,\bm{S}):= \{\bm{U} \in {\rm O}(N) \mid \det(\bm{U})=\det(\bm{S})\}$,
  i.e.,
  \begin{equation}
    {\rm O}(N,\bm{S}) =
    \begin{cases}
      {\rm SO}(N) & (\mathrm{if}\ \det(\bm{S}) = 1) \\
      {\rm O}(N)\setminus {\rm SO}(N) & (\mathrm{if}\ \det(\bm{S}) = -1).
    \end{cases}\label{eq:O_S}
  \end{equation}
  Then, for
  $\bm{S}\in {\rm O}(N)$ and
  $E_{N,N}(\bm{S})$
  defined just after~\eqref{eq:ILCT_O},
  ${\rm O}(N)\setminus E_{N,N}(\bm{S})$
  is a dense subset of
  ${\rm O}(N,\bm{S})$,
  i.e., the closure of
  ${\rm O}(N)\setminus E_{N,N}(\bm{S})$
  is
  ${\rm O}(N,\bm{S})$.
\end{fact}
\begin{proof}
  It suffices to show for
  $\bm{U} \in {\rm O}(N,\bm{S})$
  that there exists a sequence
  $(\bm{U}_{n})_{n=0}^{\infty} \subset {\rm O}(N)\setminus E_{N,N}(\bm{S})$
  such that
  $\lim_{n\to\infty}\bm{U}_{n} = \bm{U}$.

  Let
  $\bm{U} \in  {\rm O}(N,\bm{S})$.
  Then,
  $\bm{S}^{\T}\bm{U}$
  can be expressed as
  $\bm{S}^{\T}\bm{U} = \bm{Q}^{\T}\bm{\Lambda}\bm{Q}$
  with some
  $\bm{Q} \in {\rm O}(N)$
  and
  \begin{equation}
    \bm{\Lambda} = \diag\left(\bm{I}_{k_{1}}, -\bm{I}_{k_{2}}, \bm{R}(\theta_{1}), \bm{R}(\theta_{2}),\ldots, \bm{R}(\theta_{m})\right)
    \in {\rm O}(N) \label{eq:orthogonal_normal}
  \end{equation}
  (see~\cite[IV.~\S5]{S_linear}), where
  $k_{1}, k_{2},m \in \mathbb{N}\cup\{0\}$
  satisfy
  $k_{1} + k_{2} + 2m = N$,
  and
  $\bm{R}(\theta) := \begin{bmatrix} \cos(\theta) & -\sin(\theta) \\ \sin(\theta) & \cos(\theta) \end{bmatrix} \in {\rm SO}(2)$
  $(\theta \in [0,2\pi)\setminus \{0,\pi\})$.
  The relation
  $\det(\bm{U}) = \det(\bm{S}) = \det(\bm{S}^{\T})$
  ensures
  $\bm{S}^{\T}\bm{U} \in {\rm SO}(N)$,
  thus the number
  $k_{2}$
  must be even.
  Define
  $\bm{U}_{n} = \bm{S}\bm{Q}^{\T}\bm{\Lambda}(\pi + 1/n)\bm{Q} \in {\rm O}(N)$
  $(n \in \mathbb{N})$,
  where
  $\bm{\Lambda}(\pi + 1/n)\in {\rm SO}(N)$
  is given by replacing each diagonal block matrix
  $-\bm{I}_{2} \in {\rm SO}(2)$
  in
  $-\bm{I}_{k_{2}}$~[in~\eqref{eq:orthogonal_normal}]
  with
  $\bm{R}(\pi+1/n)$.
  From
  $\det(\bm{I}_{2}+\bm{R}(\pi+1/n))\neq 0$
  for
  $n\in \mathbb{N}$,
  we have
  $\bm{U}_{n} \in {\rm O}(N)\setminus E_{N,N}(\bm{S})$
  and
  $\lim_{n\to \infty} \bm{U}_{n} = \bm{U}$,
  which implies
  ${\rm O}(N)\setminus E_{N,N}(\bm{S})$
  is dense in
  ${\rm O}(N,\bm{S})$.
\end{proof}

\begin{lemma}[Matrix norms]\label{lemma:norm_basic}
  \mbox{}
  \begin{enumerate}[label=(\alph*)]
      \item \label{enum:norm_upper}
        For
        $\bm{A} \in \mathbb{R}^{l\times m}$
        and
        $\bm{B} \in \mathbb{R}^{m\times n}$,
        it holds
        $\|\bm{A}\bm{B}\|_{F} \leq \|\bm{A}\|_{2} \|\bm{B}\|_{F}$
        and
        $\|\bm{A}\bm{B}\|_{F} \leq \|\bm{A}\|_{F} \|\bm{B}\|_{2}$.
      \item \label{enum:IV_inv_norm}
        For
        $\bm{V} \in Q_{N,N}:=\{\bm{V} \in \mathbb{R}^{N\times N} \mid \bm{V}^{\T} = -\bm{V}\}$,
        we have
        $\sigma_{i}(\bm{I}+\bm{V}) \geq 1\ (1\leq i \leq N)$,
        $\|\bm{I}+\bm{V}\|_{2}^{2} =1+\|\bm{V}\|_{2}^{2}$
        and
        $\|(\bm{I}+\bm{V})^{-1}\|_{2} \leq 1$,
        where
        $\sigma_{i}(\cdot)$
        stands for the $i$th largest singular value of a given matrix.
      \item \label{enum:norm_Lipschitz}
        For
        $\bm{V}_{1},\bm{V}_{2} \in Q_{N,N}$,
        $\left\| (\bm{I}+\bm{V}_1)^{-1} - (\bm{I}+\bm{V}_2)^{-1} \right\|_{F} \leq \|\bm{V}_1 - \bm{V}_2 \|_{F}$.
      \item \label{enum:determinant_positive}
        $\det(\bm{I}+\bm{V}) > 0$
        for all
        $\bm{V} \in Q_{N,N}$.
      \item \label{enum:norm_determinant}
        For
        $\bm{V} \in Q_{N,N}$,
        it holds
        $\sqrt{1 + \|\bm{V}\|_{2}^{2}} \leq \det(\bm{I}+\bm{V})\leq (1 + \|\bm{V}\|_{2}^{2})^{N/2}$.
  \end{enumerate}
\end{lemma}
\begin{proof}\footnote{For readers' convenience, we present a complete proof.}
  \ref{enum:norm_upper}
  Let
  $\bm{b}_{i} \in \mathbb{R}^{m}$
  be the $i$th column vector of
  $\bm{B}$.
  Then, it holds
  \begin{align}
    \|\bm{A}\bm{B}\|_{F}^{2}
    = \sum_{i=1}^{n} \|\bm{A}\bm{b}_{i}\|^{2}
    = \sum_{\bm{b}_{i}\neq \bm{0}} \left\|\bm{A}\frac{\bm{b}_{i}}{\|\bm{b}_{i}\|}\right\|^{2} \|\bm{b}_{i}\|^{2}
    \leq \sum_{i=1}^{n} \|\bm{A}\|_{2}^{2} \|\bm{b}_{i}\|^{2}
    = \|\bm{A}\|_{2}^{2}\|\bm{B}\|_{F}^{2},
  \end{align}
  where
  $\|\cdot\|$
  stands for the Euclidean norm for vectors.
  Thus, we have
  $\|\bm{A}\bm{B}\|_{F} \leq \|\bm{A}\|_{2}\|\bm{B}\|_{F}$.
  By taking the transpose of
  $\bm{A}\bm{B}$
  in the previous inequality, we have
  $\|\bm{A}\bm{B}\|_{F} = \|\bm{B}^{\T}\bm{A}^{\T}\|_{F} \leq \|\bm{B}^{\T}\|_{2}\|\bm{A}^{\T}\|_{F} = \|\bm{A}\|_{F}\|\bm{B}\|_{2}$.

  \ref{enum:IV_inv_norm}
  For
  $1\leq i \leq N$,
  let
  $\lambda_{i}(\bm{Y})$
  be the $i$th largest eigenvalue of a symmetric matrix
  $\bm{Y}\in \mathbb{R}^{N\times N}$.
  Then, we have the expression
  $\sigma_{i}(\bm{I}+\bm{V})=\sqrt{\lambda_{i}\left((\bm{I}+\bm{V})^{\T}(\bm{I}+\bm{V})\right)} = \sqrt{\lambda_{i}(\bm{I}+\bm{V}^{\T}\bm{V})} = \sqrt{1+\sigma_{i}^{2}(\bm{V})} \geq 1\ (1\leq i \leq N)$,
  which asserts
  $\|\bm{I}+\bm{V}\|_{2}^{2} = \sigma_{1}^{2}(\bm{I}+\bm{V}) = 1 + \sigma_{1}^{2}(\bm{V}) = 1 + \|\bm{V}\|_{2}^{2}$
  and
  $\|(\bm{I}+\bm{V})^{-1}\|_{2} = \sigma_{N}^{-1}(\bm{I}+\bm{V}) \leq 1$.

  \ref{enum:norm_Lipschitz}
  By (a) and (b),
  $\|(\bm{I}+\bm{V}_1)^{-1}-(\bm{I}+\bm{V}_2)^{-1}\|_{F}
  = \|(\bm{I}+\bm{V}_1)^{-1}((\bm{I}+\bm{V}_2)-(\bm{I}+\bm{V}_1))(\bm{I}+\bm{V}_2)^{-1}\|_{F} 
  \leq   \|(\bm{I}+\bm{V}_1)^{-1}\|_{2} \|(\bm{I}+\bm{V}_2)^{-1} \|_{2} \|\bm{V}_1 - \bm{V}_2 \|_{F} 
  \leq \|\bm{V}_{1} - \bm{V}_{2}\|_{F}$.

  \ref{enum:determinant_positive}
  The nonsingularity of
  $\bm{I}+\bm{V}$
  (see~\ref{enum:IV_inv_norm})
  yields
  $\det(\bm{I}+\bm{V})\neq 0$,
  and
  $\det(\bm{I}+\bm{0}) = 1$.
  Since
  $\det(\bm{I} + \cdot)$
  is continuous and
  $Q_{N,N}$
  is connected,
  $\det(\bm{I}+\bm{V})$
  is a positive-valued.

  \ref{enum:norm_determinant}
  Let
  $\bm{I}+\bm{V}=\bm{Q}_{1}\bm{\Sigma}\bm{Q}_{2}^{\T}$
  be a singular value decomposition with
  $\bm{Q}_{1},\bm{Q}_{2} \in {\rm O}(N)$
  and a nonnegative diagonal matrix
  $\bm{\Sigma}\in \mathbb{R}^{N\times N}$.
  Then, we obtain
  $|\det(\bm{I}+\bm{V})| = |\det(\bm{Q}_{1}\bm{\Sigma}\bm{Q}_{2}^{\
  T})| = \det(\bm{\Sigma}) = \prod_{i=1}^{N} \sigma_{i}(\bm{I}+\bm{V})$,
  implying thus
  $\det(\bm{I}+\bm{V}) = \prod_{i=1}^{N} \sigma_{i}(\bm{I}+\bm{V})$
  by~\ref{enum:determinant_positive}.
  Moreover by~\ref{enum:IV_inv_norm}, we have
  $\det(\bm{I}+\bm{V}) \geq \sigma_{1}(\bm{I}+\bm{V}) = \|\bm{I}+\bm{V}\|_{2} = \sqrt{1+\|\bm{V}\|_{2}^{2}}$
  and
  $\det(\bm{I}+\bm{V}) \leq \sigma_{1}^{N}(\bm{I}+\bm{V}) = \|\bm{I}+\bm{V}\|_{2}^{N} = (1+\|\bm{V}\|_{2}^{2})^{N/2}$.
\end{proof}

\begin{fact}[Derivative of matrix functions (see, e.g.,~{\cite[Appendix~D]{matrix_differential}})]\label{fact:matrix_differential}
  Let
  $D \subset \mathbb{R}$
  be an open interval.
  Then, the following hold:
  \begin{enumerate}[label=(\alph*)]
    \item \label{enum:product_rule}
      Let
      $\bm{X}:\mathbb{R}\to \mathbb{R}^{N \times M}$
      and
      $\bm{Y}:\mathbb{R}\to \mathbb{R}^{M\times L}$
      be differentiable on
      $D$.
      Then,
      \begin{equation}
        \frac{d}{dt}\bm{X}(t)\bm{Y}(t) = \left(\frac{d}{dt}\bm{X}(t)\right) \bm{Y}(t) + \bm{X}(t) \left( \frac{d}{dt}\bm{Y}(t) \right).
      \end{equation}
    \item \label{enum:inverse_rule}
      Let
      $\bm{X}:\mathbb{R}\to\mathbb{R}^{N\times N}$
      be differentiable and invertible on
      $D$.
      Then,
      \begin{equation}
        \frac{d}{dt}\bm{X}^{-1}(t) = - \bm{X}^{-1}(t)\left(\frac{d}{dt}\bm{X}(t)\right)\bm{X}^{-1}(t).
      \end{equation}
    \end{enumerate}
\end{fact}

\begin{fact}[The Schur complement formula~{\cite[Sec.~0.8.5]{H12}}]~\label{fact:Schur}
  Let
  $\dbra{\bm{X}}_{22} \in \mathbb{R}^{(N-p)\times (N-p)}$
  be a nonsingular block matrix of
  $\bm{X}\in \mathbb{R}^{N\times N}$.
  Define a Schur complement matrix of
  $\bm{X}$
  by
  $\bm{M}:= \dbra{\bm{X}}_{11} - \dbra{\bm{X}}_{12}\dbra{\bm{X}}_{22}^{-1}\dbra{\bm{X}}_{21}$.
  Then,
  $\bm{M}$
  is nonsingular if and only if
  $\bm{X}$
  is nonsingular, and the inversion
  $\bm{X}^{-1}$
  can be expressed as
  \begin{equation}
    \bm{X}^{-1} =
    \begin{bmatrix}
      \bm{M}^{-1} & -\bm{M}^{-1}\dbra{\bm{X}}_{12}\dbra{\bm{X}}_{22}^{-1} \\
      -\dbra{\bm{X}}_{22}^{-1}\dbra{\bm{X}}_{21}\bm{M}^{-1} & \dbra{\bm{X}}_{22}^{-1}+\dbra{\bm{X}}_{22}^{-1}\dbra{\bm{X}}_{21}\bm{M}^{-1}\dbra{\bm{X}}_{12}\dbra{\bm{X}}_{22}^{-1}
    \end{bmatrix}.
  \end{equation}
  Moreover, it holds
  $\det(\bm{X}) = \det(\dbra{\bm{X}}_{22})\det(\bm{M})$.
\end{fact}

\begin{fact}[The Sherman-Morrison-Woodbury formular~{\cite[Sec.~0.7.4]{H12}}]~\label{fact:SMW}
  For nonsingular matrices
  $\bm{A} \in \mathbb{R}^{N\times N}$,
  $\bm{R} \in \mathbb{R}^{p\times p}$,
  and rectangular matrices
  $\bm{X} \in \mathbb{R}^{N\times p}$,
  $\bm{Y} \in \mathbb{R}^{p\times N}$,
  let
  $\bm{B} = \bm{A} + \bm{X}\bm{R}\bm{Y} \in \mathbb{R}^{N\times N}$.
  If
  $\bm{B}$
  and
  $\bm{R}^{-1} + \bm{Y}\bm{A}^{-1}\bm{X}$
  are nonsingular, then
  \begin{equation}
    \bm{B}^{-1} = \bm{A}^{-1} - \bm{A}^{-1}\bm{X}(\bm{R}^{-1} + \bm{Y}\bm{A}^{-1}\bm{X})^{-1}\bm{Y}\bm{A}^{-1}.
  \end{equation}
\end{fact}

\section{Retraction-based strategy for optimization over $\St(p,N)$ }\label{appendix:retraction}
We summarize a standard strategy for optimization over
$\St(p,N)$.
\begin{definition}[Retraction~\cite{manifold_book}]\label{definition:retraction}
  The set of mappings
  $R_{\bm{U}}:T_{\bm{U}}\St(p,N)\to \St(p,N):\bm{\mathcal{D}}\mapsto  R_{\bm{U}}(\bm{\mathcal{D}})$
  defined at each
  $\bm{U}\in \St(p,N)$
  is called a retraction of
  $\St(p,N)$
  if it satisfies (i)~$R_{\bm{U}}(\bm{0}) = \bm{U}$;
  (ii)~$\left.\frac{d}{dt}\right|_{t=0} R_{\bm{U}}(t\bm{\mathcal{D}}) = \bm{\mathcal{D}}$
  for all
  $\bm{U}\in \St(p,N)$
  and
  $\bm{\mathcal{D}}\in T_{\bm{U}}\St(p,N)$.
\end{definition}
Retractions serve as certain approximations of \textit{the exponential mapping}
$\mathop{\mathrm{Exp}}_{\bm{U}}$\footnote{
  The exponential mapping $\mathop{\mathrm{Exp}}_{\bm{U}}:T_{\bm{U}}\St(p,N)\to\St(p,N)$ at $\bm{U}\in\St(p,N)$ is defined as a mapping that assigns a given direction
  $\bm{\mathcal{D}} \in T_{\bm{U}}\St(p,N)$
  to a point on the geodesic of
  $\St(p,N)$
  with the initial velocity
  $\bm{\mathcal{D}}$.
  The exponential mapping is also a special instance of retractions of
  $\St(p,N)$.
  However, due to its high computational complexity,
  computationally simpler retractions have been used extensively for Problem~\ref{problem:origin}~\cite{manifold_book}.}.
Many examples of retractions for
$\St(p,N)$
are known, e.g., with QR decomposition, with polar decomposition and with the Euclidean projection~\cite{manifold_book,T08,A12} as well as with the Cayley transform~\cite{T08,W13}.

In the view that
$\St(p,N)$
is a Riemannian manifold, Problem~\ref{problem:origin} has been tackled with retractions as an application of the standard strategies for optimization defined over Riemannian manifold.
In such a strategy for
$\St(p,N)$
based on a retraction~\cite{manifold_book,E98,N02,N05,A07,T08,A12,R12,W13,H15,J15,M15,S15,Z17,K18}, the computation for updating the estimate
$\bm{U}_{n}\in \St(p,N)$
to
$\bm{U}_{n+1}\in\St(p,N)$
at
$n$th iteration is decomposed into:
(i) determine a search direction
$\bm{\mathcal{D}}_{n} \in T_{\bm{U}_{n}}\St(p,N)$;
(ii) assign
$R_{\bm{U}_{n}}(\bm{\mathcal{D}}_{n}) = R_{\bm{U}_{n}}(\bm{0}+\bm{\mathcal{D}}_{n}) \in \St(p,N)$
to a new estimate
$\bm{U}_{n+1}$.
Along this strategy, optimization algorithms designed originally over a single vector space have been extended to those designed over tangent spaces, to
$\St(p,N)$,
by using additional tools, e.g., {\it a vector transport}~\cite{manifold_book} and the inversion mapping of retractions~\cite{Zhu-Sato20}, if necessary.
Such extensions have been made for many schemes, e.g., the gradient descent method~\cite{E98,N02,N05,T08}, the conjugate gradient method~\cite{R12,S15,Z17,Zhu-Sato20}, Newton's method~\cite{E98,M15}, quasi-Newton's method~\cite{R12, H15}, the Barzilai–Borwein method~\cite{W13,J15} and the trust-region method~\cite{A07,K18}.

\section{Proof of Proposition~\ref{proposition:inverse}}\label{appendix:inverse}
The second equality in~\eqref{eq:Cayley_inv_alt} is verified by
$(\bm{I}-\bm{V})(\bm{I}+\bm{V})^{-1}
= (2\bm{I} - (\bm{I}+\bm{V}))(\bm{I}+\bm{V})^{-1}
= 2(\bm{I}+\bm{V})^{-1} - \bm{I}$.
Fact~\ref{fact:Schur} and
$\dbra{\bm{I}+\bm{V}}_{22} = \bm{I}_{N-p}$
guarantee the  non-singularity of
$\bm{M} := \bm{I}_{p} + \dbra{\bm{V}}_{11} + \dbra{\bm{V}}_{21}^{\T}\dbra{\bm{V}}_{21}$
and
\begin{equation}
  (\bm{I}+\bm{V})^{-1} =
  \begin{bmatrix}
    \bm{M}^{-1} & \bm{M}^{-1}\dbra{\bm{V}}_{21}^{\T} \\
    -\dbra{\bm{V}}_{21}\bm{M}^{-1} & \bm{I}_{N-p} - \dbra{\bm{V}}_{21}\bm{M}^{-1}\dbra{\bm{V}}_{21}^{\T}
  \end{bmatrix}\label{eq:IV_inv}
\end{equation}
which implies
  $(\bm{I}+\bm{V})^{-1}\bm{I}_{N\times p} = \begin{bmatrix} \bm{I}_{p} &-\dbra{\bm{V}}_{21}^{\T} \end{bmatrix}^{\T}\bm{M}^{-1}$
and the expressions of
$\Xi\circ\varphi_{\bm{S}}^{-1}$ 
in~\eqref{eq:Cayley_inv}.

In the following, we will show
$\Phi_{\bm{S}}^{-1}=\Upsilon_{\bm{S}}:=\Xi\circ\varphi_{\bm{S}}^{-1}$
on
$Q_{N,p}(\bm{S})$
by dividing 4 steps.

\underline{\bf (I) Proof of
$\Upsilon_{\bm{S}}(Q_{N,p}(\bm{S})):= \left\{\Upsilon_{\bm{S}}(\bm{V}) \mid \bm{V}\in Q_{N,p}(\bm{S}) \right\} \subset \St(p,N)\setminus E_{N,p}(\bm{S})$.}
For every
$\bm{V} \in Q_{N,p}(\bm{S})$,
\eqref{eq:commute} ensures
\begin{align}
  \Upsilon_{\bm{S}}(\bm{V})^{\T}\Upsilon_{\bm{S}}(\bm{V})
  & = \bm{I}_{N\times p}^{\T} (\bm{I}+\bm{V})^{-\T}(\bm{I}-\bm{V})^{\T}\bm{S}^{\T}\bm{S}(\bm{I}-\bm{V})(\bm{I}+\bm{V})^{-1}\bm{I}_{N\times p} \\
  & = \bm{I}_{N\times p}^{\T}(\bm{I}-\bm{V})^{-1}(\bm{I}+\bm{V})(\bm{I}+\bm{V})^{-1}(\bm{I}-\bm{V})\bm{I}_{N\times p}
  = \bm{I}_{p},
\end{align}
thus
$\Upsilon_{\bm{S}}(\bm{V}) \in \St(p,N)$.
$\Upsilon_{\bm{S}}(\bm{V}) \not\in E_{N,p}(\bm{S})$
is confirmed by the expression in~\eqref{eq:Cayley_inv}, i.e.,
\begin{align}
  & \bm{I}_{p} + \bm{S}_{\rm le}^{\T}\Upsilon_{\bm{S}}(\bm{V})
   = \bm{I}_{p} + \bm{S}_{\rm le}^{\T}(2(\bm{S}_{\rm le} - \bm{S}_{\rm ri}\dbra{\bm{V}}_{21})\bm{M}^{-1} - \bm{S}_{\rm le}) \\
  & = \bm{I}_{p} + 2\bm{M}^{-1} - \bm{I}_{p}    = 2\bm{M}^{-1}, \ (\because \bm{S}_{\rm le}^{\T}\bm{S}_{\rm le} =\bm{I}_{p} \ {\rm and}\ \bm{S}_{\rm le}^{\T}\bm{S}_{\rm ri} = \bm{0}\ \textrm{from}\ \bm{S}^{\T}\bm{S}=\bm{I})\label{eq:block_inverse}
\end{align}
and
$\det(\bm{I}_{p}+\bm{S}_{\rm le}^{\T}\Upsilon_{\bm{S}}(\bm{V})) = 2^p/\det(\bm{M}) \neq 0$.

\underline{\bf (II) Proof of
$\Upsilon_{\bm{S}} \circ \Phi_{\bm{S}}(\bm{U})=\bm{U}\ (\bm{U}\in \St(p,N)\setminus E_{N,p}(\bm{S}))$.}
Let
$\bm{U} \in \St(p,N)\setminus E_{N,p}(\bm{S})$
and
$\bm{V} := \Phi_{\bm{S}}(\bm{U})$
in~\eqref{eq:Cayley}.
Then, by
\begin{align}
  & \bm{I}_{p} + \bm{A}_{\bm{S}}(\bm{U}) + \bm{B}^{\T}_{\bm{S}}(\bm{U})\bm{B}_{\bm{S}}(\bm{U}) \\
  & = \bm{I}_{p} +  2(\bm{I}_{p}+\bm{S}_{\rm le}^{\T}\bm{U})^{-\T}\Skew(\bm{U}^{\T}\bm{S}_{\rm le})(\bm{I}_{p}+\bm{S}_{\rm le}^{\T}\bm{U})^{-1}
  + (\bm{I}_{p}+\bm{S}_{\rm le}^{\T}\bm{U})^{-\T}\bm{U}^{\T}\bm{S}_{\rm ri}\bm{S}_{\rm ri}^{\T}\bm{U}(\bm{I}_{p}+\bm{S}_{\rm le}^{\T}\bm{U})^{-1} \\
  & = (\bm{I}_{p}+\bm{S}_{\rm le}^{\T}\bm{U})^{-\T}\left( (\bm{I}_{p}+\bm{S}_{\rm le}^{\T}\bm{U})^{\T}(\bm{I}_{p}+\bm{S}_{\rm le}^{\T}\bm{U})
  + (\bm{U}^{\T}\bm{S}_{\rm le} -\bm{S}_{\rm le}^{\T}\bm{U}) + \bm{U}^{\T}\bm{S}_{\rm ri}\bm{S}_{\rm ri}^{\T}\bm{U}\right )(\bm{I}_{p}+\bm{S}_{\rm le}^{\T}\bm{U})^{-1} \\
  & = (\bm{I}_{p}+\bm{S}_{\rm le}^{\T}\bm{U})^{-\T}(\bm{I}_{p}+2\bm{U}^{\T}\bm{S}_{\rm le} + \bm{U}^{\T}\bm{S}_{\rm le}\bm{S}_{\rm le}^{\T}\bm{U} + \bm{U}^{\T}\bm{S}_{\rm ri}\bm{S}_{\rm ri}^{\T}\bm{U})(\bm{I}_{p}+\bm{S}_{\rm le}^{\T}\bm{U})^{-1}\\
  & = (\bm{I}_{p}+\bm{U}^{\T}\bm{S}_{\rm le})^{-1}(2\bm{I}_{p}+2\bm{U}^{\T}\bm{S}_{\rm le})(\bm{I}_{p}+\bm{S}_{\rm le}^{\T}\bm{U})^{-1}
  = 2(\bm{I}_{p}+\bm{S}_{\rm le}^{\T}\bm{U})^{-1},\ 
  (\because \bm{S}\bm{S}^{\T}= \bm{S}_{\rm le}\bm{S}_{\rm le}^{\T} + \bm{S}_{\rm ri}\bm{S}_{\rm ri}^{\T} =\bm{I})
\end{align}
we deduce with~\eqref{eq:Cayley_inv}
\begin{align}
  & \Upsilon_{\bm{S}}(\bm{V})
   = 2(\bm{S}_{\rm le} - \bm{S}_{\rm ri}\bm{B}_{\bm{S}}(\bm{U}))(\bm{I}_{p}+\bm{A}_{\bm{S}}(\bm{U})+\bm{B}^{\T}_{\bm{S}}(\bm{U})\bm{B}_{\bm{S}}(\bm{U}))^{-1} - \bm{S}_{\rm le} \\
  & = \left( \bm{S}_{\rm le} + \bm{S}_{\rm ri}\bm{S}_{\rm ri}^{\T}\bm{U}(\bm{I}_{p}+\bm{S}_{\rm le}^{\T}\bm{U})^{-1}\right )(\bm{I}_{p}+\bm{S}_{\rm le}^{\T}\bm{U}) - \bm{S}_{\rm le}
  = (\bm{S}_{\rm le}\bm{S}_{\rm le}^{\T}+\bm{S}_{\rm ri}\bm{S}_{\rm ri}^{\T})\bm{U}
   = \bm{U}.
\end{align}

\underline{\bf (III) Proof of
$\Phi_{\bm{S}}\circ \Upsilon_{\bm{S}}(\bm{V}) = \bm{V}\ (\bm{V}\in Q_{N,p}(\bm{S}))$.}
Let
$\bm{V} \in Q_{N,p}(\bm{S})$
and
$\bm{U}:= \Upsilon_{\bm{S}}(\bm{V}) \overset{\eqref{eq:Cayley_inv}}{=}  2(\bm{S}_{\rm le}-\bm{S}_{\rm ri}\dbra{\bm{V}}_{21})\bm{M}^{-1} - \bm{S}_{\rm le}$
with
$\bm{M}:=\bm{I}_{p}+\dbra{\bm{V}}_{11}+\dbra{\bm{V}}_{21}^{\T}\dbra{\bm{V}}_{21}$.
It suffices to show
$\bm{A}_{\bm{S}}\circ \Upsilon_{\bm{S}}(\bm{V}) = \dbra{\bm{V}}_{11}$
and
$\bm{B}_{\bm{S}}\circ \Upsilon_{\bm{S}}(\bm{V}) = \dbra{\bm{V}}_{21}$.
Then, by the definition of
$\Phi_{\bm{S}}$
in~\eqref{eq:Cayley},~\eqref{eq:Cay_A} and~\eqref{eq:Cay_B}, and by
\begin{align}
  \bm{S}_{\rm le}^{\T}\bm{U}
  & = \bm{S}_{\rm le}^{\T}\left( 2(\bm{S}_{\rm le} - \bm{S}_{\rm ri}\dbra{\bm{V}}_{21})\bm{M}^{-1} - \bm{S}_{\rm le}\right )
  = 2\bm{M}^{-1} - \bm{I}_{p}  \ (\because \bm{S}_{\rm le}^{\T}\bm{S}_{\rm le} =\bm{I}_{p}, \ \bm{S}_{\rm le}^{\T}\bm{S}_{\rm ri} = \bm{0}) \\
  \bm{S}_{\rm ri}^{\T}\bm{U}
  & = \bm{S}_{\rm ri}^{\T}\left( 2(\bm{S}_{\rm le} - \bm{S}_{\rm ri}\dbra{\bm{V}}_{21})\bm{M}^{-1} - \bm{S}_{\rm le}\right )
   = -2\dbra{\bm{V}}_{21}\bm{M}^{-1}.\ (\because \bm{S}_{\rm ri}^{\T}\bm{S}_{\rm ri} =\bm{I}_{N-p}, \bm{S}_{\rm ri}^{\T}\bm{S}_{\rm le} = \bm{0}),
\end{align}
each block matrix in~\eqref{eq:Cayley} can be evaluated as
\begin{align}
  & \bm{A}_{\bm{S}}(\bm{U})
  = (\bm{I}_{p}+2\bm{M}^{-1}-\bm{I}_{p})^{-\T}\left( (2\bm{M}^{-1}-\bm{I}_{p})^{\T}- (2\bm{M}^{-1}-\bm{I}_{p})\right )(\bm{I}_{p}+2\bm{M}^{-1}-\bm{I}_{p})^{-1} \\
  & = 2^{-1}\bm{M}^{\T}( \bm{M}^{-\T}- \bm{M}^{-1}  )\bm{M} = 2^{-1} (\bm{M}-\bm{M}^{\T}) \\
  & = 2^{-1}\left( (\bm{I}_{p}+\dbra{\bm{V}}_{11}+\dbra{\bm{V}}_{21}^{\T}\dbra{\bm{V}}_{21}) - (\bm{I}_{p}+\dbra{\bm{V}}_{11}^{\T}+\dbra{\bm{V}}_{21}^{\T}\dbra{\bm{V}}_{21})\right )
  = \dbra{\bm{V}}_{11} \ (\because \dbra{\bm{V}}_{11}^{\T} = -\dbra{\bm{V}}_{11}), \\
  & \bm{B}_{\bm{S}}(\bm{U})
  = - (-2\dbra{\bm{V}}_{21}\bm{M}^{-1})(\bm{I}_{p}+2\bm{M}^{-1}-\bm{I}_{p})^{-1}
  = 2\dbra{\bm{V}}_{21}\bm{M}^{-1}(2\bm{M}^{-1})^{-1}
  = \dbra{\bm{V}}_{21},
\end{align}
which implies
$\Phi_{\bm{S}}\circ \Upsilon_{\bm{S}}(\bm{V}) = \bm{V}$.

\underline{\bf (IV) Proof of diffeomorphism of $\Phi_{\bm{S}}$ and $\Phi_{\bm{S}}^{-1}$.}
From (II) and (III), we have seen
$\Phi_{\bm{S}}^{-1}=\Upsilon_{\bm{S}}$,
and both
$\Phi_{\bm{S}}$
and
$\Phi_{\bm{S}}^{-1}$
are homeomorphic between their domains and images, and consist of finite numbers of matrix additions, matrix multiplications and matrix inversions, which are all smooth.
Therefore,
$\Phi_{\bm{S}}$
and
$\Phi_{\bm{S}}^{-1}$
are diffeomorphic between their domains and images.

\section{Proof of Theorem~\ref{theorem:dense}}\label{appendix:dense}
\ref{enum:St_Xi_O}
From the definition of
$\varphi_{\bm{S}}^{-1}$
in~\eqref{eq:ILCT_O},
$\Phi_{\bm{S}}^{-1}$
is the restriction of
$\Xi \circ \varphi_{\bm{S}}^{-1}$
to
$Q_{N,p}(\bm{S})$,
which implies
$\St(p,N)\setminus E_{N,p}(\bm{S}) \overset{Prop.~\ref{proposition:inverse}}{=} \Phi_{\bm{S}}^{-1}(Q_{N,p}(\bm{S})) = \Xi \circ \varphi_{\bm{S}}^{-1}(Q_{N,p}(\bm{S}))$.
Thus, it suffices to show for every
$\bm{V} \in Q_{N,N}(\bm{S})$
that there exists
$\widehat{\bm{V}} \in Q_{N,p}(\bm{S})$
satisfying
$\Phi_{\bm{S}}^{-1}(\widehat{\bm{V}}) = \Xi \circ \varphi_{\bm{S}}^{-1}(\bm{V})$,
which is verified by the following lemma.
\begin{lemma} \label{lemma:translation_V}
  Let
  $\bm{S} \in {\rm O}(N)$
  and
  $\bm{V} = \begin{bmatrix} \bm{A}  & -\bm{B}^{\T} \\ \bm{B} & \bm{C} \end{bmatrix} \in Q_{N,N}(\bm{S})$
  with
  $\bm{A}\in Q_{p,p}$,
  $\bm{B}\in \mathbb{R}^{(N-p)\times p}$,
  and
  $\bm{C}\in Q_{N-p,N-p}$.
  Define
  \begin{equation}
    \widehat{\bm{V}}
    := \begin{bmatrix} \widehat{\bm{A}}:=\bm{A} - \bm{B}^{\T}(\bm{I}_{N-p}+\bm{C})^{-\T}\bm{C}(\bm{I}_{N-p}+\bm{C})^{-1}\bm{B} & -\widehat{\bm{B}}^{\T} \\
      \widehat{\bm{B}}:=(\bm{I}_{N-p}+\bm{C})^{-1}\bm{B} & \bm{0}_{N-p}
    \end{bmatrix} \in \mathbb{R}^{N\times N}. \label{eq:Vhat}
  \end{equation}
  Then,
  $\widehat{\bm{V}} \in Q_{N,p}(\bm{S})$
  and
  $\Phi_{\bm{S}}^{-1}(\widehat{\bm{V}}) = \Xi\circ\varphi_{\bm{S}}^{-1}(\bm{V})$
\end{lemma}
\begin{proof}
  From the skew-symmetries of
  $\bm{A}$
  and
  $\bm{C}$,
  we have
  $\widehat{\bm{A}}^{\T} = \bm{A}^{\T} - \bm{B}^{\T}(\bm{I}_{N-p}+\bm{C})^{-\T}\bm{C}^{\T}(\bm{I}_{N-p}+\bm{C})^{-1}\bm{B} = -\bm{A} + \bm{B}^{\T}(\bm{I}_{N-p}+\bm{C})^{-\T}\bm{C}(\bm{I}_{N-p}+\bm{C})^{-1}\bm{B} = -\widehat{\bm{A}}$,
  thus
  $\widehat{\bm{V}} \in Q_{N,p}(\bm{S})$.

  By letting
  $\bm{M}:=\bm{I}_{p}+\bm{A}+\bm{B}^{\T}(\bm{I}_{N-p}+\bm{C})^{-1}\bm{B} \in \mathbb{R}^{p\times p}$,
  Fact~\ref{fact:Schur} yields
  \begin{equation}
    (\bm{I}+\bm{V})^{-1} =
    \begin{bmatrix}
      \bm{M}^{-1} & \bm{M}^{-1}\bm{B}^{\T}(\bm{I}_{N-p}+\bm{C})^{-1} \\
      -(\bm{I}_{N-p}+\bm{C})^{-1}\bm{B}\bm{M}^{-1} &
      (\bm{I}_{N-p}+\bm{C})^{-1} - (\bm{I}_{N-p}+\bm{C})^{-1}\bm{B}\bm{M}^{-1}\bm{B}^{\T}(\bm{I}_{N-p}+\bm{C})^{-1}
    \end{bmatrix}
  \end{equation}
  from the non-singularities of
  $\bm{I}+\bm{V}$
  and
  $\bm{I}_{N-p}+\bm{C}$
  (see Lemma~\ref{lemma:norm_basic}~\ref{enum:IV_inv_norm}).
  The expressions in~\eqref{eq:ILCT_O} and~\eqref{eq:inv_origin_Cayley} assert that
  \begin{align}
    & \Xi\circ\varphi_{\bm{S}}^{-1}(\bm{V})
    = \bm{S}(\bm{I}-\bm{V})(\bm{I}+\bm{V})^{-1}\bm{I}_{N\times p}
     = 2\bm{S}(\bm{I}+\bm{V})^{-1}\bm{I}_{N\times p} - \bm{S}\bm{I}_{N\times p}\\
    & = 2\bm{S}\begin{bmatrix} (\bm{I}_{p}+\bm{A}+\bm{B}^{\T}(\bm{I}_{N-p}+\bm{C})^{-1}\bm{B})^{-1} \\ -(\bm{I}_{N-p}+\bm{C})^{-1}\bm{B}(\bm{I}_{p}+\bm{A}+\bm{B}^{\T}(\bm{I}_{N-p}+\bm{C})^{-1}\bm{B})^{-1}\end{bmatrix} - \bm{S}\bm{I}_{N\times p}. \\
  \end{align}
  On the other hand, from~\eqref{eq:Cayley_inv}, we obtain
  \begin{equation}
    \Phi_{\bm{S}}^{-1}(\widehat{\bm{V}})
    = 2\bm{S} \begin{bmatrix}
      (\bm{I}_{p}+\widehat{\bm{A}} +\widehat{\bm{B}}^{\T}\widehat{\bm{B}})^{-1} \\
      - \widehat{\bm{B}} (\bm{I}_{p}+\widehat{\bm{A}} +\widehat{\bm{B}}^{\T}\widehat{\bm{B}})^{-1}
    \end{bmatrix} - \bm{S}\bm{I}_{N\times p}.
  \end{equation}
  Clearly to get
  $\Xi\circ\varphi_{\bm{S}}^{-1}(\bm{V}) = \Phi_{\bm{S}}^{-1}(\widehat{\bm{V}})$,
  it suffices to show
  $\bm{A}+\bm{B}^{\T}(\bm{I}_{N-p}+\bm{C})^{-1}\bm{B} = \widehat{\bm{A}} +\widehat{\bm{B}}^{\T}\widehat{\bm{B}}$
  because
  $(\bm{I}_{N-p}+\bm{C})^{-1}\bm{B} = \widehat{\bm{B}}$
  holds automatically by the definition of
  $\widehat{\bm{B}}$
  in~\eqref{eq:Vhat}.
  The equation
  $\bm{A}+\bm{B}^{\T}(\bm{I}_{N-p}+\bm{C})^{-1}\bm{B} = \widehat{\bm{A}} +\widehat{\bm{B}}^{\T}\widehat{\bm{B}}$
  is verified by
  $\bm{C}^{\T}=-\bm{C}$
  and by
  {
    \thickmuskip=0.0\thickmuskip
    \medmuskip=0.0\medmuskip
    \thinmuskip=0.0\thinmuskip
    \begin{align}
      & \widehat{\bm{A}} +\widehat{\bm{B}}^{\T}\widehat{\bm{B}}
      = \bm{A} - \bm{B}^{\T}(\bm{I}_{N-p}+\bm{C})^{-\T}\bm{C}(\bm{I}_{N-p}+\bm{C})^{-1}\bm{B} + \bm{B}^{\T}(\bm{I}_{N-p}+\bm{C})^{-\T}(\bm{I}_{N-p}+\bm{C})^{-1}\bm{B} \\
      & = \bm{A} +\bm{B}^{\T}(\bm{I}_{N-p}-\bm{C})^{-1}(\bm{I}_{N-p}-\bm{C})(\bm{I}_{N-p}+\bm{C})^{-1}\bm{B}
      = \bm{A} + \bm{B}^{\T}(\bm{I}_{N-p}+\bm{C})^{-1}\bm{B}.
    \end{align}
  }
\end{proof}
\ref{enum:dense}
(Openness)
By the continuity of
$g:\mathbb{R}^{N\times p}\to\mathbb{R}:\bm{X} \mapsto \det(\bm{I}_{p}+\bm{S}_{\rm le}^{\T}\bm{X})$,
the preimage
$g^{-1}(\{0\})$
is closed on
$\mathbb{R}^{N\times p}$.
Since
$E_{N,p}(\bm{S})=g^{-1}(\{0\}) \cap \St(p,N)$
is closed in
$\St(p,N)$,
$\St(p,N)\setminus E_{N,p}(\bm{S})$
is open in
$\St(p,N)$.

(Denseness)
It suffices to show, for every
$\bm{U} \in E_{N,p}(\bm{S})$,
there exists a sequence
$(\bm{U}_{n})_{n=0}^{\infty} \subset \St(p,N)\setminus E_{N,p}(\bm{S})$
such that
$\lim_{n\to \infty}\bm{U}_{n} = \bm{U}$.
Let
$\bm{U} = \Xi(\widetilde{\bm{U}}) \in E_{N,p}(\bm{S}) \subset \St(p,N)$
with
$\widetilde{\bm{U}}:= \begin{bmatrix} \bm{U} & \bm{U}_{\perp} \end{bmatrix} \in {\rm O}(N)$,
where
$\bm{U}_{\perp} \in \St(N-p,N)$
satisfies
$\bm{U}^{\T}\bm{U}_{\perp} = \bm{0}$.
Then,
$\Xi({\rm O}(N)\setminus E_{N,N}(\bm{S})) = \St(p,N)\setminus E_{N,p}(\bm{S})$
(see~\ref{enum:St_Xi_O})
ensures
$\widetilde{\bm{U}} \in E_{N,N}(\bm{S})$.
By using the denseness of
${\rm O}(N)\setminus E_{N,N}(\bm{S})$
in
${\rm O}(N,\bm{S})$ (see Fact~\ref{fact:dense}),
we can construct a sequence
$(\widetilde{\bm{U}}_{n})_{n=0}^{\infty} \subset {\rm O}(N)\setminus E_{N,N}(\bm{S})$
such that
$\lim_{n\to\infty} \widetilde{\bm{U}}_{n} = \widetilde{\bm{U}}$.
Moreover by defining
$(\bm{U}_{n})_{n=0}^{\infty} :=(\Xi(\widetilde{\bm{U}}_{n}))_{n=0}^{\infty} \subset \Xi({\rm O}(N)\setminus E_{N,N}(\bm{S})) \overset{\ref{enum:St_Xi_O}}{=} \St(p,N)\setminus E_{N,p}(\bm{S})$,
the continuity of
$\Xi$
yields
$\lim_{n\to\infty}\bm{U}_{n} = \lim_{n\to\infty}\Xi(\widetilde{\bm{U}}_{n}) = \Xi(\widetilde{\bm{U}}) = \bm{U}$.

\ref{enum:intersection_dense}
$\St(p,N) \setminus E_{N,p}(\bm{S}_{i})\ (i=1,2)$
are open dense subsets of
$\St(p,N)$
from Theorem~\ref{theorem:dense}~\ref{enum:dense}.
The openness of
$\Delta(\bm{S}_{1},\bm{S}_{2})$
is clear.
To show the denseness of
$\Delta(\bm{S}_{1},\bm{S}_{2})$
in
$\St(p,N)$,
choose
$\bm{U} \in \St(p,N)$
and
$\epsilon > 0$
arbitrarily.
By the open denseness of
$\St(p,N)\setminus E_{N,p}(\bm{S}_{1})$,
there exist
$\bm{U}_{1} \in B_{\St(p,N)}(\bm{U},\epsilon) \cap \St(p,N)\setminus E_{N,p}(\bm{S}_{1})$
and
$\epsilon_{1} > 0$
satisfying
$B_{\St(p,N)}(\bm{U}_{1},\epsilon_{1}) \subset B_{\St(p,N)}(\bm{U},\epsilon) \cap \St(p,N)\setminus E_{N,p}(\bm{S}_{1})$,
where
$B_{\St(p,N)}(\bm{U},\epsilon) := B_{\mathbb{R}^{N\times p}}(\bm{U},\epsilon) \cap \St(p,N)$.
The denseness of
$\St(p,N) \setminus E_{N,p}(\bm{S}_{2})$
in
$\St(p,N)$
yields the existence of
$\bm{U}_{2} \in B_{\St(p,N)}(\bm{U}_{1},\epsilon_{1}) \cap \St(p,N)\setminus E_{N,p}(\bm{S}_{2})$,
from which we obtain
$\bm{U}_{2} \in B_{\St(p,N)}(\bm{U}_{1},\epsilon_{1}) \cap \St(p,N)\setminus E_{N,p}(\bm{S}_{2}) \subset B_{\St(p,N)}(\bm{U},\epsilon) \cap \St(p,N)\setminus E_{N,p}(\bm{S}_{1}) \cap \St(p,N) \setminus E_{N,p}(\bm{S}_{2}) = B_{\St(p,N)}(\bm{U},\epsilon) \cap \Delta(\bm{S}_{1},\bm{S}_{2})$.

\ref{enum:characterize_singular}
From~\eqref{eq:block_inverse}, we have
$\bm{I}_{p}+\bm{S}_{\rm le}^{\T}\Phi_{\bm{S}}^{-1}(\bm{V}) = 2\bm{M}^{-1}$
for
$\bm{V}\in Q_{N,p}(\bm{S})$,
where
$\bm{M}:=\bm{I}_{p}+\dbra{\bm{V}}_{11}+\dbra{\bm{V}}_{21}^{\T}\dbra{\bm{V}}_{21} \in \mathbb{R}^{p\times p}$
is the Schur complement matrix of
$\bm{I}+\bm{V}\in \mathbb{R}^{N\times N}$.
Fact~\ref{fact:Schur} yields
$g(\bm{V}) = \det(2\bm{M}^{-1}) = 2^{p}(\det(\bm{M}))^{-1} = 2^{p}(\det(\bm{I}+\bm{V}))^{-1}$
due to
$\dbra{\bm{I}+\bm{V}}_{22} = \bm{I}_{N-p}$.
Lemma~\ref{lemma:norm_basic}~\ref{enum:determinant_positive} ensures
$g(\bm{V})>0$.
By Lemma~\ref{lemma:norm_basic}~\ref{enum:norm_determinant}, we have
$\det(\bm{I}+\bm{V}) \geq \sqrt{1 + \|\bm{V}\|_{2}^{2}} \to \infty$
as
$\|\bm{V}\|_{2} \to \infty$,
implying thus
$\lim_{\substack{\bm{V}\in Q_{N,p}(\bm{S}) \\ \|\bm{V}\|_{2}\to\infty}} g(\bm{V}) = 0$.

Assume that
$(\bm{V}_{n})_{n=0}^{\infty} \subset Q_{N,p}(\bm{S})$
satisfies
$\lim_{n\to \infty} g(\bm{V}_{n}) = 0$.
By
$0 < \det(\bm{I}+\bm{V}_{n}) \leq (1+\|\bm{V}_{n}\|_{2}^{2})^{N/2}$
in Lemma~\ref{lemma:norm_basic}~\ref{enum:norm_determinant}, we have
$g(\bm{V}_{n}) = 2^{p}(\det(\bm{I}+\bm{V}_{n}))^{-1} \geq 2^{p}/(1+\|\bm{V}_{n}\|_{2}^{2})^{N/2}$.
The assumption asserts
$\|\bm{V}_{n}\|_{2}\to \infty$
as
$n\to \infty$.

\section[On the choice of selection for the inversion of G-L$^2$CT in Proposition~\ref{proposition:inverse}]{On the choice of $\Xi:{\rm O}(N)\to \St(p,N)$ for $\Phi_{\bm{S}}^{-1}$ in Proposition~\ref{proposition:inverse}} \label{appendix:how_to_choice_projection}
For
$2p < N$,
let
$\bm{\mathfrak{U}} \in \St(p,N)$
and
$\bm{\mathfrak{U}}_{\perp} \in \St(N-p,N)$
satisfy
$\bm{\mathfrak{U}}^{\T}\bm{\mathfrak{U}}_{\perp} = \bm{0}$,
and
$\bm{\mathfrak{S}}:= [\bm{\mathfrak{U}}\ \bm{\mathfrak{U}}_{\perp}] \in {\rm O}(N)$.
From
$\bm{\mathfrak{U}} = \bm{\mathfrak{S}}\bm{I}_{N\times p}$,
we have
\begin{align}
  & (\bm{V} \in Q_{N,p}) \quad
  \Xi_{\langle\bm{\mathfrak{U}}\rangle}\circ \varphi^{-1}_{\bm{S}}(\bm{V})
  = \bm{S}(\bm{I}-\bm{V})(\bm{I}+\bm{V})^{-1}\bm{\mathfrak{S}}\bm{I}_{N\times p}
  = \bm{S}(\bm{I}-\bm{V})\bm{\mathfrak{S}}(\bm{I}+\bm{\mathfrak{S}}^{\T}\bm{V}\bm{\mathfrak{S}})^{-1}\bm{I}_{N\times p} \\
  & = \bm{S}\bm{\mathfrak{S}}(\bm{I}-\bm{\mathfrak{S}}^{\T}\bm{V}\bm{\mathfrak{S}})(\bm{I}+\bm{\mathfrak{S}}^{\T}\bm{V}\bm{\mathfrak{S}})^{-1}\bm{I}_{N\times p}
  = \Xi\circ\varphi_{\bm{S}\bm{\mathfrak{S}}}^{-1}(\bm{\mathfrak{S}}^{\T}\bm{V}\bm{\mathfrak{S}}). \label{eq:Xi_U}
\end{align}
From
$\bm{\mathfrak{S}}^{\T}\bm{V}\bm{\mathfrak{S}} \in Q_{N,N}$,
Theorem~\ref{theorem:dense}~\ref{enum:St_Xi_O} ensures
$\Xi_{\langle\bm{\mathfrak{U}}\rangle}\circ \varphi^{-1}_{\bm{S}}(Q_{N,p}) \subset \Xi\circ\varphi_{\bm{S}\bm{\mathfrak{S}}}^{-1}(Q_{N,N}) = \St(p,N) \setminus E_{N,p}(\bm{S}\bm{\mathfrak{S}})$.

In the following, let us consider the case of
$\bm{\mathfrak{U}}_{\rm up} = \bm{0} \in \mathbb{R}^{p\times p}$
to show that
$\Xi_{\langle\bm{\mathfrak{U}}\rangle}\circ \varphi^{-1}_{\bm{S}}$
is not injective on
$Q_{N,p}$.
Since
$\Xi_{\langle\bm{\mathfrak{U}}\rangle}\circ \varphi^{-1}_{\bm{S}}$
does not depend on
$\bm{\mathfrak{U}}_{\perp}$,
we can assume, without loss of generality,
$\bm{\mathfrak{S}} = \begin{bmatrix} \bm{0} & \bm{I}_{p} & \bm{0} \\ \bm{Z} & \bm{0} & \bm{Z}_{\perp} \end{bmatrix}$,
$\bm{\mathfrak{U}} = \begin{bmatrix} \bm{0} \\ \bm{Z} \end{bmatrix}$
and
$\bm{\mathfrak{U}}_{\perp} = \begin{bmatrix} \bm{I}_{p} & \bm{0} \\ \bm{0} & \bm{Z}_{\perp} \end{bmatrix}$
with
$\bm{Z} \in \St(p,N-p)$
and
$\bm{Z}_{\perp} \in \St(N-2p,N-p)$
satisfying
$\bm{Z}^{\T}\bm{Z}_{\perp} = \bm{0}$.
We have
\begin{align}
  & (\bm{V}\in Q_{N,p}) \quad
  \bm{\mathfrak{S}}^{\T}\bm{V}\bm{\mathfrak{S}}
  = \begin{bmatrix} \bm{0} & \bm{Z}^{\T} \\ \bm{I}_{p} & \bm{0} \\ \bm{0} & \bm{Z}_{\perp}^{\T} \end{bmatrix}
  \begin{bmatrix} \dbra{\bm{V}}_{11} & -\dbra{\bm{V}}_{21}^{\T} \\ \dbra{\bm{V}}_{21} & \bm{0} \end{bmatrix}
  \begin{bmatrix} \bm{0} & \bm{I}_{p} & \bm{0} \\ \bm{Z} & \bm{0} & \bm{Z}_{\perp} \end{bmatrix} \\
   & = \begin{bmatrix} \bm{Z}^{\T}\dbra{\bm{V}}_{21} & \bm{0} \\ \dbra{\bm{V}}_{11} & -\dbra{\bm{V}}_{21}^{\T} \\ \bm{Z}_{\perp}^{\T}\dbra{\bm{V}}_{21} & \bm{0} \end{bmatrix}
  \begin{bmatrix} \bm{0} & \bm{I}_{p} & \bm{0} \\ \bm{Z} & \bm{0} & \bm{Z}_{\perp} \end{bmatrix}
  = \begin{bmatrix}
    \bm{0} & \bm{Z}^{\T}\dbra{\bm{V}}_{21} & \bm{0} \\
    -\dbra{\bm{V}}_{21}^{\T}\bm{Z} & \dbra{\bm{V}}_{11} & -\dbra{\bm{V}}_{21}^{\T}\bm{Z}_{\perp} \\
    \bm{0} & \bm{Z}_{\perp}^{\T}\dbra{\bm{V}}_{21} & \bm{0}
  \end{bmatrix}. \label{eq:Y^TVY}
\end{align}

Now, by using
$\alpha \in \mathbb{R}\setminus \{0\}$,
define
$\bm{V}(\alpha) \in Q_{N,p}$
as
$\dbra{\bm{V}(\alpha)}_{11} = \bm{0}$
and
$\dbra{\bm{V}(\alpha)}_{21} = \alpha\begin{bmatrix} \bm{0}_{(N-p)\times p} & \bm{Z}_{\perp} \end{bmatrix} \begin{bmatrix} \bm{0}_{p\times (N-2p)} & \bm{I}_{p} \end{bmatrix}^{\T}$,
where
$\dbra{\bm{V}(\alpha)}_{21} \neq \bm{0}$
is guaranteed by
$\bm{Z}_{\perp} \in \St(N-2p,N-p)$
and
$0 < N-2p$.
Then,
$\bm{Z}^{\T}\dbra{\bm{V}(\alpha)}_{21} = \bm{0}$
and~\eqref{eq:Y^TVY} with
$\bm{V} = \bm{V}(\alpha)$
yield
\begin{equation}
  \bm{\mathfrak{S}}^{\T}\bm{V}(\alpha)\bm{\mathfrak{S}}
  = \begin{bmatrix}
    \bm{0} & \bm{0} & \bm{0} \\
    \bm{0} & \bm{0} & -\dbra{\bm{V}(\alpha)}_{21}^{\T}\bm{Z}_{\perp} \\
    \bm{0} & \bm{Z}_{\perp}^{\T}\dbra{\bm{V}(\alpha)}_{21} & \bm{0}
  \end{bmatrix}
  =:
  \begin{bmatrix}
    \bm{A} & -\bm{B}^{\T} \\
    \bm{B} & \bm{C}
  \end{bmatrix} \in Q_{N,N},
  \label{eq:Y^TV(a)Y}
\end{equation}
where
$\bm{A} = \bm{0} \in Q_{p,p}$,
$\bm{B} = \bm{0} \in \mathbb{R}^{(N-p)\times p}$
and
$\bm{C} = \begin{bmatrix} \bm{0} & -\dbra{\bm{V}(\alpha)}_{21}^{\T}\bm{Z}_{\perp} \\ \bm{Z}_{\perp}^{\T}\dbra{\bm{V}(\alpha)}_{21} & \bm{0} \end{bmatrix} \in Q_{N-p,N-p}$.

Finally, by applying Lemma~\ref{lemma:translation_V} to~\eqref{eq:Y^TV(a)Y} and~\eqref{eq:Xi_U}, we deduce
$\Xi_{\langle\bm{\mathfrak{U}}\rangle}\circ \varphi^{-1}_{\bm{S}}(\bm{V}(\alpha)) \overset{\eqref{eq:Xi_U}}{=} \Xi\circ\varphi_{\bm{S}\bm{\mathfrak{S}}}^{-1}(\bm{\mathfrak{S}}^{\T}\bm{V}(\alpha)\bm{\mathfrak{S}}) \overset{\rm Lemma~\ref{lemma:translation_V}}{=} \Phi_{\bm{S}\bm{\mathfrak{S}}}^{-1}(\bm{0}) = \bm{S}\bm{\mathfrak{S}}\bm{I}_{N\times p} = \bm{S}\bm{\mathfrak{U}}$
for all
$\alpha \in \mathbb{R} \setminus\{0\}$.
This implies that infinitely many
$\bm{V}(\alpha) \in Q_{N,p}\ (\alpha \in \mathbb{R}\setminus \{0\})$
achieve
$\Xi_{\langle\bm{\mathfrak{U}}\rangle}\circ \varphi^{-1}_{\bm{S}}(\bm{V}(\alpha)) = \bm{S}\bm{\mathfrak{U}}$,
and clearly
$\Xi_{\langle\bm{\mathfrak{U}}\rangle}\circ \varphi^{-1}_{\bm{S}}$
is not injective.

\section{Proof of Proposition~\ref{proposition:gradient}}\label{appendix:gradient}
The differentiability of
$f\circ\Phi_{\bm{S}}^{-1}$
is verified by the differentiabilities of
$f$
and
$\Phi_{\bm{S}}^{-1}$.
Let
$\bm{V},\bm{D} \in Q_{N,p}(\bm{S})$.
From the chain rule, we obtain
\begin{equation}
  \left.\frac{d}{dt} (f\circ \Phi_{\bm{S}}^{-1})(\bm{V}+t\bm{D})\right|_{t=0}
    = \trace\left( \nabla f(\bm{U})^{\T}\left.\frac{d}{dt}\Phi_{\bm{S}}^{-1}(\bm{V}+t\bm{D})\right|_{t=0}\right).\label{eq:grad_element}
\end{equation}
Moreover, by
$\Phi_{\bm{S}}^{-1}(\bm{V}) = 2\bm{S}(\bm{I}+\bm{V})^{-1}\bm{I}_{N\times p} - \bm{S}\bm{I}_{N\times p}$
and Fact~\ref{fact:matrix_differential},
we deduce
\begin{align}
  \left.\frac{d}{dt} \Phi_{\bm{S}}^{-1}(\bm{V}+t\bm{D})\right|_{t=0}
  & = 2\bm{S}\left.\frac{d}{dt} ( \bm{I} + \bm{V} + t\bm{D})^{-1}\right|_{t=0}\bm{I}_{N\times p} \\
  & = -2\bm{S}(\bm{I}+\bm{V})^{-1}\bm{D}(\bm{I}+\bm{V})^{-1}\bm{I}_{N\times p}.
\end{align}
Therefore, we have
\begin{align}
  & \left.\frac{d}{dt} (f\circ \Phi_{\bm{S}}^{-1})(\bm{V}+t\bm{D})\right|_{t=0}
    = \trace( -2\nabla f(\bm{U})^{\T}\bm{S}(\bm{I}+\bm{V})^{-1}\bm{D}(\bm{I}+\bm{V})^{-1}\bm{I}_{N\times p}) \\
    & = \trace( -2(\bm{I}+\bm{V})^{-1}\bm{I}_{N\times p}\nabla f(\bm{U})^{\T}\bm{S}(\bm{I}+\bm{V})^{-1}\bm{D})
    = \trace(-2\overline{\bm{W}}^{f}_{\bm{S}}(\bm{V})\bm{D}),
\end{align}
where
$\overline{\bm{W}}^{f}_{\bm{S}}(\bm{V})$
is defined in~\eqref{eq:matrix_W_bar}.
Furthermore, we have
$\trace(-2\overline{\bm{W}}^{f}_{\bm{S}}(\bm{V})\bm{D}) = \trace(-2\Skew(\overline{\bm{W}}^{f}_{\bm{S}}(\bm{V}))\bm{D}) = \trace(-2\Skew(\bm{W}^{f}_{\bm{S}}(\bm{V}))\bm{D})$,
where the first equality follows by
$\trace(\bm{X}^{\T}\bm{D}) = \frac{1}{2}(\trace(\bm{X}^{\T}\bm{D}) + \trace(\bm{X}\bm{D}^{\T})) = \frac{1}{2}(\trace(\bm{X}\bm{D}) - \trace(\bm{X}\bm{D})) = 0$
for any symmetric matrix
$\bm{X} \in \mathbb{R}^{N\times N}$
and the second equality follows by~\eqref{eq:matrix_W} and
$\dbra{\bm{D}}_{22} = \bm{0}$.
Therefore, we obtain
\begin{equation} \label{eq:f_S_derivative_chain}
  (\bm{D}\in Q_{N,p}(\bm{S})) \quad \left.\frac{d}{dt} (f\circ \Phi_{\bm{S}}^{-1})(\bm{V}+t\bm{D})\right|_{t=0}
  = \trace(-2\Skew(\bm{W}^{f}_{\bm{S}}(\bm{V}))\bm{D}).
\end{equation}

On the other hand, by letting
$\nabla (f\circ\Phi_{\bm{S}}^{-1})(\bm{V}) \in Q_{N,p}(\bm{S})$
be the gradient of
$f\circ \Phi_{\bm{S}}^{-1}$
at
$\bm{V}$,
it follows
\begin{equation} \label{eq:f_S_derivative_direction}
  (\bm{D}\in Q_{N,p}(\bm{S})) \quad \left.\frac{d}{dt} (f\circ \Phi_{\bm{S}}^{-1})(\bm{V}+t\bm{D})\right|_{t=0}
    = \trace(\nabla (f\circ \Phi_{\bm{S}}^{-1})(\bm{V})^{\T}\bm{D}).
\end{equation}
By noting
$2\Skew(\bm{W}^{f}_{\bm{S}}(\bm{V})) \in Q_{N,p}(\bm{S})$,
\eqref{eq:f_S_derivative_chain} and~\eqref{eq:f_S_derivative_direction} imply
$\nabla (f\circ \Phi_{\bm{S}}^{-1})(\bm{V}) = -2\Skew(\bm{W}^{f}_{\bm{S}}(\bm{V}))^{\T} = 2\Skew(\bm{W}^{f}_{\bm{S}}(\bm{V}))$.
By applying~\eqref{eq:IV_inv} to~\eqref{eq:matrix_W_bar},
the expression~\eqref{eq:gradient_block_original} is derived as
{
  \thickmuskip=0.0\thickmuskip
  \medmuskip=0.0\medmuskip
  \thinmuskip=0.0\thinmuskip
  \begin{align}
    & \overline{\bm{W}}^{f}_{\bm{S}}(\bm{V})=(\bm{I}+\bm{V})^{-1}\bm{I}_{N\times p}\nabla f(\bm{U})^{\T}\bm{S}(\bm{I}+\bm{V})^{-1} \\
    & =
    \begin{bmatrix}
      \bm{M}^{-1}\\
      -\dbra{\bm{V}}_{21}\bm{M}^{-1}
    \end{bmatrix}
    \nabla f(\bm{U})^{\T}
    \begin{bmatrix}
      (\bm{S}_{\rm le}-\bm{S}_{\rm ri}\dbra{\bm{V}}_{21})\bm{M}^{-1} & (\bm{S}_{\rm le} -\bm{S}_{\rm ri}\dbra{\bm{V}}_{21})\bm{M}^{-1}\dbra{\bm{V}}_{21}^{\T} +\bm{S}_{\rm ri}
    \end{bmatrix}.
  \end{align}
}

By substituting
$\bm{V} = \bm{0}$
into~\eqref{eq:matrix_W_bar}, and by
$\Phi_{\bm{S}}^{-1}(\bm{0}) = \bm{S}\bm{I}_{N\times p} = \bm{S}_{\rm le}$,
we deduce
\begin{equation}
  (\bm{S} \in {\rm O}(N)) \quad \overline{\bm{W}}_{\bm{S}}^{f}(\bm{0}) = \bm{I}_{N\times p}\nabla f(\Phi_{\bm{S}}^{-1}(\bm{0}))^{\T}\bm{S} = \begin{bmatrix} \nabla f(\bm{S}_{\rm le})^{\T}\bm{S}_{\rm le} & \nabla f(\bm{S}_{\rm le})^{\T}\bm{S}_{\rm ri} \\ \bm{0} & \bm{0} \end{bmatrix} \overset{\eqref{eq:matrix_W}}{=} \bm{W}_{\bm{S}}^{f}(\bm{0}) \label{eq:W_W_bar}
\end{equation}
and
\begin{equation}
  \nabla f_{\bm{S}}(\bm{0})
  \overset{\eqref{eq:grad_propo}}{=}  2\Skew(\bm{W}_{\bm{S}}^{f}(\bm{0})) = \begin{bmatrix} \nabla f(\bm{S}_{\rm le})^{\T}\bm{S}_{\rm le}-\bm{S}_{\rm le}^{\T}\nabla f(\bm{S}_{\rm le}) & \nabla f(\bm{S}_{\rm le})^{\T}\bm{S}_{\rm ri} \\ -\bm{S}_{\rm ri}^{\T}\nabla f(\bm{S}_{\rm le}) & \bm{0} \end{bmatrix}.
\end{equation}

\section{Proof of Proposition~\ref{proposition:change_center}}\label{appendix:change_center}
\underline{\bf (I) Proof of Proposition~\ref{proposition:change_center}~\ref{enum:change_formula}.}
We need the following lemma to show Proposition~\ref{proposition:change_center}~\ref{enum:change_formula}.
Figure~\ref{fig:flow_chart_transformation} illustrates the relation between the following lemma and Proposition~\ref{proposition:change_center}~\ref{enum:change_formula}.
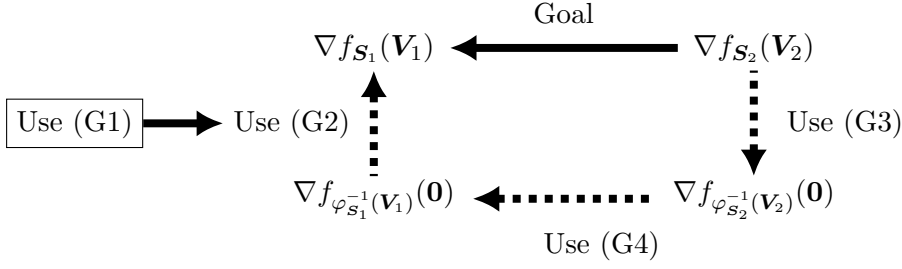
\begin{figure}[t]
  \renewcommand{\thefigure}{\Alph{section}}
  \begin{tikzpicture}
    \draw[line width=3pt, arrows = {Latex[width=10pt, length=10pt]-}, dashed] (1.3,0) -- (3.7,0);
    \draw[line width=3pt, arrows = {Latex[width=10pt, length=10pt]-}, dashed] (5,0.3) -- (5,1.7);
    \draw[line width=3pt, arrows = {-Latex[width=10pt, length=10pt]}] (4,2) -- (1,2);
    \draw[line width=3pt, arrows = {Latex[width=10pt, length=10pt]-}, dashed] (0,1.7) -- (0,0.3);
    \draw (2.5,2.2)node[above]{Goal};
    \draw (5.3,1)node[right, align=left]{Use~\eqref{eq:translate_Vto0}};
    \node[left=5pt, align=left] (a) at (0,1) {Use~\eqref{eq:translate_0toV}};
    \node[draw, left=30pt of a] (b) {Use~\eqref{eq:W_change}};
    \draw[line width=3pt, arrows = {Latex[width=10pt, length=10pt]-}] (a) -- (b);
    \draw (3,-0.3)node[below, align=left]{Use~\eqref{eq:translate_0to0}};
    \draw (0,0)node{$\nabla f_{\varphi_{\bm{S}_{1}}^{-1}(\bm{V}_{1})}(\bm{0})$};
    \draw (5,0)node{$\nabla f_{\varphi_{\bm{S}_{2}}^{-1}(\bm{V}_{2})}(\bm{0})$};
    \draw (5,2)node{$\nabla f_{\bm{S}_{2}}(\bm{V}_{2})$};
    \draw (0,2)node{$\nabla f_{\bm{S}_{1}}(\bm{V}_{1})$};
  \end{tikzpicture}

  \caption{A flow chart represents the overview of the proof of Proposition~\ref{proposition:change_center}~\ref{enum:change_formula}. The goal is to derive a transformation formula from $\nabla f_{\bm{S}_{2}}(\bm{V}_{2})$ to $\nabla f_{\bm{S}_{1}}(\bm{V}_{1})$
    under
$\Phi_{\bm{S}_{1}}^{-1}(\bm{V}_{1})= \Phi_{\varphi_{\bm{S}_{1}}^{-1}(\bm{V}_{1})}^{-1}(\bm{0})= \Phi_{\varphi_{\bm{S}_{2}}^{-1}(\bm{V}_{2})}^{-1}(\bm{0})=\Phi_{\bm{S}_{2}}^{-1}(\bm{V}_{2})$.}
  \label{fig:flow_chart_transformation}
\end{figure}

\begin{lemma} \label{lemma:translate_W}
  Let
  $f:\mathbb{R}^{N\times p}\to \mathbb{R}$
  be a differentiable function, and let
  $\bm{S} \in {\rm O}(N)$,
  $\bm{V} \in Q_{N,p}(\bm{S})$
  and
  $\bm{S}' := \varphi_{\bm{S}}^{-1}(\bm{V})\in {\rm O}(N)$
  in~\eqref{eq:ILCT_O}, implying thus
  $\Phi_{\bm{S}}^{-1}(\bm{V}) = \varphi_{\bm{S}}^{-1}(\bm{V})\bm{I}_{N\times p} = \bm{S}'\bm{I}_{N\times p} = \Phi_{\bm{S}'}^{-1}(\bm{0})$.
  Then, the following hold:
  \begin{enumerate}[label=(\alph*)]
    \item \label{enum:W_change}
      For
      $\bm{W}^{f}_{\bm{S}}(\bm{V})$
      in~\eqref{eq:matrix_W} and
      $\overline{\bm{W}}^{f}_{\bm{S}}(\bm{V})$
      in~\eqref{eq:matrix_W_bar},
      we have
      $\overline{\bm{W}}^{f}_{\bm{S}}(\bm{V}) = (\bm{I}+\bm{V})^{-1}\bm{W}^{f}_{\bm{S}'}(\bm{0})(\bm{I}+\bm{V})^{-T}$
      and
      \begin{equation} \label{eq:W_change}
        \bm{W}^{f}_{\bm{S}}(\bm{V}) = (\bm{I}+\bm{V})^{-1}\bm{W}^{f}_{\bm{S}'}(\bm{0})(\bm{I}+\bm{V})^{-\T}
        -
        \begin{bmatrix} \bm{0} & \bm{0} \\ \bm{0} & \bm{I}_{N-p} \end{bmatrix}
        (\bm{I}+\bm{V})^{-1}\bm{W}^{f}_{\bm{S}'}(\bm{0})(\bm{I}+\bm{V})^{-\T}
        \begin{bmatrix} \bm{0} & \bm{0} \\ \bm{0} & \bm{I}_{N-p} \end{bmatrix}.
      \end{equation}
    \item \label{enum:gradient_change}
      The gradients of
      $f_{\bm{S}}:=f\circ\Phi_{\bm{S}}^{-1}$
      and
      $f_{\bm{S}'}:=f\circ\Phi_{\bm{S}'}^{-1}$
      satisfy
      \begin{align}
        \nabla f_{\bm{S}}(\bm{V})
        & = (\bm{I}+\bm{V})^{-1}\nabla f_{\bm{S}'}(\bm{0})(\bm{I}+\bm{V})^{-\T} \\
        & \quad \quad \quad
          - \begin{bmatrix} \bm{0} & \bm{0} \\ \bm{0} & \bm{I}_{N-p} \end{bmatrix}
            (\bm{I}+\bm{V})^{-1}\nabla f_{\bm{S}'}(\bm{0})(\bm{I}+\bm{V})^{-\T}
            \begin{bmatrix} \bm{0} & \bm{0} \\ \bm{0} & \bm{I}_{N-p} \end{bmatrix} \label{eq:translate_0toV} \\
      \end{align}
      and
      \begin{align}
        \nabla f_{\bm{S}'}(\bm{0})
        & = (\bm{I}+\bm{V})
        \left(
          \nabla f_{\bm{S}}(\bm{V})
          - \begin{bmatrix} \bm{0} & \bm{0} \\ \dbra{\bm{V}}_{21} & \bm{I}_{N-p} \end{bmatrix}
            \nabla f_{\bm{S}}(\bm{V})
            \begin{bmatrix} \bm{0} & \dbra{\bm{V}}_{21}^{\T} \\ \bm{0} & \bm{I}_{N-p} \end{bmatrix}
        \right)
          (\bm{I}+\bm{V})^{\T}.\\ \label{eq:translate_Vto0}
      \end{align}
    \item \label{enum:gradient_change_0}
      If
      $\widehat{\bm{S}},\widecheck{\bm{S}}\in {\rm O}(N)$
      satisfy
      $\Phi_{\widehat{\bm{S}}}^{-1}(\bm{0})=\Phi_{\widecheck{\bm{S}}}^{-1}(\bm{0}) \in \St(p,N)$,
      i.e.,
      $\widehat{\bm{S}}_{\rm le} = \widecheck{\bm{S}}_{\rm le}$
      in~\eqref{eq:Cayley_inv_origin}, then we have
      \begin{equation} \label{eq:translate_0to0}
        \nabla f_{\widehat{\bm{S}}}(\bm{0})
        = \begin{bmatrix} \bm{I}_{p} & \bm{0} \\ \bm{0} & \bm{\mathfrak{Y}} \end{bmatrix}
        \nabla f_{\widecheck{\bm{S}}}(\bm{0})
        \begin{bmatrix} \bm{I}_{p} & \bm{0} \\ \bm{0} & \bm{\mathfrak{Y}}^{\T} \end{bmatrix}, 
      \end{equation}
      where
      $\bm{\mathfrak{Y}}:=\widehat{\bm{S}}_{\rm ri}^{\T}\widecheck{\bm{S}}_{\rm ri} \in {\rm O}(N-p)$.
  \end{enumerate}
\end{lemma}
\begin{proof}
  \ref{enum:W_change}
  Combining
  $\Phi_{\bm{S}}^{-1}(\bm{V}) = \Phi_{\bm{S}'}^{-1}(\bm{0})$
  and
  $\bm{W}^{f}_{\bm{S}'}(\bm{0}) \overset{\eqref{eq:W_W_bar}}{=} \overline{\bm{W}}^{f}_{\bm{S}'}(\bm{0}) \overset{\eqref{eq:matrix_W_bar}}{=} \bm{I}_{N\times p}\nabla f(\Phi_{\bm{S}'}^{-1}(\bm{0}))^{\T}\bm{S}' = \bm{I}_{N\times p}\nabla f(\Phi_{\bm{S}}^{-1}(\bm{V}))^{\T}\bm{S}'$,
  we obtain
  \begin{align}
    & (\bm{I}+\bm{V})^{-1}\bm{W}^{f}_{\bm{S}'}(\bm{0})(\bm{I}+\bm{V})^{-\T}
    = (\bm{I}+\bm{V})^{-1}\bm{I}_{N\times p}\nabla f(\Phi_{\bm{S}}^{-1}(\bm{V}))^{\T}\bm{S}'(\bm{I}^{\T}+\bm{V}^{\T})^{-1} \\
    & \overset{\eqref{eq:ILCT_O}}{=} (\bm{I}+\bm{V})^{-1}\bm{I}_{N\times p}\nabla f(\Phi_{\bm{S}}^{-1}(\bm{V}))^{\T}\bm{S}(\bm{I}-\bm{V})(\bm{I}+\bm{V})^{-1}(\bm{I}-\bm{V})^{-1} \\
    & \overset{\eqref{eq:commute}}{=} (\bm{I}+\bm{V})^{-1}\bm{I}_{N\times p}\nabla f(\Phi_{\bm{S}}^{-1}(\bm{V}))^{\T}\bm{S}(\bm{I}+\bm{V})^{-1}(\bm{I}-\bm{V})(\bm{I}-\bm{V})^{-1} \\
    & = (\bm{I}+\bm{V})^{-1}\bm{I}_{N\times p}\nabla f(\Phi_{\bm{S}}^{-1}(\bm{V}))^{\T}\bm{S}(\bm{I}+\bm{V})^{-1}
    \overset{\eqref{eq:matrix_W_bar}}{=} \overline{\bm{W}}^{f}_{\bm{S}}(\bm{V}). \label{eq:W_0_W_S}
  \end{align}
  The relation~\eqref{eq:W_change} is obtained by substituting~\eqref{eq:W_0_W_S} to an alternative expression of~\eqref{eq:matrix_W}:
  \begin{equation}
    \bm{W}^{f}_{\bm{S}}(\bm{V}) =
      \overline{\bm{W}}^{f}_{\bm{S}}(\bm{V})
    - \begin{bmatrix} \bm{0} & \bm{0} \\ \bm{0} & \bm{I}_{N-p} \end{bmatrix} \overline{\bm{W}}^{f}_{\bm{S}}(\bm{V}) \begin{bmatrix} \bm{0} & \bm{0} \\ \bm{0} & \bm{I}_{N-p} \end{bmatrix}.
  \end{equation}

  \ref{enum:gradient_change}
  \eqref{eq:translate_0toV} is confirmed by applying~\eqref{eq:W_change} to~\eqref{eq:grad_propo} as
  \begin{align}
    & \nabla f_{\bm{S}}(\bm{V})
    \overset{\eqref{eq:grad_propo}}{=}  \bm{W}^{f}_{\bm{S}}(\bm{V})-\bm{W}^{f}_{\bm{S}}(\bm{V})^{\T} \\
    & \overset{\eqref{eq:W_change}}{=} (\bm{I}+\bm{V})^{-1}(\bm{W}^{f}_{\bm{S}'}(\bm{0}) - \bm{W}^{f}_{\bm{S}'}(\bm{0})^{\T})(\bm{I}+\bm{V})^{-\T} \\
    & \quad \quad
    - \begin{bmatrix} \bm{0} & \bm{0} \\ \bm{0} & \bm{I}_{N-p} \end{bmatrix}
    (\bm{I}+\bm{V})^{-1}(\bm{W}^{f}_{\bm{S}'}(\bm{0}) - \bm{W}^{f}_{\bm{S}'}(\bm{0})^{\T})(\bm{I}+\bm{V})^{-\T}
    \begin{bmatrix} \bm{0} & \bm{0} \\ \bm{0} & \bm{I}_{N-p} \end{bmatrix}\\
    & \overset{\eqref{eq:grad_propo}}{=} (\bm{I}+\bm{V})^{-1}\nabla f_{\bm{S}'}(\bm{0})(\bm{I}+\bm{V})^{-\T}
    - \begin{bmatrix} \bm{0} & \bm{0} \\ \bm{0} & \bm{I}_{N-p} \end{bmatrix}
    (\bm{I}+\bm{V})^{-1}\nabla f_{\bm{S}'}(\bm{0})(\bm{I}+\bm{V})^{-\T}
    \begin{bmatrix} \bm{0} & \bm{0} \\ \bm{0} & \bm{I}_{N-p} \end{bmatrix}.
  \end{align}

  To derive~\eqref{eq:translate_Vto0} from~\eqref{eq:translate_0toV},
  let first
  $\nabla f_{\bm{S}'}(\bm{0})=
  \begin{bmatrix}
    \bm{E} \in Q_{p,p} & -\bm{F}^{\T}\\
    \bm{F} \in \mathbb{R}^{(N-p)\times p} & \bm{0}_{N-p}
  \end{bmatrix} \in Q_{N,p}(\bm{S}')$
  and apply~\eqref{eq:IV_inv} with
  $\bm{M}:= \bm{I}_{p} + \dbra{\bm{V}}_{11} + \dbra{\bm{V}}_{21}^{\T}\dbra{\bm{V}}_{21} \in \mathbb{R}^{p\times p}$
  as
  \begin{align}
    &(\bm{I}+\bm{V})^{-1}\nabla f_{\bm{S}'}(\bm{0})(\bm{I}+\bm{V})^{-\T} \\
    & =
    \begin{bmatrix}
      \bm{M}^{-1} & \bm{M}^{-1}\dbra{\bm{V}}_{21}^{\T} \\
      -\dbra{\bm{V}}_{21}\bm{M}^{-1} & \bm{I}_{N-p} - \dbra{\bm{V}}_{21}\bm{M}^{-1}\dbra{\bm{V}}_{21}^{\T}
    \end{bmatrix}
    \begin{bmatrix}
      \bm{E} & -\bm{F}^{\T}\\
      \bm{F} & \bm{0}
    \end{bmatrix}
    \begin{bmatrix}
      \bm{M}^{-\T} & -\bm{M}^{-\T}\dbra{\bm{V}}_{21}^{\T} \\
      \dbra{\bm{V}}_{21}\bm{M}^{-\T} & \bm{I}_{N-p} - \dbra{\bm{V}}_{21}\bm{M}^{-\T}\dbra{\bm{V}}_{21}^{\T}
    \end{bmatrix} \\
    & =
    \begin{bmatrix}
      \bm{M}^{-1}(\bm{E}+\dbra{\bm{V}}_{21}^{\T}\bm{F}) & -\bm{M}^{-1}\bm{F}^{\T} \\
      -\dbra{\bm{V}}_{21}\bm{M}^{-1}(\bm{E}+\dbra{\bm{V}}_{21}^{\T}\bm{F})+\bm{F} & \dbra{\bm{V}}_{21}\bm{M}^{-1}\bm{F}^{\T}
    \end{bmatrix}
    \begin{bmatrix}
      \bm{M}^{-\T} & -\bm{M}^{-\T}\dbra{\bm{V}}_{21}^{\T} \\
      \dbra{\bm{V}}_{21}\bm{M}^{-\T} & \bm{I}_{N-p} - \dbra{\bm{V}}_{21}\bm{M}^{-\T}\dbra{\bm{V}}_{21}^{\T}
    \end{bmatrix} \\
    & =
    \begin{bmatrix}
      \bm{M}^{-1}\bm{G}\bm{M}^{-\T} & -\bm{M}^{-1}\bm{G}\bm{M}^{-\T}\dbra{\bm{V}}_{21}^{\T} -\bm{M}^{-1}\bm{F}^{\T}\\
      -\dbra{\bm{V}}_{21}\bm{M}^{-1}\bm{G}\bm{M}^{-\T} + \bm{F}\bm{M}^{-\T} &
      \dbra{\bm{V}}_{21}\bm{M}^{-1}\bm{G}\bm{M}^{-\T}\dbra{\bm{V}}_{21}^{\T} + \dbra{\bm{V}}_{21}\bm{M}^{-1}\bm{F}^{\T}-\bm{F}\bm{M}^{-\T}\dbra{\bm{V}}_{21}^{\T}
    \end{bmatrix},\\\label{eq:grad_change_block}
  \end{align}
  where
  $\bm{G} := \bm{E}+\dbra{\bm{V}}_{21}^{\T}\bm{F} - \bm{F}^{\T}\dbra{\bm{V}}_{21} \in \mathbb{R}^{p\times p}$
  satisfies
  $\bm{G}^{\T} = -\bm{G}$.
  By substituting~\eqref{eq:grad_change_block} to~\eqref{eq:translate_0toV}, we obtain
  \begin{equation*}
    \nabla f_{\bm{S}}(\bm{V})
    =
    \begin{bmatrix}
      \bm{M}^{-1}\bm{G}\bm{M}^{-\T}  & -\bm{M}^{-1}\bm{G}\bm{M}^{-\T}\dbra{\bm{V}}_{21}^{\T} - \bm{M}^{-1}\bm{F}^{\T}\\
      -\dbra{\bm{V}}_{21}\bm{M}^{-1}\bm{G}\bm{M}^{-\T} + \bm{F}\bm{M}^{-\T} & \bm{0}_{N-p}
    \end{bmatrix}
  \end{equation*}
  and
  \begin{equation}
    -\begin{bmatrix} \bm{0} & \bm{0} \\ \dbra{\bm{V}}_{21} & \bm{I}_{N-p} \end{bmatrix}
    \nabla f_{\bm{S}}(\bm{V})
    \begin{bmatrix} \bm{0} & \dbra{\bm{V}}_{21}^{\T} \\ \bm{0} & \bm{I}_{N-p} \end{bmatrix}
    = \begin{bmatrix} \bm{0} & \bm{0} \\ \bm{0} & \dbra{(\bm{I}+\bm{V})^{-1}\nabla f_{\bm{S}'}(\bm{0})(\bm{I}+\bm{V})^{-\T}}_{22}
    \end{bmatrix},
  \end{equation}
  from which we obtain
  \begin{equation}
    \nabla f_{\bm{S}}(\bm{V}) =
    (\bm{I}+\bm{V})^{-1}\nabla f_{\bm{S}'}(\bm{0})(\bm{I}+\bm{V})^{-\T}
    +
    \begin{bmatrix} \bm{0} & \bm{0} \\ \dbra{\bm{V}}_{21} & \bm{I}_{N-p} \end{bmatrix}
    \nabla f_{\bm{S}}(\bm{V})
    \begin{bmatrix} \bm{0} & \dbra{\bm{V}}_{21}^{\T} \\ \bm{0} & \bm{I}_{N-p} \end{bmatrix}.
  \end{equation}

  \ref{enum:gradient_change_0}
  From
  $\widehat{\bm{S}}_{\rm le}=\Phi_{\widehat{\bm{S}}}^{-1}(\bm{0}) = \Phi_{\widecheck{\bm{S}}}^{-1}(\bm{0}) = \widecheck{\bm{S}}_{\rm le}=:\bm{U}$,
  and
  \begin{align}
    {\widehat{\bm{S}}}^{\T}\widecheck{\bm{S}}
    & = \begin{bmatrix}
      \bm{I}_{p} & \bm{0} \\
      \bm{0} & {{\widehat{\bm{S}}_{\rm ri}}}^{\T}\widecheck{\bm{S}}_{\rm ri}
      \end{bmatrix}
    = \begin{bmatrix}
      \bm{I}_{p} & \bm{0} \\
      \bm{0} & \bm{\mathfrak{Y}}
    \end{bmatrix}
      \in {\rm O}(N),
  \end{align}
  we see
  $\bm{\mathfrak{Y}}\in {\rm O}(N-p)$
  and
  $\widecheck{\bm{S}}_{\rm ri} = \widehat{\bm{S}}_{\rm ri}\bm{\mathfrak{Y}}$
  by
  $\widehat{\bm{S}}_{\rm ri}\bm{\mathfrak{Y}} = \widehat{\bm{S}}_{\rm ri}\widehat{\bm{S}}_{\rm ri}^{\T}\widecheck{\bm{S}}_{\rm ri} = (\bm{I}-\widehat{\bm{S}}_{\rm le}\widehat{\bm{S}}_{\rm le}^{\T})\widecheck{\bm{S}}_{\rm ri} = (\bm{I}-\widecheck{\bm{S}}_{\rm le}\widecheck{\bm{S}}_{\rm le}^{\T})\widecheck{\bm{S}}_{\rm ri} = \widecheck{\bm{S}}_{\rm ri}\widecheck{\bm{S}}_{\rm ri}^{\T}\widecheck{\bm{S}}_{\rm ri}=\widecheck{\bm{S}}_{\rm ri}$.

  Thus, it follows from
  $\widecheck{\bm{S}}_{\rm ri} = \widehat{\bm{S}}_{\rm ri}\bm{\mathfrak{Y}}$
  and
  $\bm{U} = \widehat{\bm{S}}_{\rm le} = \widecheck{\bm{S}}_{\rm le}$
  that
  \begin{align}
    \nabla f_{\widehat{\bm{S}}}(\bm{0})
    & \overset{\eqref{eq:gradient_0}}{=} \begin{bmatrix} \nabla f(\bm{U})^{\T}\widehat{\bm{S}}_{\rm le} - {\widehat{\bm{S}}_{\rm le}}^{\T}\nabla f(\bm{U}) & \nabla f(\bm{U})^{\T}\widehat{\bm{S}}_{\rm ri} \\
    -{\widehat{\bm{S}}_{\rm ri}}^{\T}\nabla f(\bm{U}) & \bm{0}\end{bmatrix} \\
                                                      & = \begin{bmatrix} \nabla f(\bm{U})^{\T}\widecheck{\bm{S}}_{\rm le} - {\widecheck{\bm{S}}_{\rm le}}^{\T}\nabla f(\bm{U}) & \nabla f(\bm{U})^{\T}\widecheck{\bm{S}}_{\rm ri}\bm{\mathfrak{Y}}^{\T} \\
    -\bm{\mathfrak{Y}}{\widecheck{\bm{S}}_{\rm ri}}^{\T}\nabla f(\bm{U}) & \bm{0}\end{bmatrix}  \\
    & = 
    \begin{bmatrix} \bm{I}_{p} & \bm{0} \\ \bm{0} & \bm{\mathfrak{Y}} \end{bmatrix}
    \begin{bmatrix} \nabla f(\bm{U})^{\T}\widecheck{\bm{S}}_{\rm le} - {\widecheck{\bm{S}}_{\rm le}}^{\T}\nabla f(\bm{U}) & \nabla f(\bm{U})^{\T}\widecheck{\bm{S}}_{\rm ri} \\
    -{\widecheck{\bm{S}}_{\rm ri}}^{\T}\nabla f(\bm{U}) & \bm{0}\end{bmatrix}
    \begin{bmatrix} \bm{I}_{p} & \bm{0} \\ \bm{0} & \bm{\mathfrak{Y}}^{\T} \end{bmatrix}\\
                               & \overset{\eqref{eq:gradient_0}}{=}
    \begin{bmatrix} \bm{I}_{p} & \bm{0} \\ \bm{0} & \bm{\mathfrak{Y}} \end{bmatrix}
      \nabla f_{\widecheck{\bm{S}}}(\bm{0})
      \begin{bmatrix} \bm{I}_{p} & \bm{0} \\ \bm{0} & \bm{\mathfrak{Y}}^{\T} \end{bmatrix}.
  \end{align}
\end{proof}

Return to the proof of Proposition~\ref{proposition:change_center}~\ref{enum:change_formula}.
Let
$\widehat{\bm{S}_{1}}:=\varphi_{\bm{S}_{1}}^{-1}(\bm{V}_{1})\in {\rm O}(N)$
and
$\widecheck{\bm{S}_{2}} :=\varphi_{\bm{S}_{2}}^{-1}(\bm{V}_{2})\in{\rm O}(N)$.
Since
$\widehat{\bm{S}_{1}}_{\rm le} = \Phi_{\widehat{\bm{S}_{1}}}^{-1}(\bm{0}) = \Phi_{\bm{S}_{1}}^{-1}(\bm{V}_{1}) = \Phi_{\bm{S}_{2}}^{-1}(\bm{V}_{2}) = \Phi_{\widecheck{\bm{S}_{2}}}^{-1}(\bm{0}) = \widehat{\bm{S}_{2}}_{\rm le}$,
Lemma~\ref{lemma:translate_W}~\ref{enum:gradient_change_0} implies
$\bm{\mathfrak{X}}= \widehat{\bm{S}_{1}}_{\rm ri}^{\T}\widecheck{\bm{S}_{2}}_{\rm ri} \in {\rm O}(N-p)$.
Moreover from Lemma~\ref{lemma:translate_W}, we have the relations
\begin{align}
  \nabla f_{\bm{S}_{1}}(\bm{V}_{1})
  & \overset{\eqref{eq:translate_0toV}}{=} (\bm{I}+\bm{V}_{1})^{-1}\nabla f_{\widehat{\bm{S}_{1}}}(\bm{0})(\bm{I}+\bm{V}_{1})^{-\T} \\
  & \quad \quad \quad
    - \begin{bmatrix} \bm{0} & \bm{0} \\ \bm{0} & \bm{I}_{N-p} \end{bmatrix}
    (\bm{I}+\bm{V}_{1})^{-1}\nabla f_{\widehat{\bm{S}_{1}}}(\bm{0})(\bm{I}+\bm{V}_{1})^{-\T}
      \begin{bmatrix} \bm{0} & \bm{0} \\ \bm{0} & \bm{I}_{N-p} \end{bmatrix}, \\
      \nabla f_{\widehat{\bm{S}_{1}}}(\bm{0})
  & \overset{\eqref{eq:translate_0to0}}{=} \begin{bmatrix} \bm{I}_{p} & \bm{0} \\ \bm{0} & \bm{\mathfrak{X}} \end{bmatrix}
  \nabla f_{\widecheck{\bm{S}_{2}}}(\bm{0})
  \begin{bmatrix} \bm{I}_{p} & \bm{0} \\ \bm{0} & \bm{\mathfrak{X}}^{\T} \end{bmatrix}, \\
    \nabla f_{\widecheck{\bm{S}_{2}}}(\bm{0})
  & \overset{\eqref{eq:translate_Vto0}}{=} (\bm{I}+\bm{V}_{2})
  \left(
    \nabla f_{\bm{S}_{2}}(\bm{V}_{2})
  - \begin{bmatrix} \bm{0} & \bm{0} \\ \dbra{\bm{V}_{2}}_{21}& \bm{I}_{N-p} \end{bmatrix}
  \nabla f_{\bm{S}_{2}}(\bm{V}_{2})
  \begin{bmatrix} \bm{0} & \dbra{\bm{V}_{2}}_{21}^{\T} \\ \bm{0} & \bm{I}_{N-p} \end{bmatrix}
  \right)
  (\bm{I}+\bm{V}_{2})^{\T}.\\ 
\end{align}
Finally by substituting the second and last relations into the first relation, we complete the proof.

\underline{\bf (II) Proof of Proposition~\ref{proposition:change_center}~\ref{enum:change_center} and~\ref{enum:zero}.}
From Proposition~\ref{proposition:change_center}~\ref{enum:change_formula}, Lemma~\ref{lemma:norm_basic}~\ref{enum:norm_upper} and~\ref{enum:IV_inv_norm}, we obtain
\begin{align}
  &\|\nabla f_{\bm{S}_{1}}(\bm{V}_{1})\|_{F} \\
  & = \left\|\mathcal{G}_{\bm{S}_{1},\bm{S}_{2}}(\bm{V}_{1},\bm{V}_{2})
      - \begin{bmatrix} \bm{0}&\bm{0}\\\bm{0}&\bm{I}_{N-p}\end{bmatrix}
      \mathcal{G}_{\bm{S}_{1},\bm{S}_{2}}(\bm{V}_{1},\bm{V}_{2})
      \begin{bmatrix} \bm{0}&\bm{0}\\\bm{0}&\bm{I}_{N-p}\end{bmatrix}\right\|_{F}
    \leq \|\mathcal{G}_{\bm{S}_{1},\bm{S}_{2}}(\bm{V}_{1},\bm{V}_{2})\|_{F} \\
  & \leq  \|(\bm{I}+\bm{V}_{1})^{-1}\|_{2}^{2}
    \begin{Vmatrix}\bm{I}_{p}&\bm{0}\\\bm{0}&\bm{\mathfrak{X}}\end{Vmatrix}_{2}^{2}
    \|\bm{I}+\bm{V}_{2}\|_{2}^{2}
    \left\|
    \nabla f_{\bm{S}_{2}}(\bm{V}_{2})-\begin{bmatrix}\bm{0}&\bm{0}\\ \dbra{\bm{V}_{2}}_{21}&\bm{I}_{N-p}\end{bmatrix} \nabla f_{\bm{S}_{2}}(\bm{V}_{2})\begin{bmatrix}\bm{0}&\dbra{\bm{V}_{2}}_{21}^{\T}\\ \bm{0}&\bm{I}_{N-p}\end{bmatrix}
  \right\|_{F} \\
  & \leq \|\bm{I}+\bm{V}_{2}\|_{2}^{2}
    \left\|\nabla f_{\bm{S}_{2}}(\bm{V}_{2})-\begin{bmatrix}\bm{0}&\bm{0}\\ \dbra{\bm{V}_{2}}_{21}&\bm{I}_{N-p}\end{bmatrix} \nabla f_{\bm{S}_{2}}(\bm{V}_{2})\begin{bmatrix}\bm{0}&\dbra{\bm{V}_{2}}_{21}^{\T}\\ \bm{0}&\bm{I}_{N-p}\end{bmatrix}\right\|_{F}\ (\because {\rm Lemma~\ref{lemma:norm_basic}}~\ref{enum:IV_inv_norm})\\
  & \leq \|\bm{I}+\bm{V}_{2}\|_{2}^{2}\left( 1 + \begin{Vmatrix} \bm{0}&\bm{0}\\ \dbra{\bm{V}_{2}}_{21}&\bm{I}_{N-p}\end{Vmatrix}_{2}^{2}\right)\|\nabla f_{\bm{S}_{2}}(\bm{V}_{2})\|_{F}
  \leq 2\|\bm{I}+\bm{V}_{2}\|_{2}^{2}\|\nabla f_{\bm{S}_{2}}(\bm{V}_{2})\|_{F},
\end{align}
where the last inequality is derived by
$\begin{Vmatrix} \bm{0}&\bm{0}\\ \dbra{\bm{V}_{2}}_{21}&\bm{I}_{N-p}\end{Vmatrix}_{2} = 1$
from the fact that each eigenvalue of a triangular matrix equals its diagonal entry.
Finally by applying Lemma~\ref{lemma:norm_basic}~\ref{enum:IV_inv_norm} again, we obtain Proposition~\ref{proposition:change_center}~\ref{enum:change_center}, which implies Proposition~\ref{proposition:change_center}~\ref{enum:zero}.

\section[Useful properties of the gradient after Cayley parametrization for optimization]{Useful properties of $\nabla (f\circ\Phi_{\bm{S}}^{-1})$ for optimization}\label{appendix:gradient_property}
The properties of
$\nabla (f\circ\Phi_{\bm{S}}^{-1})$
in the following Proposition~\ref{proposition:gradient_property} are useful in transplanting powerful computational arts designed for optimization over a vector space into the minimization of
$f\circ\Phi_{\bm{S}}^{-1}$
over
$Q_{N,p}(\bm{S})$.
Indeed, the Lipschitz continuity of the gradient is one of the commonly used assumptions in optimization over a vector space (see, e.g.,~\cite{Reddi-Hefny-Sra16,G16,Zeyuan18,Ward-Wu-Bottou,Chen-Liu-Sun-Hong19,Tatarenko-Touri17}).
The boundedness of the gradient is a key property for distributed optimization and stochastic optimization over a vector space (see, e.g.,~\cite{Ward-Wu-Bottou,Chen-Liu-Sun-Hong19,Tatarenko-Touri17}).
The variance bounded of the gradient is also commonly assumed in stochastic optimization over a vector space (see, e.g.,~\cite{G16,Zeyuan18,Ward-Wu-Bottou}).
\begin{proposition}[Bounds for gradient after Cayley parametrizaton] \label{proposition:gradient_property}
  Let
  $f:\mathbb{R}^{N\times p}\to\mathbb{R}$
  be continuously differentiable.
  Then, for any
  $\bm{S} \in {\rm O}(N)$,
  the following hold:
  \begin{enumerate}[label=(\alph*)]
    \item (Lipschitz continuity).
      \label{enum:Lipschitz}
      If
      \begin{equation}
        (\exists L> 0, \forall \bm{U}_1,\bm{U}_2 \in \St(p,N)) \quad \|\nabla f(\bm{U}_1)-\nabla f(\bm{U}_2)\|_{F} \leq L \|\bm{U}_1 - \bm{U}_2\|_{F}
      \end{equation}
      and
      $\mu \geq \max_{\bm{U}\in \St(p,N)} \|\nabla f(\bm{U})\|_{2}$,
      then the gradient of
      $f_{\bm{S}}:= f\circ\Phi_{\bm{S}}^{-1}$
      satisfies
      \begin{equation}
        (\forall \bm{V}_1,\bm{V}_2 \in Q_{N,p}(\bm{S})) \quad\| \nabla f_{\bm{S}}(\bm{V}_1) - \nabla f_{\bm{S}}(\bm{V}_2) \|_{F}
        \leq 4(\mu + L) \|\bm{V}_1-\bm{V}_2\|_{F}. \label{eq:prop:gradient_Lipschitz}
      \end{equation}
    \item (Boundedness).
      \label{enum:bounded}
      \begin{equation}
        (\bm{V}\in Q_{N,p}(\bm{S})) \quad \|\nabla f_{\bm{S}}(\bm{V})\|_{F} \leq 2\max_{\bm{U}\in \St(p,N)} \|\nabla f(\bm{U})\|_{F}. \label{eq:bounded_gradient}
      \end{equation}
    \item (Variance boundedness).
      \label{enum:variance_bounded}
      Suppose
      $f^{\xi}:\mathbb{R}^{N\times p} \to \mathbb{R}$
      is indexed with realizations of a random variable
      $\xi$
      and continuously differentiable for each realization.
      If there exists
      $\sigma \geq 0$
      and
      $f$
      satisfies
      \begin{align}
        (\bm{U}\in \St(p,N)) \quad
        \begin{cases}
          \mathbb{E}_{\xi} [f^{\xi}(\bm{U})]   = f(\bm{U}), \\
          \mathbb{E}_{\xi} [\nabla f^{\xi}(\bm{U})]  = \nabla f(\bm{U}), \\
          \mathbb{E}_{\xi}[\|\nabla f^{\xi}(\bm{U})- \nabla f(\bm{U})\|_{F}^{2}]  \leq \sigma^{2},
        \end{cases}
      \end{align}
      we have
      \begin{equation}
        (\bm{V}\in Q_{N,p}(\bm{S})) \quad \mathbb{E}_{\xi}[\|\nabla f^{\xi}_{\bm{S}}(\bm{V}) - \nabla f_{\bm{S}}(\bm{V})\|_{F}^{2}] \leq 4\sigma^{2}. \label{eq:bounded_variance_gradient}
      \end{equation}
  \end{enumerate}
\end{proposition}
\begin{proof} 
  The existence of the maximum of
  $\|\nabla f(\cdot)\|$
  over
  $\St(p,N)$
  is guaranteed by the compactness of
  $\St(p,N)$
  and the continuities of
  $\|\cdot\|$
  and
  $\nabla f$.
  We divide the proof of (a)-(c) as follows.
  Recall that
  $\overline{\bm{W}}^{f}_{\bm{S}}(\bm{V})$
  and
  $\bm{W}^{f}_{\bm{S}}(\bm{V})$
  for
  $\bm{S} \in {\rm O}(N)$
  were respectively defined as~\eqref{eq:matrix_W_bar} and~\eqref{eq:matrix_W}, and we have
  $\nabla f_{\bm{S}}(\bm{V}):= \nabla (f\circ\Phi_{\bm{S}}^{-1})(\bm{V}) = 2\Skew(\bm{W}_{\bm{S}}^{f}(\bm{V}))$
  (see Proposition~\ref{proposition:gradient}).
  In the following, we use properties of
  $\Skew$;
  (i)
  $\|\Skew(\bm{X})\|_{F}\leq \|\bm{X}\|_{F}$
  for
  $\bm{X} \in \mathbb{R}^{N\times N}$;
  (ii) the linearity of
  $\Skew$.

  \underline{\bf (I) Proof of Proposition~\ref{proposition:gradient_property}~\ref{enum:Lipschitz}.}
  First, we introduce a useful inequalities.
  \begin{lemma}[Lipschitz continuity of $\Phi_{\bm{S}}^{-1}$]\label{lemma:Lipschitz_inverse}
    For every
    $\bm{S} \in O(N)$,
    $\Phi_{\bm{S}}^{-1}$
    is Lipschitz continuous over
    $Q_{N,p}(\bm{S})$
    with a constant
    $2$, i.e.,
    \begin{equation}
      (\bm{V}_{1},\bm{V}_{2} \in Q_{N,p}(\bm{S})) \quad
      \| \Phi_{\bm{S}}^{-1}(\bm{V}_1) - \Phi_{\bm{S}}^{-1}(\bm{V}_2)\|_{F} \leq 2\|\bm{V}_1-\bm{V}_2 \|_{F}. \label{eq:Lipschitz_inverse}
    \end{equation}
  \end{lemma}
  \begin{proof}
    From~\eqref{eq:Cayley_inv_alt} and Lemma~\ref{lemma:norm_basic}~\ref{enum:norm_upper} and~\ref{enum:norm_Lipschitz}, we have
    \begin{align}
      & \| \Phi_{\bm{S}}^{-1}(\bm{V}_1) - \Phi_{\bm{S}}^{-1}(\bm{V}_2) \|_{F}
      = \|2\bm{S}\left((\bm{I}+\bm{V}_{1})^{-1} -(\bm{I}+\bm{V}_{2})^{-1}\right)\bm{I}_{N\times p}\|_{F} \\
      & \leq 2\|\bm{S}\|_{2}\|(\bm{I}+\bm{V}_{1})^{-1} -(\bm{I}+\bm{V}_{2})^{-1}\|_{F}\|\bm{I}_{N\times p}\|_{2}
      \leq 2\|(\bm{I}+\bm{V}_1)^{-1}-(\bm{I}+\bm{V}_2)^{-1} \|_{F}
      \leq 2\|\bm{V}_1-\bm{V}_2 \|_{F}.
    \end{align}
  \end{proof}
  Return to the proof of Proposition~\ref{proposition:gradient_property}~\ref{enum:Lipschitz}.
  From~\eqref{eq:grad_propo},~\eqref{eq:matrix_W} in Proposition~\ref{proposition:gradient}, we have
  \begin{align}
    & \|\nabla f_{\bm{S}}(\bm{V}_1) - \nabla f_{\bm{S}}(\bm{V}_2) \|_{F}
    = 2\|\Skew(\bm{W}^{f}_{\bm{S}}(\bm{V}_{1})- \bm{W}^{f}_{\bm{S}}(\bm{V}_{2}))\|_{F} \\
    & \leq 2\|\bm{W}^{f}_{\bm{S}}(\bm{V}_1) - \bm{W}^{f}_{\bm{S}}(\bm{V}_2)\|_{F}
    \leq 2\|\overline{\bm{W}}^{f}_{\bm{S}}(\bm{V}_1) - \overline{\bm{W}}^{f}_{\bm{S}}(\bm{V}_2)\|_{F}. \label{eq:gradient_Lipschitz_W}
  \end{align}
  Moreover, from~\eqref{eq:matrix_W_bar}, for all
  $\bm{V}_{1},\bm{V}_{2}\in Q_{N,p}(\bm{S})$
  with
  $\bm{U}_1 := \Phi_{\bm{S}}^{-1}(\bm{V}_1), \bm{U}_2 :=  \Phi_{\bm{S}}^{-1}(\bm{V}_2) \in \St(p,N)\setminus E_{N,p}(\bm{S})$,
  we deduce
  {
    \thickmuskip=0.0\thickmuskip
    \medmuskip=0.0\medmuskip
    \thinmuskip=0.0\thinmuskip
    \begin{align}
      & \|\overline{\bm{W}}^{f}_{\bm{S}}(\bm{V}_1) - \overline{\bm{W}}^{f}_{\bm{S}}(\bm{V}_2)\|_{F} \\
      & = \|(\bm{I}+\bm{V}_1)^{-1}\bm{I}_{N\times p}\nabla f(\bm{U}_1)^{\T}\bm{S}(\bm{I}+\bm{V}_1)^{-1} - (\bm{I}+\bm{V}_2)^{-1}\bm{I}_{N\times p}\nabla f(\bm{U}_2)^{\T}\bm{S}(\bm{I}+\bm{V}_2)^{-1}\|_{F} \\
      & \leq \|(\bm{I}+\bm{V}_1)^{-1}\bm{I}_{N\times p}\nabla f(\bm{U}_1)^{\T}\bm{S}(\bm{I}+\bm{V}_1)^{-1} - (\bm{I}+\bm{V}_2)^{-1}\bm{I}_{N\times p}\nabla f(\bm{U}_1)^{\T}\bm{S}(\bm{I}+\bm{V}_1)^{-1}\|_{F} \\
      & \quad + \|(\bm{I}+\bm{V}_2)^{-1}\bm{I}_{N\times p}\nabla f(\bm{U}_1)^{\T}\bm{S}(\bm{I}+\bm{V}_1)^{-1} - (\bm{I}+\bm{V}_2)^{-1}\bm{I}_{N\times p}\nabla f(\bm{U}_2)^{\T}\bm{S}(\bm{I}+\bm{V}_1)^{-1}\|_{F} \\
      & \quad + \|(\bm{I}+\bm{V}_2)^{-1}\bm{I}_{N\times p}\nabla f(\bm{U}_2)^{\T}\bm{S}(\bm{I}+\bm{V}_1)^{-1} - (\bm{I}+\bm{V}_2)^{-1}\bm{I}_{N\times p}\nabla f(\bm{U}_2)^{\T}\bm{S}(\bm{I}+\bm{V}_2)^{-1}\|_{F}. \\
      & \label{eq:Lipschitz_triangle}
    \end{align}
  }
  The first term in the right-hand side of~\eqref{eq:Lipschitz_triangle} can be bounded as
  {
    \thickmuskip=0.0\thickmuskip
    \medmuskip=0.0\medmuskip
    \thinmuskip=0.0\thinmuskip
    \begin{align}
      &\|(\bm{I}+\bm{V}_1)^{-1}\bm{I}_{N\times p}\nabla f(\bm{U}_1)^{\T}\bm{S}(\bm{I}+\bm{V}_1)^{-1} - (\bm{I}+\bm{V}_2)^{-1}\bm{I}_{N\times p}\nabla f(\bm{U}_1)^{\T}\bm{S}(\bm{I}+\bm{V}_1)^{-1}\|_{F} \\
      & = \left\| \left((\bm{I}+\bm{V}_{1})^{-1} - (\bm{I}+\bm{V}_{2})^{-1} \right) \bm{I}_{N\times p}\nabla f(\bm{U}_1)^{\T}\bm{S}(\bm{I}+\bm{V}_1)^{-1}\right\|_{F} \\
      & \leq \|\bm{I}_{N\times p}\|_{2}\|\nabla f(\bm{U}_{1})\|_{2}\|\bm{S}\|_{2}\|(\bm{I}+\bm{V}_1)^{-1}\|_{2} \|(\bm{I}+\bm{V}_{1})^{-1} - (\bm{I}+\bm{V}_{2})^{-1}\|_{F} \ (\because \textrm{Lemma~\ref{lemma:norm_basic}~\ref{enum:norm_upper}}) \\
      & \leq \|\nabla f(\bm{U}_{1})\|_{2}\|\bm{V}_{1}-\bm{V}_{2}\|_{F}
      \leq \mu\|\bm{V}_{1}-\bm{V}_{2}\|_{F}. \ (\because \textrm{Lemma~\ref{lemma:norm_basic}~\ref{enum:IV_inv_norm} \ and~\ref{enum:norm_Lipschitz}})
    \end{align}
  }%
  \noindent
  Similarly the last term in~\eqref{eq:Lipschitz_triangle} can be bounded above by
  $\mu \|\bm{V}_{1} - \bm{V}_{2}\|_{F}$.
  The second term in~\eqref{eq:Lipschitz_triangle} can be evaluated as
  \begin{align}
    &  \|(\bm{I}+\bm{V}_2)^{-1}\bm{I}_{N\times p}\nabla f(\bm{U}_1)^{\T}\bm{S}(\bm{I}+\bm{V}_1)^{-1} - (\bm{I}+\bm{V}_2)^{-1}\bm{I}_{N\times p}\nabla f(\bm{U}_2)^{\T}\bm{S}(\bm{I}+\bm{V}_1)^{-1}\|_{F} \\
    & =\|(\bm{I}+\bm{V}_{2})^{-1}\bm{I}_{N\times p}(\nabla f(\bm{U}_{1})-\nabla f(\bm{U}_{2}))^{\T}\bm{S}(\bm{I}+\bm{V}_1)^{-1}\|_{F} \\
    & \leq \|(\bm{I}+\bm{V}_{2})^{-1}\|_{2}\|\bm{I}_{N\times p}\|_{2}\|\bm{S}\|_{2}\|(\bm{I}+\bm{V}_{1})^{-1}\|_{2}\|\nabla f(\bm{U}_1) - \nabla f(\bm{U}_2) \|_{F} \quad (\because \textrm{Lemma~\ref{lemma:norm_basic}~\ref{enum:norm_upper}})\\
    & \leq \|\nabla f(\bm{U}_1) - \nabla f(\bm{U}_2) \|_{F} 
    \leq L\|\bm{U}_1 - \bm{U}_2\|_{F} \ (\because \textrm{Lemma~\ref{lemma:norm_basic}~\ref{enum:IV_inv_norm} and Lipschitz continuity of }\nabla f) \\[-1em]
    & = L\|\Phi_{\bm{S}}^{-1}(\bm{V}_1)- \Phi_{\bm{S}}^{-1}(\bm{V}_2)\|_{F}
    \leq 2L\|\bm{V}_1-\bm{V}_2\|_{F} \quad (\because \textrm{Lemma~\ref{lemma:Lipschitz_inverse}}).
  \end{align}
  Therefore, the left-hand side of~\eqref{eq:Lipschitz_triangle} is bounded as
  \begin{align}
    (\bm{V}_{1},\bm{V}_{2} \in Q_{N,p}(\bm{S}))\ \|\overline{\bm{W}}^{f}_{\bm{S}}(\bm{V}_1) - \overline{\bm{W}}^{f}_{\bm{S}}(\bm{V}_2)\|_{F}
    & \leq 2(\mu+L)\|\bm{V}_1 - \bm{V}_2\|_{F},
  \end{align}
  which is combined with~\eqref{eq:gradient_Lipschitz_W} to get~\eqref{eq:prop:gradient_Lipschitz}.

  \underline{\bf (II) Proof of Proposition~\ref{proposition:gradient_property}~\ref{enum:bounded}.}
  From~\eqref{eq:grad_propo},~\eqref{eq:matrix_W} in Proposition~\ref{proposition:gradient}, we have
  \begin{align}
    \|\nabla f_{\bm{S}}(\bm{V})\|_{F}
    & = 2\|\Skew(\bm{W}^{f}_{\bm{S}}(\bm{V}))\|_{F}
    \leq 2\|\bm{W}^{f}_{\bm{S}}(\bm{V})\|_{F}
    \leq 2\|\overline{\bm{W}}^{f}_{\bm{S}}(\bm{V})\|_{F}.
  \end{align}
  By using Lemma~\ref{lemma:norm_basic}~\ref{enum:norm_upper} and~\ref{enum:IV_inv_norm}, we get
  \begin{align}
    & \|\overline{\bm{W}}^{f}_{\bm{S}}(\bm{V}) \|_{F}
    = \|(\bm{I}+\bm{V})^{-1}\bm{I}_{N\times p}\nabla f(\Phi_{\bm{S}}^{-1}(\bm{V}))^{\T}\bm{S}(\bm{I}+\bm{V})^{-1}\|_{F} \\
    & \leq \|(\bm{I}+\bm{V})^{-1}\|_{2}^{2}\|\bm{I}_{N\times p}\|_{2}\|\bm{S}\|_{2}\|\nabla f(\Phi_{\bm{S}}^{-1}(\bm{V}))\|_{F}
    \leq \|\nabla f(\Phi_{\bm{S}}^{-1}(\bm{V}))\|_{F}
    \leq \max_{\bm{U}\in \St(p,N)} \|\nabla f(\bm{U})\|_{F},
  \end{align}
  which implies~\eqref{eq:bounded_gradient}.

  \underline{\bf (III) Proof of Proposition~\ref{proposition:gradient_property}~\ref{enum:variance_bounded}.}
  From~\eqref{eq:grad_propo},~\eqref{eq:matrix_W} in Proposition~\ref{proposition:gradient}, we obtain, for each
  $\xi$,
  \begin{align}
    & \|\nabla f_{\bm{S}}^{\xi}(\bm{V}) - \nabla f_{\bm{S}}(\bm{V})\|_{F}^{2}
    = 4\|\Skew(\bm{W}_{\bm{S}}^{f^{\xi}}(\bm{V}) - \bm{W}_{\bm{S}}^{f}(\bm{V}))\|_{F}^{2} \\
    & \leq 4\|\bm{W}_{\bm{S}}^{f^{\xi}}(\bm{V}) - \bm{W}_{\bm{S}}^{f}(\bm{V})\|_{F}^{2}
    \leq 4\|\overline{\bm{W}}_{\bm{S}}^{f^{\xi}}(\bm{V}) - \overline{\bm{W}}_{\bm{S}}^{f}(\bm{V})\|_{F}^{2} \\
    & = 4\|(\bm{I}+\bm{V})^{-1}\bm{I}_{N\times p}(\nabla f^{\xi}(\Phi_{\bm{S}}^{-1}(\bm{V})) - \nabla f(\Phi_{\bm{S}}^{-1}(\bm{V})))\bm{S}(\bm{I}+\bm{V})^{-1}\|_{F}^{2} \\
    & \leq 4\|(\bm{I}+\bm{V})^{-1}\|_{2}^{4}\|\bm{I}_{N\times p}\|_{2}^{2}\|\bm{S}\|_{2}^{2}\|\nabla f^{\xi}(\Phi_{\bm{S}}^{-1}(\bm{V})) - \nabla f(\Phi_{\bm{S}}^{-1}(\bm{V}))\|_{F}^{2} \ (\because \textrm{Lemma~\ref{lemma:norm_basic}~\ref{enum:norm_upper}}) \\
    & \leq 4\|\nabla f^{\xi}(\Phi_{\bm{S}}^{-1}(\bm{V})) - \nabla f(\Phi_{\bm{S}}^{-1}(\bm{V}))\|_{F}^{2}. \quad (\because \textrm{Lemma~\ref{lemma:norm_basic}~\ref{enum:IV_inv_norm}})
  \end{align}
  By taking the expectation of both sides, we get~\eqref{eq:bounded_variance_gradient}.

\end{proof}

\section{Proof of Proposition~\ref{proposition:mobility}}\label{appendix:mobility}
Application of~\eqref{eq:Cayley_inv_alt} to
$\bm{U}(\tau) := \Phi_{\bm{S}}^{-1}(\bm{V}+\tau \bm{\mathcal{E}}) = 2\bm{S}(\bm{I}+\bm{V}+\tau \bm{\mathcal{E}})^{-1}\bm{I}_{N\times p}-\bm{S}\bm{I}_{N\times p}$
yields
$\bm{U}(\tau) - \bm{U}(0) = 2\bm{S}\left((\bm{I}+\bm{V}+\tau \bm{\mathcal{E}})^{-1} - (\bm{I} + \bm{V})^{-1}\right)\bm{I}_{N\times p}
= -2\tau\bm{S} (\bm{I}+\bm{V}+\tau\bm{\mathcal{E}})^{-1}\bm{\mathcal{E}}(\bm{I}+\bm{V})^{-1}\bm{I}_{N\times p}$
(for the 2nd equality, see the proof of Lemma~\ref{lemma:norm_basic}~\ref{enum:norm_Lipschitz}) and
\begin{align}
  & \|\bm{U}(\tau) - \bm{U}(0)\|_{F}
  = \|2\tau \bm{S}(\bm{I}+\bm{V}+\tau\bm{\mathcal{E}})^{-1}\bm{\mathcal{E}}(\bm{I}+\bm{V})^{-1}\bm{I}_{N\times p}\|_{F} \\
  & \leq  2\tau \|\bm{S}\|_{2}\|(\bm{I}+\bm{V}+\tau\bm{\mathcal{E}})^{-1}\|_{2}\|\bm{\mathcal{E}}\|_{F}\|(\bm{I}+\bm{V})^{-1}\bm{I}_{N\times p}\|_{2}\  (\because \textrm{Lemma~\ref{lemma:norm_basic}~\ref{enum:norm_upper}})\\
  & \leq 2\tau \|(\bm{I}+\bm{V})^{-1}\bm{I}_{N\times p}\|_{2}
  \leq 2\tau \begin{Vmatrix} \bm{I}_{p} & -\dbra{\bm{V}}_{21}^{\T} \end{Vmatrix}_{2}\|\bm{M}^{-1}\|_{2}, \label{eq:proof_mobility}
\end{align}
where
$\bm{M}=\bm{I}_{p}+\dbra{\bm{V}}_{11}+\dbra{\bm{V}}_{21}^{\T}\dbra{\bm{V}}_{21} \in \mathbb{R}^{p\times p}$,
the second last inequality is derived by Lemma~\ref{lemma:norm_basic}~\ref{enum:IV_inv_norm}, and the last inequality is derived by~\eqref{eq:IV_inv}.

To evaluate the norm
$\|\bm{M}^{-1}\|_{2}$,
let
$\bm{I}_{p}+\dbra{\bm{V}}_{21}^{\T}\dbra{\bm{V}}_{21} = \bm{Q}(\bm{I}_{p}+\bm{\Sigma})\bm{Q}^{\T}$
be the eigenvalue decomposition,
where
$\bm{Q} \in \textrm{O}(p)$
is an orthogonal matrix and
$\bm{\Sigma}\in\mathbb{R}^{p\times p}$
is a diagonal matrix whose entries are non-negative.
Then, we have
\begin{equation}
  \bm{M} = \bm{Q}(\bm{I}_{p}+\bm{\Sigma})^{1/2}\left(\bm{I}_{p}+(\bm{I}_{p}+\bm{\Sigma})^{-1/2}\bm{Q}^{\T}\dbra{\bm{V}}_{11}\bm{Q}(\bm{I}_{p}+\bm{\Sigma})^{-1/2}\right)(\bm{I}_{p}+\bm{\Sigma})^{1/2}\bm{Q}^{\T}.
\end{equation}
The norm
$\|\bm{M}^{-1}\|_{2}=\|(\bm{I}_{p}+\dbra{\bm{V}}_{11}+\dbra{\bm{V}}_{21}^{\mathrm{T}}\dbra{\bm{V}}_{21})^{-1}\|_{2}$
is bounded above as
{
  \thickmuskip=0.0\thickmuskip
  \medmuskip=0.0\medmuskip
  \thinmuskip=0.0\thinmuskip
  \begin{align}
    & \|\bm{M}^{-1}\|_{2}
    = \left\|\bm{Q}(\bm{I}_{p}+\bm{\Sigma})^{-1/2}\left(\bm{I}_{p}+(\bm{I}_{p}+\bm{\Sigma})^{-1/2}\bm{Q}^{\T}\dbra{\bm{V}}_{11}\bm{Q}(\bm{I}_{p}+\bm{\Sigma})^{-1/2}\right)^{-1}(\bm{I}_{p}+\bm{\Sigma})^{-1/2}\bm{Q}^{\T}\right\|_{2} \\
    & \leq \|(\bm{I}_{p}+\bm{\Sigma})^{-1/2}\|_{2}^{2}
    \left\|\left(\bm{I}_{p}+(\bm{I}_{p}+\bm{\Sigma})^{-1/2}\bm{Q}^{\T}\dbra{\bm{V}}_{11}\bm{Q}(\bm{I}_{p}+\bm{\Sigma})^{-1/2}\right)^{-1}\right\|_{2}
    \leq \|(\bm{I}_{p}+\bm{\Sigma})^{-1}\|_{2}, \label{eq:mobility_proof_A}
  \end{align}
}%
\noindent
where the last inequality is derived from the skew-symmetry of
$(\bm{I}_{p}+\bm{\Sigma})^{-1/2}\bm{Q}^{\T}\dbra{\bm{V}}_{11}\bm{Q}(\bm{I}_{p}+\bm{\Sigma})^{-1/2}$
and Lemma~\ref{lemma:norm_basic}~\ref{enum:IV_inv_norm}.
Moreover, by
$\|(\bm{I}_{p}+\bm{\Sigma})^{-1}\|_{2}= (1+\sigma^{2}_{\min}(\dbra{\bm{V}}_{21}))^{-1}$,
we have
$\|\bm{M}^{-1}\|_{2} \leq (1+\sigma^{2}_{\min}(\dbra{\bm{V}}_{21}))^{-1}$.
Furthermore, from the definition of the spectral norm, we have
$\begin{Vmatrix} \bm{I}_{p} & -\dbra{\bm{V}}_{21}^{\T} \end{Vmatrix}_{2} = \sqrt{\lambda_{\max}(\bm{I}_{p}+\dbra{\bm{V}}_{21}^{\T}\dbra{\bm{V}}_{21})} = \sqrt{1+\|\dbra{\bm{V}}_{21}\|_{2}^{2}}$.
By substituting these relations into~\eqref{eq:proof_mobility}, we completed the proof of~\eqref{eq:mobility_rate}.
The equation~\eqref{eq:bound_ratio} is verified by
$\sigma_{\min}(\dbra{\bm{V}}_{21}) \leq \sigma_{\max}(\dbra{\bm{V}}_{21}) = \|\dbra{\bm{V}}_{21}\|_{2}$.

\section[Gradient of the composite of the Cayley retraction and the cost function]{Gradient of $f\circ R_{\bm{U}}^{\rm Cay}$} \label{appendix:gradient_retraction}
\begin{proposition} \label{proposition:gradient_retraction}
  Let
  $\bm{U} \in \St(p,N)$
  and
  $\bm{U}_{\perp} \in \St(N-p,N)$
  satisfy
  $\bm{U}^{\T}\bm{U}_{\perp} = \bm{0}$.
  For a differentiable function
  $f:\mathbb{R}^{N\times p}\to \mathbb{R}$,
  the Cayley transform-based retraction
  $R^{\rm Cay}$
  in~\eqref{eq:Cayley_retraction}, and
  $\bm{U} \in \St(p,N)$,
  the function
  $f\circ R_{\bm{U}}^{\rm Cay}:T_{\bm{U}}\St(p,N) \to \mathbb{R}$
  is differentiable with
  \begin{equation}
    (\bm{\mathcal{V}} \in T_{\bm{U}}\St(p,N)) \quad \nabla (f\circ R_{\bm{U}}^{\rm Cay})(\bm{\mathcal{V}})
    = -2\bm{P}_{\bm{U}}\Skew(\bm{Z}\bm{U}\nabla f(R_{\bm{U}}^{\rm Cay}(\bm{\mathcal{V}}))^{\T}\bm{Z})\bm{U}
  \end{equation}
  where
  $\bm{P}_{\bm{U}}:=\bm{I} - \bm{U}\bm{U}^{\T}/2 \in \mathbb{R}^{N\times N}$
  and
  $\bm{Z} := (\bm{I}+\Skew(\bm{U}\bm{\mathcal{V}}^{\T}\bm{P}_{\bm{U}}))^{-1} \in \mathbb{R}^{N\times N}$.
  The matrix
  $\bm{Z}$
  can be expressed as
  $\bm{Z} = \bm{I} - \bm{A}(\bm{I}_{2p} + \bm{B}^{\T}\bm{A})^{-1}\bm{B}^{\T}$
  with
  $\bm{A} = \begin{bmatrix} \bm{U} & \bm{P}_{\bm{U}}\bm{\mathcal{V}}/2 \end{bmatrix} \in \mathbb{R}^{N\times 2p}$
  and
  $\bm{B} = \begin{bmatrix} \bm{P}_{\bm{U}}\bm{\mathcal{V}}/2 & -\bm{U} \end{bmatrix} \in \mathbb{R}^{N\times 2p}$.
\end{proposition}
\begin{proof}
  Let
  $\bm{\mathcal{V}},\bm{\mathcal{D}} \in T_{\bm{U}}\St(p,N)$.
  From the chain rule and Fact~\ref{fact:matrix_differential}, we obtain
  \begin{align}
    & \left.\frac{d(f\circ R_{\bm{U}}^{\rm Cay})}{dt}(\bm{\mathcal{V}} + t\bm{\mathcal{D}})\right|_{t=0} = \trace\left(\nabla f(R_{\bm{U}}^{\rm Cay}(\bm{\mathcal{V}}))^{\T}\left.\frac{dR_{\bm{U}}^{\rm Cay}}{dt}(\bm{\mathcal{V}}+t\bm{\mathcal{D}})\right|_{t=0}\right) \\
    & = -2\trace\left(\nabla f(R_{\bm{U}}^{\rm Cay}(\bm{\mathcal{V}}))^{\T} \bm{Z}\Skew(\bm{U}\bm{\mathcal{D}}^{\T}\bm{P}_{\bm{U}})\bm{Z}\bm{U}\right) \\
    & = \trace\left(\bm{U}^{\T}\bm{Z}\bm{U}\nabla f(R_{\bm{U}}^{\rm Cay}(\bm{\mathcal{V}}))^{\T}\bm{Z}\bm{P}_{\bm{U}}\bm{\mathcal{D}}\right)
    - \trace\left(\bm{\mathcal{D}}^{\T}\bm{P}_{\bm{U}}\bm{Z}\bm{U}\nabla f(R_{\bm{U}}^{\rm Cay}(\bm{\mathcal{V}}))^{\T}\bm{Z}\bm{U}\right) \\
    & = \trace\left(\bm{U}^{\T}\bm{Z}\bm{U}\nabla f(R_{\bm{U}}^{\rm Cay}(\bm{\mathcal{V}}))^{\T}\bm{Z}\bm{P}_{\bm{U}}\bm{\mathcal{D}}\right)
    - \trace\left(\bm{U}^{\T}\bm{Z}^{\T}\nabla f(R_{\bm{U}}^{\rm Cay}(\bm{\mathcal{V}}))\bm{U}^{\T}\bm{Z}^{\T}\bm{P}_{\bm{U}}\bm{\mathcal{D}}\right) \\
    & = \trace\left(2\bm{U}^{\T}\Skew(\bm{Z}\bm{U}\nabla f(R_{\bm{U}}^{\rm Cay}(\bm{\mathcal{V}}))^{\T}\bm{Z})\bm{P}_{\bm{U}}\bm{\mathcal{D}}\right)
    = \trace\left(\left(-2\bm{P}_{\bm{U}}\Skew(\bm{Z}\bm{U}\nabla f(R_{\bm{U}}^{\rm Cay}(\bm{\mathcal{V}}))^{\T}\bm{Z})\bm{U}\right)^{\T}\bm{\mathcal{D}}\right) \label{eq:retraction_derivative_chain}
  \end{align}
  due to
  $R_{\bm{U}}^{\rm Cay}(\bm{\mathcal{V}})  = 2(\bm{I}+\Skew(\bm{U}\bm{\mathcal{V}}^{\T}\bm{P}_{\bm{U}}))^{-1}\bm{U} - \bm{U}$
  (see~\eqref{eq:Cayley_retraction} and~\eqref{eq:inv_origin_Cayley}) and
  \begin{align}
    \left.\frac{dR_{\bm{U}}^{\rm Cay}}{dt}(\bm{\mathcal{V}}+t\bm{\mathcal{D}})\right|_{t=0}
      & = -2(\bm{I}+\Skew(\bm{U}\bm{\mathcal{V}}^{\T}\bm{P}_{\bm{U}}))^{-1}\Skew(\bm{U}\bm{\mathcal{D}}^{\T}\bm{P}_{\bm{U}})(\bm{I}+\Skew(\bm{U}\bm{\mathcal{V}}^{\T}\bm{P}_{\bm{U}}))^{-1}\bm{U} \\
      & = -2\bm{Z}\Skew(\bm{U}\bm{\mathcal{D}}^{\T}\bm{P}_{\bm{U}})\bm{Z}\bm{U}.
  \end{align}
  For
  $\bm{\mathcal{W}}:= -2\bm{P}_{\bm{U}}\Skew(\bm{Z}\bm{U}\nabla f(R_{\bm{U}}^{\rm Cay}(\bm{\mathcal{V}}))^{\T}\bm{Z})\bm{U} \in \mathbb{R}^{N\times p}$,
  we have
  $\bm{U}^{\T}\bm{\mathcal{W}} + \bm{\mathcal{W}}^{\T}\bm{U} = \bm{0}$
  because
  $\bm{U}^{\T}\bm{\mathcal{W}} = -\bm{U}^{\T}\Skew(\bm{Z}\bm{U}\nabla f(R_{\bm{U}}^{\rm Cay}(\bm{\mathcal{V}}))^{\T}\bm{Z})\bm{U}$
  is skew-symmetric.
  Fact~\ref{fact:stiefel}~\ref{enum:tangent} yields
  $\bm{\mathcal{W}} \in T_{\bm{U}}\St(p,N)$.

  On the other hand, we obtain
  \begin{equation} \label{eq:retraction_derivative_direction}
    \left.\frac{d(f\circ R_{\bm{U}}^{\rm Cay})}{dt}(\bm{\mathcal{V}} + t\bm{\mathcal{D}})\right|_{t=0}
      = \trace(\nabla (f\circ R_{\bm{U}}^{\rm Cay})(\bm{\mathcal{V}})^{\T}\bm{\mathcal{D}}).
  \end{equation}
  From~\eqref{eq:retraction_derivative_chain},~\eqref{eq:retraction_derivative_direction} and
  $\bm{\mathcal{W}} \in T_{\bm{U}}\St(p,N)$,
  it holds
  $\nabla (f\circ R_{\bm{U}}^{\rm Cay})(\bm{\mathcal{V}}) = \bm{\mathcal{W}}$.

  In the following, let us consider the expression of
  $\bm{Z}$
  along the discussion in~\cite[Lemma 4]{W13}.
  From
  $\bm{I} + \Skew(\bm{U}\bm{\mathcal{V}}^{\T}\bm{P}_{\bm{U}}) = \bm{I} + \bm{A}\bm{B}^{\T}$,
  we have
  $\bm{Z} = (\bm{I}+\bm{A}\bm{B}^{\T})^{-1}$.
  Then, applying the Sherman-Morrison-Woodbury formula (see Fact~\ref{fact:SMW}) to
  $\bm{Z}$,
  we obtain
  $\bm{Z} = \bm{I} - \bm{A}(\bm{I}_{2p} + \bm{B}^{\T}\bm{A})^{-1}\bm{B}^{\T}$.
\end{proof}

\end{document}

%% file: bibpunct.tex
\bibpunct[, ]{[}{]}{,}{n}{,}{,}